\documentclass[opre]{GDDArxiv}

\OneAndAHalfSpacedXI 

\usepackage{endnotes}
\usepackage{bm}
\let\footnote=\endnote

\usepackage{enumitem}
\usepackage{cases}
\usepackage{amsmath}
\usepackage{algorithm}
\usepackage{algpseudocode}
\usepackage{caption}
\usepackage{float}
\usepackage{comment}
\usepackage{subcaption} 
\usepackage{tikz}
\usepackage{booktabs,siunitx,multirow, makecell}
\usepackage{graphicx}
\usepackage{xcolor}
\usepackage{amsmath}
\usepackage{amssymb}
\usepackage{mathtools}
\usepackage{empheq}
\usepackage{bbm}
\usepackage{bm}
\usepackage{appendix}
\usepackage[authoryear]{natbib}
\usepackage[normalem]{ulem}
 \bibpunct[, ]{(}{)}{,}{a}{}{,}%
 \def\bibfont{\small}%

\usepackage{soul}

\TheoremsNumberedThrough     
\ECRepeatTheorems

\EquationsNumberedThrough    

\usepackage[hidelinks]{hyperref}

\begin{document}


\RUNAUTHOR{Zhang and Jiang}

\RUNTITLE{Generalized Dual Decomposition}

\TITLE{Generalized Dual Decomposition}

\ARTICLEAUTHORS{%

\AUTHOR{Pengyu Zhang, Ruiwei Jiang}
\AFF{Industrial \& Operations Engineering,
University of Michigan, Ann Arbor, \EMAIL{\{pengyuz, ruiwei\}@umich.edu}}
} 

\ABSTRACT{%
We study two-stage stochastic optimization models with mixed-integer decision variables appearing in both stages. For these models, dual decomposition enables parallel computing implementation and can quickly provide a lower bound for the optimal value. However, the lower bound thus obtained is not exact in general due to the lack of strong duality. In this paper, we propose a generalized dual decomposition (GDD) that extends the linear regularizer used in dual decomposition to a general nonlinear one, which still admits parallelization while exhibiting strong duality. By encoding the nonlinear regularizers through parameterization and cutting planes, we establish the convergence of a GDD algorithm to achieve global optimum. 
In addition, we discuss strategies for solving the GDD scenario subproblems more efficiently, including pruning and valid inequalities. Furthermore, we extend GDD to a more general, constrained form that subsumes, as special cases, robust optimization, chance-constrained programs, and bilevel optimization with multiple followers. 
Finally, numerical experiments demonstrate the strong duality of GDD, 
its computational efficacy versus primal (Benders-type) decomposition algorithms, and the speedup through parallel computing.
}%




\maketitle
%
%
%
\section{Introduction}
We consider a two-stage stochastic optimization model~\citep{kuccukyavuz2017introduction,shapiro2021lectures,ntaimo2024computational}
\begin{align}
v^* \ := \ \min_{\bm{x} \in X} \ \ \sum_{i=1}^N f_i(\bm{x}), \tag{SP} \label{problem:primal}
\end{align}
where \(\bm{x}\) denotes first-stage (or here-and-now) decision variables living in a feasible region \(X \subseteq \mathbb{Z}^{p_1} \times \mathbb{R}^{n_1-p_1}\) and, for each scenario \(i \in [N]:=\{1, \ldots, N\}\), \(f_i(\bm{x})\) denotes the optimal value of a second-stage (or recourse) mixed-integer program that may be linear or nonlinear. As a popular example, the recourse function \(f_i(\cdot)\) is defined through a mixed-integer linear program (MILP)
\begin{align*}
f_i(\bm{x}) \ := \ (\bm{c}^i)^{\top}\bm{x} + \min_{\bm{y} \in Y} \ & \ (\bm{q}^i)^{\top} \bm{y} \\
\text{s.t.} \ & \ \bm{T}^i\bm{x} + \bm{W}^i\bm{y} = \bm{h}^i
\end{align*}
and \(Y \subseteq \mathbb{Z}^{p_2}\times\mathbb{R}^{n_2-p_2}\) for all \(i \in [N]\). Note here that the objective coefficient \(\bm{q}^i\) subsumes the probability of scenario \(i\) occurring. 
The~\eqref{problem:primal} model finds a wide range of applications in, e.g., energy systems~\citep{munoz2015scalable}, healthcare~\citep{shehadeh2021using}, telecommunication~\citep{sen1994network}, and transportation~\citep{paul2019supply}. 

The main difficulty in solving~\eqref{problem:primal} is the potentially massive aggregation of the recourse functions $f_i(\cdot)$, which are in general nonlinear and nonconvex. Although the resulting model can be written as a deterministic mixed-integer program, its size quickly becomes computationally intractable and, as demonstrated in the subsequent numerical experiments (see Section~\ref{sec:numerical}), the ensuing branch-and-bound tree takes up prohibitively enormous memory. This motivates the use of decomposition methods, which exploit the fact that the recourse variables \(\bm{y}\) dwell separately in each scenario and only the here-and-now variables \(\bm{x}\) link all scenarios together.

The first paradigm of decomposition methods extend the Benders or L-shaped method~\citep{van1969shaped} for solving two-stage stochastic linear programs to the mixed-integer case~\citep[see, e.g.,][]{sen2005algorithms,kuccukyavuz2017introduction}. These methods derive cutting planes for the epigraph of the recourse function \(f_i(\cdot)\) (or its expectation, \(\sum_{i=1}^N f_i(\cdot)\)) and achieve global optimum of~\eqref{problem:primal} by iteratively incorporating these cutting planes to better approximate the epigraph. Examples of cutting planes include integer optimality cuts~\citep{laporte1993integer,angulo2016improving}, parametric Gomory cuts~\citep{gade2014decomposition,zhang2014finitely}, disjunctive cuts~\citep[see, e.g.,][]{sen2005c,chen2011finite,qi2017ancestral}, Lagrangian cuts~\citep[see, e.g.,][]{zou2019stochastic,rahmaniani2020benders,chen2022generating}, scaled cuts~\citep{van2024converging}, reverse norm cuts~\citep{ahmed2022stochastic}, and ReLU cuts~\citep{deng2024relu}.


A second paradigm adopts dual decomposition (DD), which was first proposed by~\cite{rockafellar1976nonanticipativity} to the best of our knowledge. DD makes copies \(\{\bm{x}^i: i \in [N]\}\) of the here-and-now variable \(\bm{x}\) for all scenarios and rewrite~\eqref{problem:primal} through nonanticipativity constraints, e.g., \(\bm{x}^i = \bm{x}\) for all \(i \in [N]\). Then, relaxing these constraints and then penalizing their violations yield a Lagrangian relaxation of~\eqref{problem:primal}. DD is appealing because (i) for fixed Lagrangian multipliers, the relaxation becomes \emph{separable} across all scenario subproblems, allowing us to solve them in parallel; and (ii) the lower bound obtained from solving this relaxation is oftentimes promising in a variety of applications~\citep{deng2018parallel}. Furthermore,~\cite{rockafellar1991scenarios} proposed a progressive hedging (PH) algorithm, which augments each scenario subproblem in DD with a penalty term to push its solution towards an average across scenarios~\citep[see][for a different augmented Lagrangian relaxation]{mulvey1995new}. Later,~\cite{watson2011progressive} and~\cite{gade2016obtaining} derived lower bounds in each iteration of PH and improved its convergence.

Unfortunately, when~\eqref{problem:primal} involves mixed-integer decision variables in both stages, existing approaches in the DD paradigm do not admit strong duality and the ensuing lower bound is not exact in general. To achieve global optimum,~\cite{caroe1999dual} and~\cite{lubin2013parallelizing} incorporated DD in a branch-and-bound algorithm, and~\cite{ahmed2013scenario} and~\cite{deng2018parallel} assumed that \(\bm{x}\) is purely binary and progressively shrank the feasible region \(X\) through ``no-good cuts''~\citep[see, e.g.,][]{angulo2016improving}. Notably, in the same setting of purely binary \(\bm{x}\),~\cite{cifuentes2024lagrangian} added monomial nonanticipativity constraints, e.g., \(x^i_1x^i_2 = x_1x_2\), \(x^i_1x^i_2x^i_3 = x_1x_2x_3\) and so on, into the DD formulation and successfully established \emph{strong duality} when all possible monomial constraints are incorporated.

This paper follows the works in the DD paradigm and proposes a generalized dual decomposition (GDD) approach to solving~\eqref{problem:primal} with mixed-integer decision variables in both stages. GDD allows us to solve scenario subproblems in parallel and meanwhile admits strong duality of~\eqref{problem:primal}. Although our experiments focus on linear objective functions and constraints (except for the integrality restrictions), GDD can be readily deployed to the nonlinear case, for example,~\eqref{problem:primal} with quadratic objective functions. Our main contributions are as follows.

\begin{enumerate}
\item We propose a GDD formulation for~\eqref{problem:primal} based on nonlinear regularizers. We prove that this formulation admits strong duality.
\item We propose an algorithm to iteratively approximate the nonlinear regularizers through parameterization and cutting planes. We derive sufficient conditions for these approximations so that the algorithm converges to global optimum of the GDD formulation.
\item We study strategies to speed up the solution of each scenario subproblem in the GDD algorithm. In particular, we focus on a class of piecewise constant regularizers to propose a pruning strategy and to characterize their convex envelope.
\item We extend GDD beyond~\eqref{problem:primal} to robust optimization, chance constraints, and bilevel optimization with multiple followers.
\item We conduct extensive numerical experiments to demonstrate the effectiveness of GDD, including its strong duality, its comparison with alternative solution strategies such as the deterministic equivalent formulation and the primal decomposition methods, and the speedup through parallel computing.
\end{enumerate}

In the remainder of this paper, we introduce the GDD formulation and prove its strong duality in Section~\ref{sec:model}, describe the algorithm and discuss algorithmic improvement strategies in Section~\ref{sec:GDDMIPAlgDesign}, present extensions in Section~\ref{sec:extensions}, report numerical results in Section~\ref{sec:numerical}, before concluding the paper in Section~\ref{sec:conclude}. All proofs are provided in the Appendix.



\paragraph{Notation.}
Throughout this paper, lowercase letters (e.g., \(a\)) denote scalars, bold lowercase letters (e.g., \(\bm{a}\)) denote vectors, and bold uppercase letters (e.g., \(\bm{A}\)) denote matrices. For any positive integer \(N\), we let \([N] := \{1,2,\ldots,N\}\). For a binary vector \(\bm{x}\), we use \(I(\bm{x}) := \{j : x_j = 1\}\) to denote the index set of its nonzero components. For a set \(A\), we use \(\mathbbm{1}(\bm{x} \in A)\) to denote its indicator function, \(\operatorname{conv}(A)\) to denote its convex hull, 
\(\operatorname{cl}(A)\) its closure, and \(\operatorname{Proj}_{\bm{x}}(A)\) its projection onto the \(\bm{x}\)-space. We let \(\bm{e}\) and \(\bm{e}_i\) denote an all-one vector and the \(i\)th unit vector, respectively.

\section{A Generalized Dual Decomposition Formulation} \label{sec:model}
The primary computational challenge of \ref{problem:primal} stems from the potentially large number of scenarios, $N$. The challenge of finding a solution balancing objectives across all scenarios is revealed by the nonanticipativity constraints in the following equivalent formulation of~\eqref{problem:primal}:
\begin{align}
    \min_{\bm{x}, \bm{x}^i \in X, i \in [N]}\ &\  \sum_{i=1}^N f_i(\bm{x}^i) \nonumber\\
    \text{s.t.}\ &\ \bm{x}^i = \bm{x} \quad \forall i \in [N]. \tag{SP-NA}
    \label{problem:primalNA}
\end{align}
By applying Lagrangian relaxation to the nonanticipativity constraints, the standard DD yields a lower bound with a naturally decomposable structure:
\begin{align}
v_{\text{DD}} \ := \ \max_{\bm{\lambda}_i} \ & \ \sum_{i=1}^N \ \min_{\bm{x}^i \in X} \Big\{f_i(\bm{x}^i) + \bm{\lambda}^\top_i \bm{x}^i \Big\} \nonumber \\
\text{s.t.} \ & \ \sum_{i=1}^N \bm{\lambda}_i = 0 \tag{DD} \label{problem:dual_decomposition} \\
& \ \bm{\lambda}_i \in \mathbb{R}^{n_1} \quad \forall i \in [N], \nonumber
\end{align}
where \(\bm{\lambda}_i \in \mathbb{R}^{n_1}\) represents the Lagrangian multiplier for constraint \(\bm{x}^i = \bm{x}\). The term \(\bm{\lambda}_i^\top \bm{x}^i\) within each scenario subproblem objective functions serves as a linear regularizer to drive consensus among all scenario copies \(\bm{x}^i\) of the here-and-now variables. In fact, the positive duality gap between~\eqref{problem:primal} and~\eqref{problem:dual_decomposition} can be partly attributed to the linear regularizer, which has limited capability of coordinating the scenario subproblems.

\begin{example}[\ref{problem:dual_decomposition} Lacking Strong Duality] \label{ex:DD-strong-duality}
Consider an instance of~\eqref{problem:primal} with \(N = 2\),
\begin{align*}
    \min_{\bm{x} \in X} \ \Big\{f_1(\bm{x}) + f_2(\bm{x})\Big\},
\end{align*}
where \(X = \{0, 1\}\times [0, 1]\), i.e., \(x_1 \in \{0,1\}\) and \(x_2 \in [0, 1]\). In addition, $f_1, f_2$ are defined through 
\begin{align*}
f_1(\bm{x}) \ := & \ \min_{\bm{y} \in \{0, 1\}\times \mathbb{R}_+}\Big\{2y_1 + 4y_2: \ y_1 + y_2 \geq x_1+x_2-1\Big\}, \\
f_2(\bm{x}) \ := & \ \min_{\bm{y} \in \{0,1\}^2 \times \mathbb{R}^2_+}\Big\{\delta y_1 + 2y_2 + 2\delta y_3 + 4y_4 + (\delta-2)x_1: \ y_1 + y_3 \geq x_2 - x_1, \ y_2 + y_4 \geq x_1 - x_2\Big\}
\end{align*}
for a \(\delta > 0\). Equivalently,
$$
f_1(\bm{x}) = \begin{cases}
    0 & \text{if} \ x_1 = 0\\
    \min\{4 x_2, \ 2\} & \text{if} \ x_1 = 1
\end{cases}
\quad \text{and} \quad
f_2(\bm{x}) = \begin{cases}
    \min\{2\delta x_2, \ \delta\} & \text{if} \ x_1 = 0\\
    \min\{\delta + 2 - 4x_2, \ \delta\} & \text{if} \ x_1 = 1.
\end{cases}
$$
We depict \(f_1, f_2\) in Fig.~\ref{fig:3dF1F2}. This instance has an optimal value of zero and a (unique) optimal solution of $\bm{x} = [0, 0]^{\top}$.

\begin{figure}
    \centering
    \includegraphics[width=\linewidth]{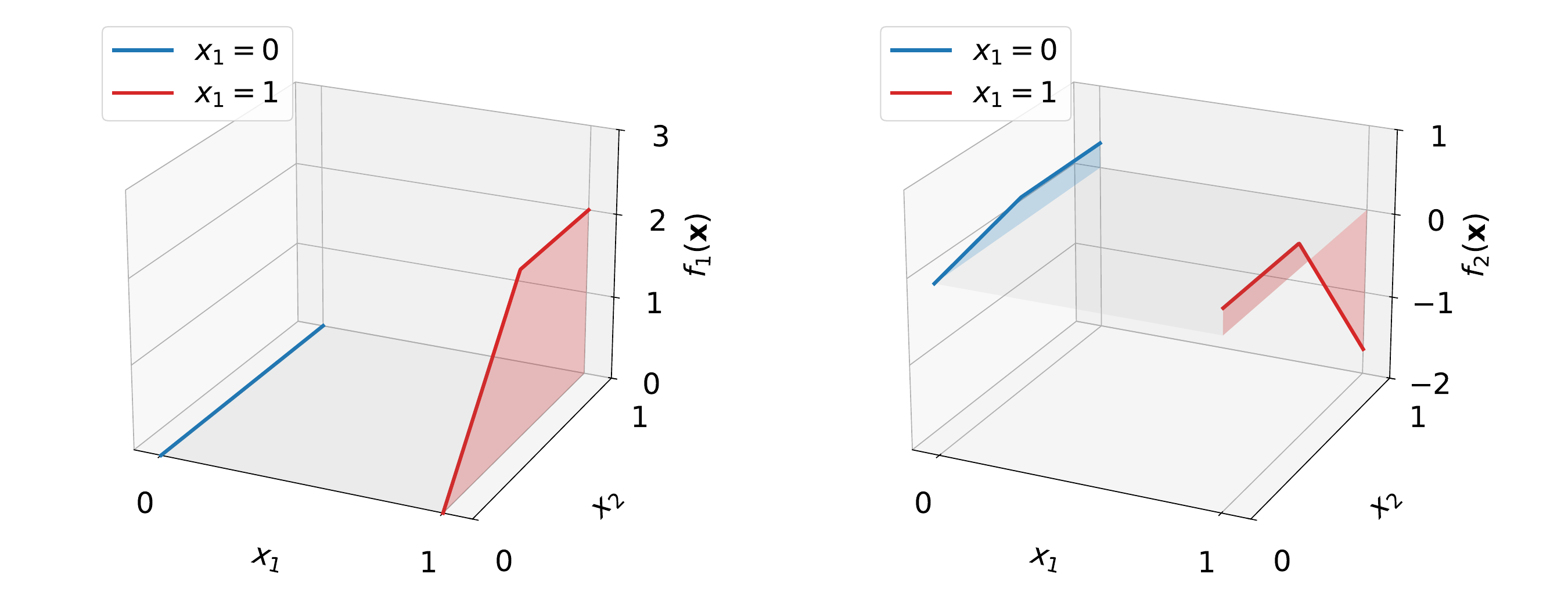}
    \caption{Visualization of $f_1$ and $f_2$ in Example~\ref{ex:DD-strong-duality} with \(\delta = 0.3\)}
    \label{fig:3dF1F2}
\end{figure}
Let $\bm{\lambda} := (\lambda_1, \lambda_2),- \bm{\lambda} \in \mathbb{R}^2$ denote the Lagrangian multipliers for $f_1$ and $f_2$ in~\eqref{problem:dual_decomposition}, respectively. Then, the~\eqref{problem:dual_decomposition} formulation becomes
\[
v_{\text{DD}} \ = \ \max_{\lambda_1, \lambda_2 \in \mathbb{R}} \Big\{ \min_{\bm{x} \in \{0, 1\}\times [0, 1]}\big\{f_1(\bm{x}) + \bm{\lambda}^{\top}\bm{x}\big\} + \min_{\bm{x} \in \{0, 1\}\times [0, 1]}\big\{f_2(\bm{x}) - \bm{\lambda}^{\top}\bm{x}\big\} \Big\}.
\]
Suppose that~\eqref{problem:dual_decomposition} admits strong duality in this instance and achieves the optimal value of zero. Then, for some \(\bm{\lambda}\), we have \(\min_{\bm{x}}\{f_1(\bm{x}) + \bm{\lambda}^\top \bm{x}\} + \min_{\bm{x}}\{f_2(\bm{x}) - \bm{\lambda}^\top \bm{x}\} \geq 0\). But \(f_1(\bm{0}) + \bm{\lambda}^\top \bm{0} = f_2(\bm{0}) - \bm{\lambda}^\top \bm{0} = 0\), that is, the solution \(\bm{x} = \bm{0}\) is feasible and attains a zero objective value in both scenario subproblems. Thus, \(\bm{x}=\bm{0}\) is an optimal solution to both subproblems. In particular, comparing \(\bm{x} = \bm{0}\) with \(\bm{x} = \bm{e}_1, \bm{e}_2\), that \(\bm{0} \in \text{arg}\min_{\bm{x}}\{f_1(\bm{x}) + \bm{\lambda}^\top \bm{x}\}\) implies 
\begin{subequations}
\begin{equation}
0 \ \leq \ f_1(\bm{e}_1) + \bm{\lambda}^\top \bm{e}_1 \ = \ \lambda_1 \quad \text{and} \quad 0 \ \leq \ f_1(\bm{e}_2) + \bm{\lambda}^\top \bm{e}_2 \ = \ \lambda_2. \label{dd-weak-note-1}
\end{equation}
Likewise, comparing \(\bm{x} = \bm{0}\) with \(\bm{x} = \bm{e}\), that \(\bm{0} \in \text{arg}\min_{\bm{x}}\{f_2(\bm{x}) - \bm{\lambda}^\top \bm{x}\}\) implies 
\begin{equation}
0 \ \leq \ f_2(\bm{e}) - \bm{\lambda}^\top \bm{e} \ = \ \delta - 2 - \lambda_1 - \lambda_2. \label{dd-weak-note-2}
\end{equation}
\end{subequations}
In other words, through~\eqref{dd-weak-note-1}--\eqref{dd-weak-note-2}, strong duality of~\eqref{problem:dual_decomposition} implies that \(0 \leq \lambda_1 + \lambda_2 \leq \delta - 2\), or equivalently \(\delta \geq 2\). Conversely,~\eqref{problem:dual_decomposition} fails to uphold strong duality whenever $\delta < 2$. \hfill\(\Box\)
\end{example}

To recover strong duality, we propose to adopt general nonlinear regularizers \(g_i(\bm{x}): X \rightarrow \mathbb{R}\), which generalize~\eqref{problem:dual_decomposition} to the following~\eqref{problem:generalized_dual_decomposition} formulation.
\begin{align}
v_{\text{GDD}} \ := \ \max_{g_i(\cdot)} \ & \ \sum_{i=1}^N \ \min_{\bm{x}^i \in X} \Big\{f_i(\bm{x}^i) + g_i(\bm{x}^i)\Big\} \nonumber \\
\text{s.t.} \ & \ \sum_{i=1}^N g_i(\bm{x}) = 0 \quad \forall \bm{x} \in X \tag{GDD} \label{problem:generalized_dual_decomposition} \\
& \ g_i: X \rightarrow \mathbb{R} \qquad \forall i \in [N]. \nonumber
\end{align}

We first establish the weak and strong duality of~\eqref{problem:generalized_dual_decomposition}.
\begin{theorem}[Weak and Strong Duality] \label{prop:strong_duality}
For any $\{g_i(\cdot)\}_{i \in [N]}$ such that $\sum_{i=1}^N g_i(\bm{x}) = 0$ for all $\bm{x} \in X$, it holds that \(\sum_{i=1}^N \min_{\bm{x}^i \in X} \big\{f_i(\bm{x}^i) + g_i(\bm{x}^i)\big\} \leq v^*\). Furthermore, it holds that \emph{\(v_{\text{GDD}} = v^*\)}, i.e., the optimal values of~\eqref{problem:primal} and~\eqref{problem:generalized_dual_decomposition} coincide.
\end{theorem}

Notably, we impose no structural assumptions on the objective functions $f_i$. Consequently, the strong duality of the GDD framework holds for general two-stage stochastic programs, e.g., \(\bm{x}\) can be mixed-integer and \(f_i\) can be quadratic. 
\begin{example}[Strong Duality of~\ref{problem:generalized_dual_decomposition}] \label{ex:GDD-strong}
Continuing on Example~\ref{ex:DD-strong-duality}, we define piecewise constant regularizers $g_1, g_2$ through
$$
g_1(\bm{x}) = \begin{cases}
    0 & \text{if} \ x_1 = 0\\
    0 & \text{if} \ x \in \{1\} \times [0, \frac{1}{2}]\\
    \delta - 2 & \text{if} \ x \in \{1\} \times (\frac{1}{2}, 1]
\end{cases}
\quad \text{and} \quad
g_2(\bm{x}) = \begin{cases}
    0 & \text{if} \ x_1 = 0\\
    0 & \text{if} \ x \in \{1\} \times [0, \frac{1}{2}]\\
    2 - \delta & \text{if} \ x \in \{1\} \times (\frac{1}{2}, 1].
\end{cases}
$$
We depict \(g_1, g_2\) in Fig.~\ref{fig:3dG1G2}. It is clear that $g_1(\bm{x}) + g_2(\bm{x}) = 0$ for all $\bm{x} \in X$, 
\begin{figure}
    \centering
    \includegraphics[width=\linewidth]{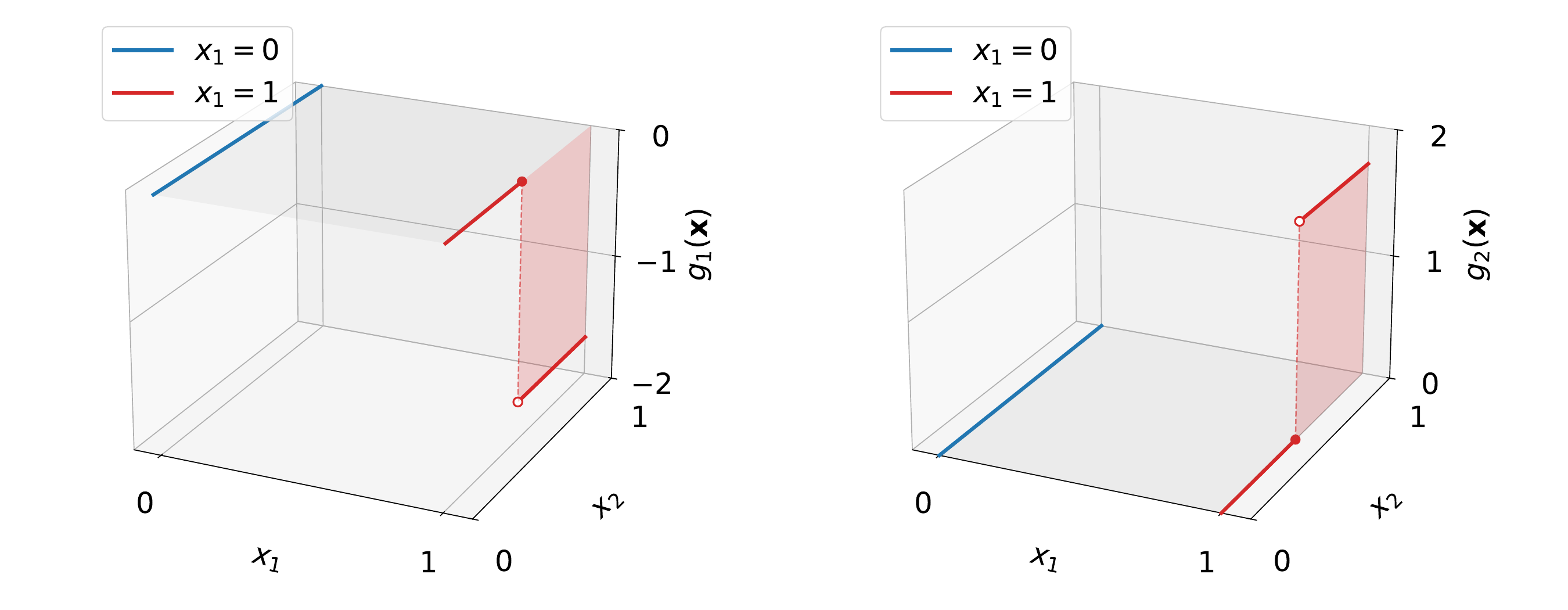}
    \caption{Visualization of $g_1$ and $g_2$ in Example~\ref{ex:GDD-strong} with \(\delta = 0.3\)}
    \label{fig:3dG1G2}
\end{figure}
$$
f_1(\bm{x}) + g_1(\bm{x}) = \begin{cases}
    0 & \text{if} \ x_1 = 0\\
    4x_2 & \text{if} \ x \in \{1\} \times [0, \frac{1}{2}]\\
    \delta & \text{if} \ x \in \{1\} \times (\frac{1}{2}, 1]
\end{cases}
\ \text{and} \
f_2(\bm{x}) + g_2(\bm{x}) = \begin{cases}
    \min\{2\delta x_2, \ \delta\} & \text{if} \ x_1 = 0\\
    \delta & \text{if} \ x \in \{1\} \times [0, \frac{1}{2}]\\
    4-4x_2 & \text{if} \ x \in \{1\} \times (\frac{1}{2}, 1].
\end{cases}
$$
It follows that both scenario subproblems produce an optimal value of zero and thus~\eqref{problem:generalized_dual_decomposition} admits strong duality with $v_{\text{GDD}} = \min_{\bm{x}} \{f_1(\bm{x}) + g_1(\bm{x})\} + \min_{\bm{x}} \{f_2(\bm{x}) + g_2(\bm{x})\} = 0$. \hfill \(\Box\)
\end{example}
Furthermore, the strong duality of~\eqref{problem:generalized_dual_decomposition} can also be interpreted through convexifying the summation of the \(f_i\)'s. This is in contrast with~\eqref{problem:dual_decomposition}, which convexifies each \(f_i\) \emph{before} taking the summation. We remark that this dual representation for~\eqref{problem:dual_decomposition} has been shown in various contexts in the literature~\citep[see, e.g., Proposition 2 of][]{caroe1999dual}.
\begin{proposition}[Dual Representation of~\ref{problem:generalized_dual_decomposition}] \label{prop:GDDDual}
It holds that
\begin{align*}
v_{\operatorname{GDD}} \ = \ \min_{x \in X} \ \overline{\operatorname{co}}\left(\sum_{i=1}^{N} f_i\right)\left(\bm{x}\right),
\end{align*}
where \(\overline{\operatorname{co}}(f)\) refers to the convex envelope of function \(f(\cdot)\). In contrast, it holds that
\[v_{\operatorname{DD}} \ = \ \min_{x \in X} \ \sum_{i=1}^{N} \overline{\operatorname{co}}(f_i)(\bm{x}).\]
\end{proposition}
Before closing this section, we show that any optimal solution to the primal problem~\eqref{problem:primal} is also an optimal solution to each scenario subproblem in~\eqref{problem:generalized_dual_decomposition} under a group of optimal regularizers. Thus, if any subproblem admits a unique optimal solution, this property allows us to recover an optimal solution to~\eqref{problem:primal} from solving~\eqref{problem:generalized_dual_decomposition}.
\begin{proposition}[Recovering Optimal Solution from~\ref{problem:generalized_dual_decomposition}]
    \label{prop:recoverOptimal}
    Let \(\{g_i^*(\cdot): i \in [N]\}\) denote a group of optimal regularizers to~\eqref{problem:generalized_dual_decomposition}. Then, each optimal solution $\bm{x}^*$ to~\eqref{problem:primal} is also optimal for the scenario subproblem $\min_{\bm{x}^i \in X} \{f_i(\bm{x}^i) + g^*_i(\bm{x}^i)\}$ for all \(i \in [N]\).
\end{proposition}

\section{Generalized Dual Decomposition Algorithms}\label{sec:GDDMIPAlgDesign}

The~\eqref{problem:generalized_dual_decomposition} formulation is infinite-dimensional because of the functional regularizers \(g_i(\cdot)\) and potentially intractable. However, the regularizers actually admit an optimal solution in closed-form.
\begin{proposition} \label{prop:optimal-regularizer}
Let \(g^*_i(\bm{x}) := \sum_{i=1}^N f_i(\bm{x})/N - f_i(\bm{x})\) for all \(i \in [N]\). Then, it holds that
\[
v^* \ = \ v_{\operatorname{GDD}} \ = \ \sum_{i=1}^N \ \min_{\bm{x}^i \in X} \Big\{f_i(\bm{x}^i) + g^*_i(\bm{x}^i)\Big\}.
\]
\end{proposition}
Nonetheless, applying the optimal regularizers in Proposition~\ref{prop:optimal-regularizer} yields an objective function \(f_i(\bm{x}^i) + g^*_i(\bm{x}^i) = \sum_{i=1}^N f_i(\bm{x})/N\) for each \(i \in [N]\), transforming all scenario subproblems back to the original~\eqref{problem:primal} formulation and jeopardizing the hope of solving them efficiently in parallel. Fortunately, an approximation of \(g^*_i(\cdot)\) can already produce near-optimal solutions (see Proposition~\ref{prop:simpleApprox}) or even certifies an optimal solution (see Example~\ref{ex:GDD-strong}). This section proposes algorithms that iteratively improve approximations for \(g^*_i(\cdot)\) to achieve global optimum. Section~\ref{subsec:gdd-algorithm} provides a baseline algorithm and derives sufficient conditions on these approximate regularizers to achieve global optimum. In addition, we provide concrete examples of the approximate regularizers that satisfy these sufficient conditions, including piecewise constant regularizers, reverse norm cuts of~\cite{ahmed2022stochastic}, and ReLU cuts of~\cite{deng2024relu}. Section~\ref{subsec:improvement} provides algorithmic improvement strategies based on the piecewise constant regularizers.

\begin{algorithm}
\caption{An Iterative Algorithm for~\eqref{problem:generalized_dual_decomposition}}
\label{alg:MIPGDD}
\begin{algorithmic}[1]
\State \textbf{Initialize}: $g_i(\cdot, \cdot) = 0,\; \mathit{lb}=-\infty,\; \mathit{ub}=\infty,\; t=0,\; \mathcal{X}^0 = \emptyset$.
\While{(stopping criteria are not fulfilled)}
    \State Solve $\min_{\bm{x} \in X} \big\{f_i(\bm{x}) + g_i\big(\mathcal{X}^t, \bm{x}\big)\big\}$ and collect optimal solutions $\bm{x}^{i*}$ for all $i \in [N]$. \label{alg:solveSubProblem}
    \State $\mathit{ub} \leftarrow \min\Big\{\mathit{ub},\; \min_{i \in [N]} \big\{\sum_{j=1}^N f_{j}(\bm{x}^{i*})\big\} \Big\}$.
    \State $\mathit{lb} \leftarrow \max\left\{\mathit{lb},\; \sum_{i=1}^N \Big(f_i(\bm{x}^{i*}) + g_i\big(\mathcal{X}^t, \bm{x}^{i*}\big)\Big)\right\}$.
    \State Update \(g_i(\mathcal{X}^t, \bm{x})\) with \(\big\{\bm{x}^{i*}, i \in [N]\big\}\). \label{algo:gdd-update}
    \State \(\mathcal{X}^t \leftarrow \mathcal{X}^t \cup \big\{\bm{x}^{i*}, i \in [N]\big\}\), \(t \leftarrow t + 1\).
\EndWhile
\State \Return $\mathit{lb}, \mathit{ub}$
\end{algorithmic}
\end{algorithm}

\subsection{An Iterative Algorithm} \label{subsec:gdd-algorithm}
We iteratively update a group of parametric regularizers \(\{g_i(\mathcal{X}^t, \bm{x}): i \in [N]\}\), where \(\mathcal{X}^t\) collects the solutions to scenario subproblems by iteration \(t\). We present a baseline framework in Algorithm~\ref{alg:MIPGDD}. 
To enable its convergence to global optimum of~\eqref{problem:generalized_dual_decomposition}, we give sufficient conditions on the update of the regularizers (i.e., step~\ref{algo:gdd-update} of Algorithm~\ref{alg:MIPGDD}) in the following theorem.
\begin{theorem}[Sufficient Conditions for Global Optimum]\label{prop:sufficientConverge}
For all \(i \in [N]\), denote by \(\{\bm{x}^i_t: t \geq 0\}\) the sequence of optimal solutions to scenario subproblem \(i\) produced by Algorithm~\ref{alg:MIPGDD}. Suppose that the regularizers \(\{g_i(\mathcal{X}^t, \bm{x}): i \in [N]\}\) satisfy the following two conditions:
\begin{enumerate}
    \item in each iteration \(t\), it holds that $\sum_{i \in [N]} g_i(\mathcal{X}^t, \bm{x}) \leq 0$ for all $\bm{x} \in X$;
    \item for all \(i \in [N]\), any subsequence of \(\{\bm{x}^i_t: t \geq 0\}\) has a converging subsequence, denoted by \(\{\bm{x}^i_{(r)}: r \geq 0\}\), such that $\liminf_{r\rightarrow \infty} \big(f_i(\bm{x}^i_{(r)}) + g_i^*(\bm{x}^{i}_{(r)})\big) \leq \liminf_{r \rightarrow \infty} \big(f_i(\bm{x}^i_{(r)}) + g_i(\mathcal{X}^{t(r)}, \bm{x}^i_{(r)})\big)$, where \(\bm{x}^i_{(r)}\) is an optimal solution to scenario subproblem \(i\) in iteration \(t(r)\).   
\end{enumerate}
Then, the sequence of lower bounds given by Algorithm~\ref{alg:MIPGDD} converges to \(v^*\).
\end{theorem}
In the remainder of this section, we provide examples that satisfy these conditions in Sections~\ref{subsec:piecewise}--\ref{subsec:binary-tender}.


\subsubsection{Piecewise constant regularizers.} \label{subsec:piecewise}

First, we consider piecewise constant regularizers. Specifically, we specify the set of solutions \(\mathcal{X}^t\) by iteration \(t\) of Algorithm~\ref{alg:MIPGDD} to be \(\mathcal{X}^t := \{\bm{x}_{(k)}: k \in [K_t]\}\) and denote by \(K_t := |\mathcal{X}^t|\) the cardinality of \(\mathcal{X}^t\). Then, we consider regularizers
\begin{equation}
g_i\left(\mathcal{X}^t, \bm{x}\right) \ := \ \sum_{k=1}^{K_t} g_i^*(\bm{x}_{(k)}) \mathbbm{1}\big(\bm{x} \in \mathcal{P}_k^t\big) \label{eq:simpleGFunction}
\end{equation}
with \(\{\mathcal{P}_k^t: k \in [K_t]\}\) forming a disjoint partition of \(X\) in each iteration \(t\), i.e., \(X = \bigcup_{k=1}^{K_t} \mathcal{P}_k^t\) and \(\mathcal{P}_k^t \cap \mathcal{P}_j^t = \varnothing\) whenever \(k \neq j\). In addition, we notice that \(g_i^*(\bm{x}_{(k)})\) is a constant and can be computed by evaluating all \(\{f_i(\bm{x}_{(k)}): i \in [N]\}\) in parallel.
\begin{example} \label{ex:piecewise-regularizer}
Continuing on Examples~\ref{ex:DD-strong-duality}--\ref{ex:GDD-strong}, we follow Proposition~\ref{prop:optimal-regularizer} to define optimal regularizers
$$
g^*_1(\bm{x}) = \begin{cases}
    \delta x_2 & \text{if} \ x \in \{0\} \times [0, \frac{1}{2})\\
    \frac{1}{2}\delta & \text{if} \ x \in \{0\} \times [\frac{1}{2}, 1]\\
    \frac{1}{2}\delta - 2x_2 & \text{if} \ x_1 =1
\end{cases}
\quad \text{and} \quad
g^*_2(\bm{x}) = \begin{cases}
    -\delta x_2 & \text{if} \ x \in \{0\} \times [0, \frac{1}{2})\\
    -\frac{1}{2}\delta & \text{if} \ x \in \{0\} \times [\frac{1}{2}, 1]\\
    -\frac{1}{2}\delta + 2x_2 & \text{if} \ x_1 =1.
\end{cases}
$$
Then, the piecewise constant regularizers \(g_1, g_2\) in Example~\ref{ex:GDD-strong} can be specified by \(K_t = 3\) and
\begin{align*}
g_1(\bm{x}) \ = \ \sum_{k=1}^3 g_1^*(\bm{x}_{(k)}) \mathbbm{1}\big(\bm{x} \in \mathcal{P}_k\big), \quad & g_2(\bm{x}) \ = \ \sum_{k=1}^3 g_2^*(\bm{x}_{(k)}) \mathbbm{1}\big(\bm{x} \in \mathcal{P}_k\big) \\
\operatorname{with} \qquad \qquad \mathcal{P}_1 \ = \ \{0\}\times [0,1], \quad & \bm{x}_{(1)} \ = \ [0,0]^{\top}; \\[0.5em]
\mathcal{P}_2 \ = \ \{1\}\times \Big[0,\frac{1}{2}\Big), \quad & \bm{x}_{(2)} \ = \ \Big[1,\frac{1}{4}\delta\Big]^{\top}; \\[0.5em]
\mathcal{P}_3 \ = \ \{1\}\times \Big[\frac{1}{2},1\Big], \quad & \bm{x}_{(3)} \ = \ \Big[1,1-\frac{1}{4}\delta\Big]^{\top}.
\end{align*}
We depict \(g_1^*\) and \(g_1\) in Fig.~\ref{fig:g1VsG1Star}. From this figure, we observe that \(g_1\) (respectively, \(g_2\)) coincides with \(g_1^*\) (respectively, \(g_2^*\)) only at \(\bm{x}_{(1)}, \bm{x}_{(2)}, \bm{x}_{(3)}\), but already achieves global optimum. \hfill \(\Box\)

\begin{figure}[htbp]
    \centering
    \begin{subfigure}{0.48\textwidth}
        \centering
        \includegraphics[width=\linewidth]{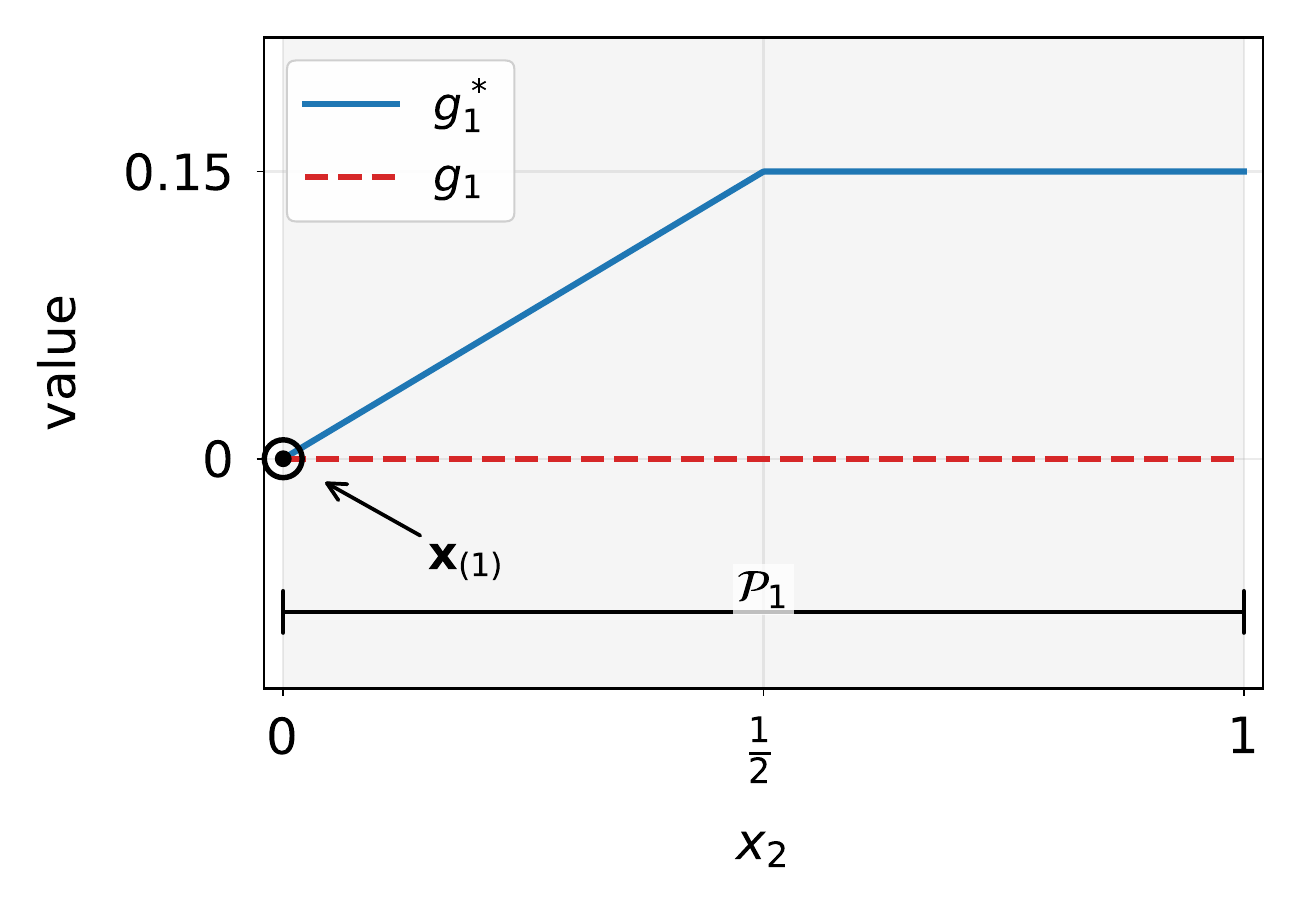}
        \caption{$x_1 = 0$}
        \label{fig:x1-0}
    \end{subfigure}
    \hfill
    \begin{subfigure}{0.48\textwidth}
        \centering
        \includegraphics[width=\linewidth]{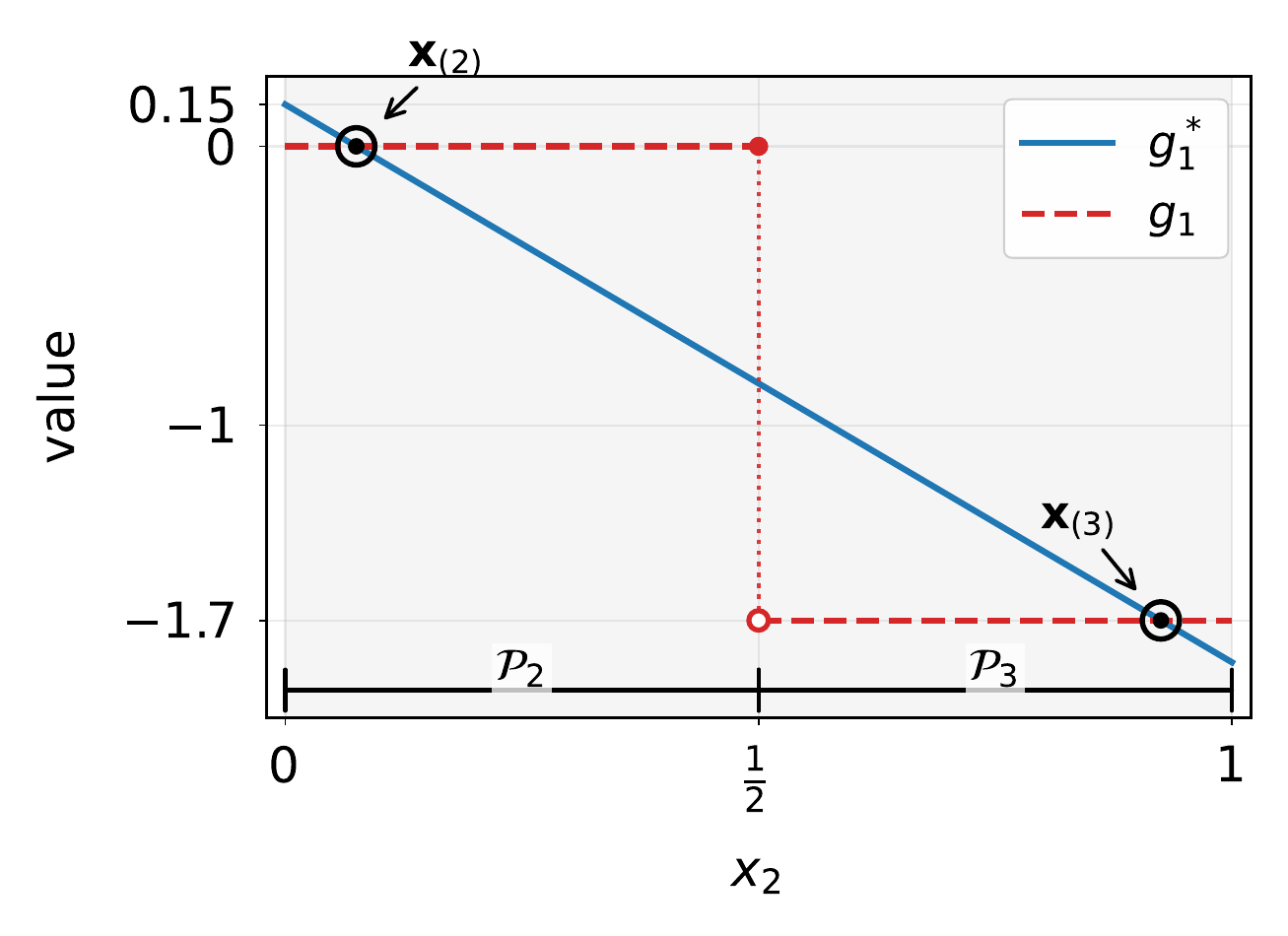}
        \caption{$x_1 = 1$}
        \label{fig:x1-1}
    \end{subfigure}
    \caption{$g_1^*$ versus $g_1$ with $\delta = 0.3$.}
    \label{fig:g1VsG1Star}
\end{figure}

\end{example}
In this paper, we construct \(\mathcal{P}_k^t\) through the Voronoi partition of \(X\). That is, \(\bm{x} \in \mathcal{P}_k^t\) only if \(\bm{x}\) is (weakly) closer to \(\bm{x}_{(k)}\) than to any other element of \(\mathcal{X}^t\). Specifically, for each \(k \in [K_t]\), \(\operatorname{cl}(\mathcal{P}_k^t)\) admits a polyhedral representation:
\begin{align*}
\operatorname{cl}(\mathcal{P}_k^t) \ := & \ \left\{\bm{x} \in X: \ \|\bm{x} - \hat{\bm{x}}_{(k)}\|_2 \leq \|\bm{x} - \hat{\bm{x}}_{(j)}\|_2, \ \forall j \in [K_t]\right\} \\[0.5em]
= & \ \left\{\bm{x} \in X: \ 2(\hat{\bm{x}}_{(j)} - \hat{\bm{x}}_{(k)})^\top \bm{x} \leq \hat{\bm{x}}_{(j)}^\top \hat{\bm{x}}_{(j)} - \hat{\bm{x}}_{(k)}^\top \hat{\bm{x}}_{(k)}, \ \forall j \in [K_t] \right\}.
\end{align*}
We remark that the membership of the boundary points of each \(\mathcal{P}^t_k\) may be arbitrary, which implies that \(g_i(\mathcal{X}^t, \bm{x})\) may not be lower-semicontinuous. For instance, \(g_1\) in Example~\ref{ex:piecewise-regularizer} is upper-semicontinuous. Nevertheless, replacing \(g_i(\mathcal{X}^t, \bm{x})\) with its lower-semicontinuous hull, denoted by \(\operatorname{lsc}(g_i)(\mathcal{X}^t, \bm{x})\), preserves the optimal value of each scenario subproblem for continuous \(f_i\). In addition, \(\operatorname{lsc}(g_i)(\mathcal{X}^t, \bm{x})\) admits a MILP representation.
\begin{proposition} \label{prop:gi-milp}
Suppose that \(\operatorname{cl}(\mathcal{P}_k^t) := \{\bm{x} \in X: \bm{A}_k\bm{x} \leq \bm{b}_k\}\) for all \(k \in [K_t]\). Then, for each \(i \in [N]\), \(\operatorname{lsc}(g_i)(\mathcal{X}^t, \bm{x})\) admits the following mixed-integer linear representation:
\begin{align*}
\operatorname{lsc}(g_i)(\mathcal{X}^t, \bm{x}) \ = \ \min_{\bm{w}, \bm{x}} \ & \ \sum_{k=1}^{K_t} g_i^*(\bm{x}_{(k)})w_k \\
\operatorname{s.t.} \ & \ \sum_{k=1}^{K_t} w_k = 1 \\
& \ \bm{A}_k\bm{x} \leq \bm{b}_k + M_k(1-w_k)\bm{e} \quad \forall k \in [K_t] \\
& \ w_k \in \{0, 1\} \quad \forall k \in [K_t],
\end{align*}
where $\displaystyle M_k := \max_{i \in \operatorname{dim}(\bm{A}_k)}\max_{\bm{x} \in X} \big\{\bm{e}_i^{\top}(\bm{A}_k\bm{x} - \bm{b}_k)\big\}$. In addition, if \(f_i(\cdot)\) is continuous then it holds that
\[
\min_{\bm{x} \in X} \Big\{f_i(\bm{x}) + g_i\big(\mathcal{X}^t, \bm{x}\big)\Big\} \ = \ \min_{\bm{x} \in X} \Big\{f_i(\bm{x}) + \operatorname{lsc}(g_i)\big(\mathcal{X}^t, \bm{x}\big)\Big\}.
\]
\end{proposition}
Following Proposition~\ref{prop:gi-milp}, we can solve the scenario subproblems in step~\ref{alg:solveSubProblem} of Algorithm~\ref{alg:MIPGDD} as MILPs, if \(f_i(\bm{x})\) arises from a two-stage stochastic MILP with complete continuous recourse (see Assumption~\ref{assump:CCR}). We shall discuss how to solve these subproblems more efficiently in Section~\ref{subsec:improvement} and demonstrate these improvement strategies numerically in Section~\ref{sec:numerical}. Before closing this section, we prove theoretical guarantees of the piecewise constant regularizers under technical assumptions.
\begin{assumption}[Complete Continuous Recourse; CCR] \label{assump:CCR}
For each \(i \in [N]\), the recourse function \(f_i(\cdot)\) is defined through
\begin{align*}
f_i(\bm{x}) \ := \ \min_{\bm{y}, \bm{z}} \ & \ \ell(\bm{y}, \bm{z}) \\
\operatorname{s.t.} \ & \ L(\bm{x}, \bm{y}, \bm{z}) \leq \bm{0} \\
& \ \bm{y} \in Y \cap \mathbb{Z}^{p_2}, \bm{z} \in Z \cap \mathbb{R}^{n_2-p_2}
\end{align*}
and satisfies either one of the following two conditions:
\begin{enumerate}
\item \(\ell:\mathbb{R}^{n_2}\rightarrow \mathbb{R}\) and \(L:\mathbb{R}^{n_1+n_2}\rightarrow \mathbb{R}^{m_2}\) are affine, \(Y\) and \(Z\) are polyhedral, and for all \(\bm{x} \in X\) and \(\bm{y} \in Y \cap \mathbb{Z}^{p_2}\) there exists a \(\bm{z} \in Z \cap \mathbb{R}^{n_2-p_2}\) such that \(L(\bm{x}, \bm{y}, \bm{z}) \leq 0\);
\item \(\ell:\mathbb{R}^{n_2}\rightarrow \mathbb{R}\) and \(L:\mathbb{R}^{n_1+n_2}\rightarrow \mathbb{R}^{m_2}\) are Lipschitz continuous, $L$ is convex in $\bm{z}$ (but not necessarily in \(\bm{x}\) or \(\bm{y}\)), \(Y\) is compact, \(Z\) is compact and convex, and there exists a $\delta > 0$ such that for all \(\bm{x} \in X\) and \(\bm{y} \in Y \cap \mathbb{Z}^{p_2}\) there exists a \(\bm{z} \in Z \cap \mathbb{R}^{n_2-p_2}\) such that \(L(\bm{x}, \bm{y}, \bm{z}) < -\delta\).
\end{enumerate}
\end{assumption}
\begin{assumption}[Compactness] \label{assumption:compact}
The feasible region $X$ of~\eqref{problem:primal} is compact.
\end{assumption}

The affine case of Assumption~\ref{assump:CCR} is standard in the literature~\citep[see, e.g.,][]{ahmed2022stochastic} and establishes the Lipschitz continuity of $f_i(\bm{x})$ (see Lemma~\ref{lemma:LipschitzCont} in Appendix~\ref{sec:LipschitzCont}). We extend CCR to the general nonlinear case through Robinson's convex extension of the Hoffman's Lemma~\citep{robinson1975application} to simplify the subsequent analyses. In practical applications of~\eqref{problem:primal}, we can always fulfill this assumption by incorporating slack variables in second-stage constraints with a sufficiently large penalty. In addition, by Theorem~4 of~\cite{feizollahi2017exact}, there exists a finite penalty coefficient such that the optimal value of~\eqref{problem:primal} remains the same after incorporating slack variables to fulfill its CCR. We are now ready to show theoretical guarantees for the piecewise constant regularizers.
\begin{proposition}[Near-Optimality] \label{prop:simpleApprox}
Under Assumptions~\ref{assump:CCR}--\ref{assumption:compact}, for any \(\epsilon > 0\), there exist piecewise constant regularizers $\{g_i(\bm{x}): \ i\in [N]\}$ defined on a finite partition $\left\{\mathcal{P}_k\right\}_{k=1}^K$ of $X$ such that $$
\sum_{i=1}^N \min_{\bm{x} \in X} \left\{f_i(\bm{x}) + g_i(\bm{x})\right\} \ \geq \ v^* - \epsilon.
$$
\end{proposition}
Proposition~\ref{prop:simpleApprox} reinforces our earlier observation through Example~\ref{ex:piecewise-regularizer} that piecewise constant regularizers can achieve global optimum in~\eqref{problem:generalized_dual_decomposition}. The next theorem enacts this observation through Algorithm~\ref{alg:MIPGDD}.
\begin{theorem} \label{prop:MIPConvergence}
Under Assumptions~\ref{assump:CCR}--\ref{assumption:compact}, the sequence of lower bounds produced by Algorithm~\ref{alg:MIPGDD} with respect to the piecewise constant regularizers defined through~\eqref{eq:simpleGFunction} converges to \(v^*\).
\end{theorem}

\subsubsection{Valid, tight, and Lipschitz continuous regularizers.} \label{subsec:reverse-relu}
Second, we consider approximating the optimal regularizers $g^*_i(\bm{x})$ through cutting planes. Specifically, we generate a valid inequality, denoted by \(\mathcal{L}_i(\bm{x}_{(k)}, \bm{x})\), for the epigraph of \(g^*_i(\bm{x})\) at each \(\bm{x}_{(k)} \in \mathcal{X}^t\). In other words, \(g^*_i(\bm{x}) \geq \mathcal{L}_i(\bm{x}_{(k)}, \bm{x})\) for all \(\bm{x} \in X\). Then, in step~\ref{algo:gdd-update} of Algorithm~\ref{alg:MIPGDD}, we approximate \(g^*_i(\bm{x})\) by
\begin{align}
g_i(\mathcal{X}^t, \bm{x}) \ := \ \max_{k \in [K_t]} \mathcal{L}_i(\bm{x}_{(k)}, \bm{x}) \label{eq:lipschitz_cut_approx}
\end{align}
for all \(i \in [N]\). The appeal of the approximation \eqref{eq:lipschitz_cut_approx} is that it preserves a global lower estimate of $g_i^*(\bm{x})$ while improving the approximation locally at the optimal solutions to scenario subproblems. Examples of \(\mathcal{L}_i(\bm{x}_{(k)}, \bm{x})\) include the reverse norm cuts of~\cite{ahmed2022stochastic}, the augmented Lagrangian cuts of~\cite{ahmed2022stochastic}, and the ReLU Lagrangian cuts of~\cite{deng2024relu}.

We next show that the convergence of Algorithm~\ref{alg:MIPGDD} to global optimum relies on three key properties of the cuts \(\mathcal{L}_i(\bm{x}_{(k)}, \bm{x})\): they must be \emph{valid} for \(g^*_i(\bm{x})\), \emph{tight} at \(\bm{x}_{(k)}\), and \emph{uniformly Lipschitz continuous} in \(\bm{x}\) so that the ensuing local improvement at \(\bm{x}_{(k)}\) can be propagated to its neighborhood.

\begin{definition}[Valid Cut]
A cut $\mathcal{L}_i(\bm{x}_{(k)}, \bm{x})$ is valid if $g_i^*(\bm{x}) \geq \mathcal{L}_i(\bm{x}_{(k)}, \bm{x})$ for all $\bm{x} \in X$.
\end{definition}

\begin{definition}[Tight Cut]
A cut $\mathcal{L}_i(\bm{x}_{(k)}, \bm{x})$ is tight at \(\bm{x}_{(k)}\) if $g_i^*(\bm{x}_{(k)}) = \mathcal{L}_i(\bm{x}_{(k)}, \bm{x}_{(k)})$.
\end{definition}

\begin{definition}[Uniformly Lipschitz Cut Family]
A family of cuts $\{\mathcal{L}_i(\hat{\bm{x}},\cdot) : \hat{\bm{x}} \in X\}$ is uniformly Lipschitz continuous if there exists a constant $L_i \ge 0$ such that
\[
    \left|\mathcal{L}_i(\hat{\bm{x}},\bm{x}) - \mathcal{L}_i(\hat{\bm{x}},\bm{y})\right|
    \ \le \ L_i \|\bm{x}-\bm{y}\|
    \qquad \forall \bm{x},\bm{y},\hat{\bm{x}} \in X.
\]
\end{definition}
In particular, the reverse norm cuts and the augmented Lagrangian cuts are valid, tight, and uniformly Lipschitz continuous. The ReLU cuts are valid and tight, and are uniformly Lipschitz continuous if all \(f_i(\cdot)\) are Lipschitz continuous. The next theorem shows that these properties are sufficient for the convergence of Algorithm~\ref{alg:MIPGDD} to global optimum.
\begin{theorem} \label{prop:LipschitzContCut}
Under Assumptions \ref{assump:CCR}-\ref{assumption:compact}, if the family of cuts $\mathcal{L}_i(\bm{x}_{(k)}, \bm{x})$ is valid, tight, and uniformly Lipschitz continuous, then the sequence of lower bounds produced by Algorithm~\ref{alg:MIPGDD} with respect to the cuts~\eqref{eq:lipschitz_cut_approx} converges to \(v^*\).
\end{theorem}

\subsubsection{Special case: binary tender.} \label{subsec:binary-tender}

In a special case of~\eqref{problem:primal}, all entries of \(\bm{x}\) take binary values. This simplifies~\eqref{problem:generalized_dual_decomposition} and Algorithm~\ref{alg:MIPGDD}, because each regularizer \(g_i(\cdot): \{0,1\}^{n_1} \rightarrow \mathbb{R}\) can now be encoded as a finite-dimensional vector \(\big[g^j_i: j \in [2^{n_1}]\big]\), where \(g^j_i=g_i\big(x^{(j)}\big)\) and \(\big\{x^{(j)}: j \in [2^{n_1}]\big\}\) denotes a permutation of the set \(\{0, 1\}^{n_1}\). In other words, the regularizer can be recast as \(g_i(\bm{x}) = \sum_{j=1}^{2^{n_1}} g^j_i \mathbbm{1}(\bm{x} = \bm{x}^{(j)})\). Accordingly, the~\eqref{problem:generalized_dual_decomposition} formulation now simplifies to
\begin{align}
v_{\text{GDD}} \ := \ \max_{\bm{g}} \ & \ \sum_{i=1}^N \ \min_{\bm{x}^i \in X} \bigg\{f_i(\bm{x}^i) + \sum_{j=1}^{2^{n_1}} g^j_i \mathbbm{1}(\bm{x}^i = \bm{x}^{(j)})\bigg\} \nonumber \\
\text{s.t.} \ & \ \sum_{i=1}^N g^j_i = 0 \quad \forall j \in [2^{n_1}]. \tag{GDD-B} \label{GDD-binary}
\end{align}
We notice that~\eqref{GDD-binary} recovers the V-Lagrangian dual formulation proposed by~\cite{cifuentes2024lagrangian} after adding all possible monomial redundant constraints (e.g., \(x^i_1x^i_2 = x_1x_2\), \(x^i_1x^i_2x^i_3 = x_1x_2x_3\), etc.) to the primal~\eqref{problem:primal} formulation (see their Section 2.4). In addition, Algorithm~\ref{alg:MIPGDD} simplifies to the subgradient algorithm of~\cite{cifuentes2024lagrangian} in this special case, which updates vectors \(\big[g^j_i: j \in [2^{n_1}]\big]\) instead of the regularizing functions \(g_i(\cdot)\). As a natural result, the convergence of Algorithm~\ref{alg:MIPGDD} to global optimum follows from that of the subgradient algorithm.

A key difference between this work and~\cite{cifuentes2024lagrangian} is that we directly generalize the dual formulation~\eqref{problem:generalized_dual_decomposition} and its regularizers, as opposed to adding redundant constraints to~\eqref{problem:primal} and then taking its Lagrangian dual. This allows us to address the more general, \emph{mixed-integer} tender variables \(\bm{x}\) without the need of finding the appropriate primal redundant constraints.

\subsection{Algorithmic Improvement Strategies} \label{subsec:improvement}
The~\ref{problem:generalized_dual_decomposition} Algorithm~\ref{alg:MIPGDD} allows us to solve all scenario subproblems in parallel. Furthermore, in the case of adopting piecewise constant regularizers, Proposition~\ref{prop:gi-milp} provides a MILP representation for each scenario subproblem. Nonetheless, the MILP formulation therein involves big-M coefficients, which often results in a weak continuous relaxation and slow computation. This section focuses on the piecewise constant regularizers and proposes several improvement strategies to circumvent this challenge and speed up solving the scenario subproblems. We demonstrate the effectiveness of these strategies numerically in Section~\ref{sec:numerical}.
\subsubsection{Polyhedral study.} The first strategy is to strengthen the MILP formulation in Proposition~\ref{prop:gi-milp}. To this end, we characterize the convex envelope of a general piecewise constant regularizer in the form of~\eqref{eq:simpleGFunction}.
\begin{proposition}\label{prop:polyStudy}
Consider a piecewise constant regularizer \(g(\bm{x}) := \sum_{k=1}^K g_k \mathbbm{1}(\bm{x} \in \mathcal{P}_k)\), where \(\{g_k: k \in [K]\}\) are fixed constants and \(\{\mathcal{P}_k: k \in [K]\}\) form a polyhedral partition of \(X\) with \(\operatorname{cl}(\mathcal{P}_k) = \{\bm{x} \in \mathbb{R}^{n_1}: \bm{A}_k \bm{x} \leq \bm{b}_k\}\), \(X = \bigcup_{k=1}^K \mathcal{P}_k\), and \(\mathcal{P}_j \cap \mathcal{P}_k = \varnothing\) whenever \(j \neq k\). Then, the convex envelope of \(g\) is given by
\begin{equation*}
\overline{\operatorname{co}}(g) \ = \ \Big\{(\bm{x}, \theta) \in X \times \mathbb{R}: \ \theta \geq \alpha + \bm{\beta}^{\top} \bm{x}, \ \forall (\alpha, \bm{\beta}) \in \operatorname{ext}(E)\Big\},
\end{equation*}
where the polyhedron \(E\) is defined through $E := \big\{(\alpha, \bm{\beta})\; \big| \; \exists \bm{\eta}_k \geq \bm{0}:\; \alpha + \bm{b}_k^\top \bm{\eta}_k \leq g_k, \; \bm{A}_k^\top \bm{\eta}_k = \bm{\beta},\; \forall k \in [K]\big\}$.
\end{proposition}
Proposition~\ref{prop:polyStudy} indicates that all facets of the target convex envelope pertain to extreme points of the polyhedron $E$. As a result, the convex envelope can be separated by optimizing a linear function over \(E\). Specifically, provided a point \((\hat{\bm{x}}, \hat{\theta})\), we can either certify that \((\hat{\bm{x}}, \hat{\theta}) \in \overline{\operatorname{co}}(g)\) if
\[
\hat{\theta} \ \geq \ \theta^* \ := \ \max_{(\alpha, \bm{\beta}) \in E} \big\{\alpha + \hat{\bm{x}}^{\top} \bm{\beta}\big\}
\]
or otherwise (i.e., if \(\hat{\theta} < \theta^*\)) separate \((\hat{\bm{x}}, \hat{\theta})\) from \(\overline{\operatorname{co}}(g)\) using the inequality \(\theta \geq \alpha^* + (\bm{\beta}^*)^{\top} \bm{x}\), where \((\alpha^*, \bm{\beta}^*)\) denotes an optimal solution to the above linear program.

\paragraph{Special case: binary tender.} When \(\bm{x}\) is binary-valued, it is even possible to derive a closed-form expression of a convex envelope closely related to \(\overline{\operatorname{co}}(g)\). When a standard subgradient algorithm is applied to solve~\eqref{GDD-binary}, the regularizers \(g_i(\bm{x}) \equiv \sum_{j=1}^{2^{n_1}} g^j_i \mathbbm{1}(\bm{x} = \bm{x}^{(j)})\) are sparse, that is, \(g^j_i \neq 0\) only for \(j\) belonging with a small subset \(\mathcal{J} \subseteq [2^{n_1}]\) and \(g^j_i = 0\) whenever \(j \notin \mathcal{J}\). This is because all \(g^j_i\) are initialized to be zero and each iteration of the subgradient algorithm updates at most \(N\) entries of \(\{g^j_i: j \in [2^{n_1}]\}\), pertaining to optimal solutions \(\{\bm{x}^{(j)}: j \in \mathcal{J}\}\) to the subproblems solved so far. As a result, the epigraph of each \(g_i(\bm{x})\) can be recast as
\[
\operatorname{epi}(g_i) \ = \ \bigcup_{j \in \mathcal{J}} \big\{\big(\theta, \bm{x}^{(j)}\big): \theta \geq g^j_i \big\} \ \cup \ \big\{(\theta, \bm{x}): \ \theta \geq 0, x \in A^c\big\},
\]
where set \(A : = \big\{\bm{x}^{(j)}, \forall j \in \mathcal{J}\big\} \subseteq \{0, 1\}^{n_1}\) collects these optimal solutions. Because each \(\big\{\big(\theta, \bm{x}^{(j)}\big): \theta \geq g^j_i \big\}\) is convex and trivial to optimize over, we focus on convexifying the set \(\big\{(\theta, \bm{x}): \ \theta \geq 0, x \in A^c\big\}\) in the next proposition.
\begin{proposition} \label{prop:convexHullProj}
For a set \(A \subseteq \{0, 1\}^{n_1}\), define its indicator function 
$I_A:\{0, 1\}^{n_1} \rightarrow \{0, 1\}$ through
$$I_A(\bm{x}) = \begin{cases}
        1 & \text{if } \bm{x} \in A\\
        0 & \text{otherwise}.
\end{cases}$$
Then, it holds that
\[
\operatorname{conv}\big\{(\theta, \bm{x}): \ \theta \geq 0, \ x \in A^c\big\} \ = \ \big\{(\theta, \bm{x}): \ \theta \geq 0, \ (0, \bm{x}) \in \overline{\operatorname{co}}(I_A) \big\}.
\]
\end{proposition}
It follows from Proposition~\ref{prop:convexHullProj} that convexifying \(\operatorname{epi}(g_i)\) boils down to computing the convex envelope of \(I_A\). While this task is challenging for a general \(A\), we address a simpler, but basic case.
\begin{definition}[Chain]
A set \(A \equiv \{\bm{x}^{(j)}: j \in \mathcal{J}\} \subseteq \{0, 1\}^{n_1}\) is called a \emph{chain} if there exists a permutation $\pi$ of $\mathcal{J}$, such that $I(\bm{x}^{(\pi_j)}) \subset I(\bm{x}^{(\pi_{j + 1})})$ for all \(j \in [|\mathcal{J}|]\), and the chain is \emph{continuous} if $|I(\bm{x}^{(\pi_{j + 1})})| = |I(\bm{x}^{(\pi_j)})| + 1$. \hfill \(\Box\)
\end{definition}
Chains are fundamental, because any subset \(A \subseteq \{0, 1\}^{n_1}\) can be written as a finite union of continuous chains, i.e., \(A = \cup_i A_i\) and each \(A_i\) is a continuous chain. Accordingly, we have \(I_A(\bm{x}) = \max_i\{I_{A_i}(\bm{x})\}\). The next proposition characterizes the convex envelope of chain indicator functions.
\begin{proposition}[Convex Envelope for Chains]\label{prop:convexEnvForChain}
Consider a chain \(A \equiv \{\bm{x}^{(j)}: j \in \mathcal{J}\}\) and, without loss of generality, assume that $I(\bm{x}^{(j)}) \subset I(\bm{x}^{(j + 1)})$ and each \(\bm{x}^{(j)}\) is non-increasing, i.e., \(x^{(j)}_i \geq x^{(j)}_{i+1}\) (all \(1\)-entries are in the front). Then, it holds that
\begin{equation*}
\overline{\operatorname{co}}(I_A) = 
\left\{ (\bm{x}, \theta) \in [0, 1]^{n_1} \times \mathbb{R}_+ \;\middle|\; 
\begin{aligned}
        \theta &\geq \sum_{i=1}^{|I(\bar{\bm{x}})|}x_{i} - \sum_{i=|I(\bar{\bm{x}})| + 1}^{n_1} x_{i} + 1 - |I(\bar{\bm{x}})| \\
        &\qquad \forall \bar{\bm{x}}\in A: \ \bar{\bm{x}} + \bm{e}_{|I(\bar{\bm{x}})| + 1} \notin A \\[0.5em]
        \theta &\geq \sum_{i=1}^{|I(\bar{\bm{x}})|}x_{i} - \sum_{i=|I(\bar{\bm{x}})| + 2}^{n_1} x_{i} + 1 - |I(\bar{\bm{x}})| \\
        &\qquad \forall \bar{\bm{x}} \in A: \ \bar{\bm{x}} + \bm{e}_{|I(\bar{\bm{x}})| + 1} \in A
\end{aligned}
\right\}.
\end{equation*}
\end{proposition}
We extend these results to a more general, tree structure.
\begin{definition}[2-Tree]
The union of two continuous chains \(A_1 \cup A_2\) forms a 2-tree if there exists an \(\bar{\bm{x}} \in A_1 \cap A_2\) such that, for any $\bm{x}_1 \in A_1, \bm{x}_2 \in A_2$, if $I(\bm{x}_1) \subset I(\bar{\bm{x}}), I(\bm{x}_2) \subset I(\bar{\bm{x}})$ then neither $I(\bm{x}_1) \subset I(\bm{x}_2)$ nor $I(\bm{x}_2) \subset I(\bm{x}_1)$. \hfill \(\Box\)
\end{definition}
\begin{proposition}[Convex Envelope for 2-Trees]\label{prop:convexEnvForTree}
Consider a 2-tree $A:=A_1 \cup A_2$ with \(A_1 \equiv \{\bm{x}^{(j)}: j \in \mathcal{J}\}\) and, without loss of generality, assume that $I(\bm{x}^{(j)}) \subset I(\bm{x}^{(j + 1)})$ and each \(\bm{x}^{(j)}\) is non-increasing, i.e., \(x^{(j)}_i \geq x^{(j)}_{i+1}\) (all \(1\)-entries are in the front for \(A_1\)-elements). Besides, let $i^*$ denote the entry index that distinguishes \(\bar{\bm{x}}\) from \(A_2\), i.e., $\bm{e}_{I(\bar{\bm{x}}) \setminus \{i^*\}} \in A_2$. Then, the following statements hold.
    \begin{enumerate}
        \item If $I(\bm{x}) \subset I(\bar{\bm{x}})$ for all $\bm{x} \in A_1 \cup A_2$, then
        \[
    \begin{aligned}
    \overline{\operatorname{co}} \left(I_A\right)
    = &\left\{ (\bm{x}, \theta) \in [0, 1]^d \times \mathbb{R}_+ \,\middle|\, 
    \begin{aligned}
    & \left(\bm{x}, \theta\right) \in \overline{\operatorname{co}}\left(I_{A_1}\right) \cap \overline{\operatorname{co}}\left(I_{A_2}\right)\\
    \end{aligned}
    \right\}
    \end{aligned}
    \]
    \item Otherwise,
    \[
    \begin{aligned}
    \overline{\operatorname{co}} \left(I_A\right)
    = &\left\{ (\bm{x}, \theta) \in [0, 1]^d \times \mathbb{R}_+ \,\middle|\, 
    \begin{aligned}
    & \left(\bm{x}, \theta\right) \in \overline{\operatorname{co}}\left(I_{A_1}\right) \cap \overline{\operatorname{co}}\left(I_{A_2}\right)\\
    &\begin{aligned}
        \theta \ \geq & \ \frac{1}{2} x_{|I(\bar{\bm{x}})|} + \frac{1}{2} x_{i^*} + \sum_{i \in I(\bar{\bm{x}}) \setminus \{|I(\bar{\bm{x}})|, i^*\}} x_i\\
        &- \frac{1}{2}x_{|I(\bar{\bm{x}})| + 1} - \sum_{i \notin I(\bar{\bm{x}}) \cup \left\{|I(\bar{\bm{x}})| + 1\right\}}x_i \\
        &+ 2 - |I(\bar{\bm{x}})|
    \end{aligned}
    \end{aligned}
    \right\}.
    \end{aligned}
    \]
    \end{enumerate}
\end{proposition}

\subsubsection{Enumeration of partitions.} The second strategy takes advantage of the piecewise \emph{constant} structure of the regularizers~\eqref{eq:simpleGFunction}. In particular, because the closure of each partition \(\operatorname{cl}(\mathcal{P}^t_k)\) admits a polyhedral representation and the regularizer \(g_i(\mathcal{X}^t, \bm{x})\) equals a constant \(g^*_i(\bm{x}_{(k)})\) on \(\mathcal{P}^t_k\), we can recast the scenario subproblem \(\min_{\bm{x}}\{f_i(\bm{x})+g_i(\mathcal{X}^t, \bm{x})\}\) as
\begin{align*}
\min_{\bm{x} \in X} \big\{f_i(\bm{x}) + g_i(\mathcal{X}^t, \bm{x})\big\} \ = \ \min_{k \in [K_t]} \bigg\{ \min_{\bm{x} \in \operatorname{cl}(\mathcal{P}_k^t)} \big\{f_i(\bm{x})\big\} + g^*_i(\bm{x}_{(k)})\bigg\}.
\end{align*}
This decomposes each scenario subproblem further into \(K_t\) smaller formulations, which once again can be solved in parallel. In addition, we notice that each formulation \(\min_{\bm{x} \in \operatorname{cl}(\mathcal{P}_k^t)} \{f_i(\bm{x})\}\) is a MILP when \(f_i\) arises from a two-stage stochastic MILP with CCR. In our experience, this strategy is oftentimes more effective in speeding up~\ref{problem:generalized_dual_decomposition} than the strong valid inequalities in the last section.

\subsubsection{Nested Voronoi Partitions and Partition Pruning.}
Although the enumeration-based solution to the scenario subproblems avoids introducing auxiliary binary variables and big-M coefficients, it still requires solving a restricted subproblem \(\min_{\bm{x} \in \operatorname{cl}(\mathcal{P}_k^t)} \{f_i(\bm{x})\}\) for each partition cell. Under a standard Voronoi partition, adding even a single new \(\bm{x}_{(k)}\) may change the entire partition, and none of the previously computed restricted subproblems can be reused. To address this issue, our third strategy introduces a \emph{nested} Voronoi partition that refines the partition only locally and therefore preserves the previously computed solutions on the unchanged cells.


Specifically, for the partition \(\{\mathcal{P}^t_k: k \in [K_t]\}\) in iteration \(t\), we denote by \(\mathcal{I}(k) := \{i \in [N]: \bm{x}^{i*} \in \mathcal{P}^t_k\}\) the indices of subproblem optimal solutions \(\bm{x}^{i*}\) lying in partition cell \(\mathcal{P}^t_k\). Then, in iteration \(t+1\), we keep \(\mathcal{P}^t_k\) unchanged if \(\mathcal{I}(k)=\varnothing\) (i.e., if no new optimal solutions appeared in this cell) and only apply a around of Voronoi partition within those \(\mathcal{P}^t_k\) with \(\mathcal{I}(k)\neq\varnothing\), based on the subset of new solutions \(\{\bm{x}^{i*}: i \in \mathcal{I}(k)\}\) appearing in those cells.

Furthermore, we mimic the branch-and-bound method to prune partition cells. To that end, for each scenario $i\in[N]$ and each partition cell $k \in [K_t]$, we let
\[
v_{k}^{it} \ := \ \min_{\bm{x} \in \operatorname{cl}(\mathcal{P}_k^t)} \big\{f_i(\bm{x})\big\} + g^*_i(\bm{x}_{(k)})
\]
denote the optimal value of the scenario subproblem restricted to cell $\mathcal{P}_k^t$. Because $\sum_{i \in [N]}v_{k}^{it}$ is a valid lower bound for the objective value of any solution in this cell, the inequality
\[
\sum_{i \in [N]}v_{k}^{it} > \mathit{ub},
\]
where \(ub\) pertains to the best incumbent solution by far, implies that no solution in \(\mathcal{P}_k^t\) will outperform the current incumbent and thus \(\mathcal{P}_k^t\) may be discarded from future consideration in scenario subproblem $i$. In Section~\ref{sec:numerical}, we demonstrate that this strategy speeds up the convergence of~\ref{problem:generalized_dual_decomposition} significantly. 

\section{Extensions of GDD} \label{sec:extensions}
We extend~\ref{problem:generalized_dual_decomposition} to a more general, constrained form. We shall see in Sections~\ref{sec:robustOptimization}--\ref{sec:bilevelOptimization} that this generalization allows us to apply~\ref{problem:generalized_dual_decomposition} (and hence its parallel-computing implementation) to several popular models in real-world applications, including robust optimization, chance constraints, and bilevel optimization with multiple followers. We start by presenting the constrained form of~\ref{problem:generalized_dual_decomposition}.
\begin{theorem}[GDD with constraints] \label{prop:gddCons}
Consider a constrained optimization model
\begin{align}
u^* \ := \ \min_{\bm{x}} \ & \ f(\bm{x}) \nonumber \\
\operatorname{s.t.} \ & \ \bm{x} \in \bigcap_{i=1}^N X_i, \tag{C-SP} \label{problem:consP}
\end{align}
where each \(X_i \subseteq \mathbb{Z}^{p_1} \times \mathbb{R}^{n_1-p_1}\) denotes a mixed-integer set and $f: \mathbb{R}^{n_1} \rightarrow \mathbb{R}$ denotes a deterministic objective function such that 
$0 \leq \inf_{\bm{x} \in \cup_{i \in [N]}X_i} f(\bm{x}) \leq \sup_{\bm{x} \in \cup_{i \in [N]} X_i} f(\bm{x}) < \infty$. 
Then,~\eqref{problem:consP} admits the following~\ref{problem:generalized_dual_decomposition} representation:
\begin{align}
u^* \ = \ \max_{g_i(\cdot)} \ & \ \frac{1}{N}\sum_{i=1}^N\min_{\bm{x}^i \in X_i} \big\{f(\bm{x}^i) + g_i(\bm{x}^i) \big\} \nonumber \\
\operatorname{s.t.} \ & \ \sum_{i=1}^N g_i(\bm{x}) = 0 \qquad \forall \bm{x} \in \bigcup_{i=1}^N X_i \tag{C-GDD} \label{problem:consD} \\
& \ g_i: \ \ \bigcup_{i=1}^N X_i \rightarrow \mathbb{R} \quad \forall i \in [N]. \nonumber
\end{align}
\end{theorem}



\subsection{Robust Optimization} \label{sec:robustOptimization}
Robust optimization considers a cost function \(f(\bm{x}, \bm{\xi}): X \times \Xi \rightarrow \mathbb{R}\) under uncertainty described by a random vector \(\bm{\xi}\), supported on a finite set \(\Xi := \{\bm{\xi}^i: i \in [N]\}\) of scenarios. Instead of minimizing the expectation of the cost function as in~\eqref{problem:primal}, robust optimization seeks to hedge against the worst-case outcome of \(\bm{\xi}\),
\begin{align}
v_{\text{RO}} \ := \ \min_{\bm{x} \in X} \max_{\bm{\xi} \in \Xi} \ f(\bm{x}, \bm{\xi}). \tag{RO} \label{problem:robust}
\end{align}
We remark that function \(f\) may be nonlinear and even nonconvex, for example, \(f\) may arise from the optimal value of a MILP subproblem. This renders the standard duality techniques~\citep[see, e.g.,][]{ben2009robust} inapplicable. Nevertheless, because we consider finite possibilities of \(\bm{\xi}\),~\eqref{problem:robust} can be recast as an epigraphical form by introducing an auxiliary variable $\theta$ to obtain
\begin{align*}
v_{\text{RO}} \ = \ \min_{(\bm{x}, \theta) \in X \times \mathbb{R}} \ & \ \theta \\
\text{s.t.} \ & \ \theta \geq f(\bm{x}, \bm{\xi}^i) \quad \forall i \in [N].
\end{align*}
Now, if we specify the sets \(X_i := \left\{(\bm{x}, \theta) \in X \times \mathbb{R}: \theta \geq f(\bm{x}, \bm{\xi}^i)\right\}\) for all \(i \in [N]\), then Theorem~\ref{prop:gddCons} produces the following~\ref{problem:generalized_dual_decomposition} representation for~\eqref{problem:robust}:
\begin{align*}
v_{\text{RO}} \ = \ \max_{g_i(\cdot, \cdot)} \ & \ \frac{1}{N}\sum_{i=1}^N\min_{(\bm{x}^i, \theta^i) \in X_i} \big\{f(\bm{x}^i, \bm{\xi}^i) + g_i(\bm{x}^i, \theta^i) \big\} \nonumber \\
\operatorname{s.t.} \ & \ \sum_{i=1}^N g_i(\bm{x}, \theta) = 0 \qquad \forall (\bm{x}, \theta) \in \bigcup_{i=1}^N X_i \\
& \ g_i: \ \ \bigcup_{i=1}^N X_i \rightarrow \mathbb{R} \qquad \forall i \in [N]. \nonumber
\end{align*}
In the above representation, each subproblem involves only one scenario \(\bm{\xi}^i\), becomes a deterministic optimization model, and can be solved in parallel with other subproblems. Hence, one can adapt Algorithm~\ref{alg:MIPGDD} to solve~\eqref{problem:robust}.


\subsection{Chance-Constraint Program}
\label{sec:chanceConstraints}
When the randomness \(\bm{\xi}\) arises from safety conditions of a system, a convenient way towards safe operations is to adopt a chance-constraint program,
\begin{align}
v_{\text{CC}} \ := \ \min_{\bm{x} \in X} \ & \ f(\bm{x}) \nonumber \\
\text{s.t.} \ & \ \mathbb{P}\Big\{S(\bm{x}, \bm{\xi}) \leq 0\Big\} \ \geq \ 1 - \beta, \tag{CC} \label{problem:chance}
\end{align}
where \(f(\cdot)\) evaluates the operational cost, \(\mathbb{P}\) denotes an empirical probability distribution of \(\bm{\xi}\) through its past observations \(\{\bm{\xi}^i: i \in [N]\}\), inequality \(S(\bm{x}, \bm{\xi}) \leq 0\) models the system safety conditions, and \(\beta \in (0, 1)\) denotes a risk threshold when operating the system. With the help of auxiliary binary variables \(z_i\) indicating whether \(S(\bm{x}, \bm{\xi}^i) \leq 0\), we rewrite~\eqref{problem:chance} as a mixed-integer program~\citep[see, e.g.,][]{luedtke2014branch},
\begin{align*}
v_{\text{CC}} \ = \ \min_{\bm{x} \in X, \bm{z} \in \left\{0, 1\right\}^N} \ & \ f(\bm{x}) \\
\text{s.t.} \ & \ z_i = 0 \ \Rightarrow \ S(\bm{x}, \bm{\xi}^i) \leq 0 \quad \forall i \in [N] \\
& \ \frac{1}{N}\sum_{i=1}^N z_i \leq \beta.
\end{align*}
We remark that the inequalities \(S(\bm{x}, \bm{\xi}^i) \leq 0\) may take various forms and existing works have provided techniques for reformulating and strengthening the associated logic constraints using MILP~\citep[see, e.g.,][]{luedtke2010integer,kuccukyavuz2022chance}. Now, if we specify the sets \(X_i:= \{(\bm{x}, \bm{z}) \in X \times \{0, 1\}^N: z_i = 0 \ \Rightarrow \ S(\bm{x}, \bm{\xi}^i) \leq 0, \ \bm{e}^{\top} \bm{z} \leq \beta N\}\) for all \(i \in [N]\), then Theorem~\ref{prop:gddCons} produces the following~\ref{problem:generalized_dual_decomposition} representation for~\eqref{problem:chance}:
\begin{align*}
v_{\text{CC}} \ = \ \max_{g_i(\cdot, \cdot)} \ & \ \frac{1}{N}\sum_{i=1}^N\min_{(\bm{x}^i, \bm{z}^i) \in X_i} \big\{f(\bm{x}^i) + g_i(\bm{x}^i, \bm{z}^i) \big\} \nonumber \\
\operatorname{s.t.} \ & \ \sum_{i=1}^N g_i(\bm{x}, \bm{z}) = 0 \qquad \forall (\bm{x}, \bm{z}) \in \bigcup_{i=1}^N X_i \\
& \ g_i: \ \ \bigcup_{i=1}^N X_i \rightarrow \mathbb{R} \qquad \forall i \in [N]. \nonumber
\end{align*}
We notice that, in the above representation, each subproblem involves the logic constraint for an individual scenario \(\bm{\xi}^i\) only and thus becomes deterministic. As in the~\ref{problem:consD} representation for~\eqref{problem:robust}, all subproblems in this~\eqref{problem:chance} representation can be solved in parallel, and hence one can adapt Algorithm~\ref{alg:MIPGDD} to solve~\eqref{problem:chance}.



\subsection{Bilevel Optimization with Multiple Followers}
\label{sec:bilevelOptimization}
The final example is concerned with bilevel optimization, where a leader and a group of \(N\) followers take actions sequentially in a Stackelberg manner. We denote the leader's actions by \(\bm{x}\) and the followers' reactions by \(\bm{y}_{1:N}:=[\bm{y}_1, \ldots, \bm{y}_N]\), where \(\bm{y}_i\) denotes the actions of follower \(i\). In anticipation of the followers' (optimal) reactions, the leader solves a model
\begin{align}
v_{\text{BM}} \ := \ \min_{\bm{x}, \bm{y}_{1:N}} \ & \ f(\bm{x}, \bm{y}_{1:N}) \nonumber \\
\text{s.t.} \ & \ \bm{y}_i \in \arg\min_{\bm{y}} \{h_i(\bm{x}, \bm{y})\} \quad \forall i \in [N] \tag{BM} \label{model:bilevel} \\
& \ \left(\bm{x}, \bm{y}_{1:N}\right) \in X, \nonumber
\end{align}
where \(f: \mathbb{R}^{n_1+Nn_2} \rightarrow \mathbb{R}\) and \(\{h_i: \mathbb{R}^{n_1+n_2} \rightarrow \mathbb{R}, i \in [N]\}\) represent the leader's and the followers' cost functions, respectively. A computational challenge of solving~\eqref{model:bilevel} is the co-existence of multiple optimality conditions \(\bm{y}_i \in \arg\min_{\bm{y}} \{h_i(\bm{x}, \bm{y})\}\) in one formulation, which amplifies its size drastically when \(N\) is large. We can alleviate this challenge by applying Theorem~\ref{prop:gddCons} and recasting~\eqref{model:bilevel} as
\begin{align*}
v_{\text{BM}} \ = \ \max_{g_i(\cdot, \cdot)} \ & \ \frac{1}{N}\sum_{i=1}^N\min_{\big(\bm{x}^i, \bm{y}^i_{1:N}\big) \in X_i} \big\{f(\bm{x}^i, \bm{y}^i_{1:N}) + g_i(\bm{x}^i, \bm{y}^i_{1:N}) \big\} \nonumber \\
\operatorname{s.t.} \ & \ \sum_{i=1}^N g_i(\bm{x}, \bm{y}_{1:N}) = 0 \qquad \forall (\bm{x}, \bm{y}_{1:N}) \in \bigcup_{i=1}^N X_i \\
& \ g_i: \ \ \bigcup_{i=1}^N X_i \rightarrow \mathbb{R} \qquad \quad \forall i \in [N], \nonumber
\end{align*}
where \(X_i := \{(\bm{x}, \bm{y}_{1:N}) \in X: \bm{y}_i \in \arg\min_{\bm{y}} \{h_i(\bm{x}, \bm{y})\}\}\). In contrast to~\eqref{model:bilevel}, this~\ref{problem:consD} representation distributes the \(N\) optimality conditions among subproblems, each of which is now a bilevel model with an individual follower. This allows us to solve these simpler bilevel models in parallel using Algorithm~\ref{alg:MIPGDD}.



\section{Numerical Experiments} \label{sec:numerical}

We conduct numerical experiments to demonstrate the proposed~\eqref{problem:generalized_dual_decomposition} formulation and Algorithm~\ref{alg:MIPGDD} for solving~\eqref{problem:primal}. We report the results for \(\bm{x}\) being binary in Section~\ref{sec:GDDForBinaryNumerical} and \(\bm{x}\) being mixed-integer in Section~\ref{sec:GDD-mit-numerical}. In both cases, we report
\begin{enumerate}[label=(\roman*)]
    \item the optimality gap proved by~\ref{problem:dual_decomposition} and~\ref{problem:generalized_dual_decomposition};
    \item the comparisons with solving deterministic equivalent formulations (DEF, solved by Gurobi) and with implementing primal (Benders-type) decomposition algorithms; and
    \item the speedup of~\ref{problem:generalized_dual_decomposition} through the strategies in Section~\ref{subsec:improvement} and through parallel computing.
\end{enumerate}

We implemented~\ref{problem:generalized_dual_decomposition} and all benchmark methods on the Great Lakes Cluster provided by the information and technology services at the University of Michigan. For each experiment, we used 32 CPU cores from the Intel(R) Xeon(R) Gold 6154 CPU @ 3.00GHz and a RAM of 32GB.

\subsection{GDD for binary tender} \label{sec:GDDForBinaryNumerical}
We solve stochastic multiple binary knapsack (SM\(b\)K) problem instances~\citep[see][]{angulo2016improving} with additional constraint uncertainty. For completeness, we describe its model as follows:
\begin{align}
    \min_{\bm{x} \in \{0, 1\}^{n}, \bm{z} \in \{0, 1\}^{n}} \ & \ \bm{c}^\top \bm{x} + \bm{d}^\top \bm{z} + \mathcal{Q}(\bm{x}) \tag{SM\(b\)K} \\
    \text{s.t. } \ & \ \bm{A}\bm{x} + \bm{C}\bm{z} \geq \bm{b}, \notag
\end{align}
where $\mathcal{Q}(\bm{x}):=\mathbb{E}\left[Q_{\bm{\xi}}\left(\bm{x}\right)\right]$, \(\bm{\xi}\) follows a discrete probability distribution with \(N\) scenarios, and 
\begin{align*}
    Q_{\bm{\xi}}\left(\bm{x}\right):=\min_{y\in \left\{0, 1\right\}^n} \ & \ \bm{q}\left(\bm{\xi}\right)^\top\bm{y}\\
    \text{s.t. } \ & \ \bm{W}\bm{y} \geq \bm{h}(\bm{\xi}) - \bm{T}(\bm{\xi})\bm{x}
\end{align*}
with $\bm{c} \in \mathbb{R}^n, \bm{d} \in \mathbb{R}^n, \bm{A} \in \mathbb{R}^{m_1 \times n}, \bm{C} \in \mathbb{R}^{m_1 \times n}, \bm{b} \in \mathbb{R}^{m_1}, \bm{q}(\bm{\xi}) \in \mathbb{R}^n, \bm{W} \in \mathbb{R}^{m_2 \times n}, \bm{h}(\bm{\xi}) \in \mathbb{R}^{m_2 \times n}$, and $\bm{T}(\bm{\xi}) \in \mathbb{R}^{m_2 \times n}$.

\paragraph{Instance Generation.} We follow~\cite{angulo2016improving} to generate test instances with \(n = 10\), \(m_1 = 50\), \(m_2 = 20 \). We randomly sample the parameters from uniform distributions: entries of $\bm{A}, \bm{C}$ follow \(\text{Unif}[0, 100]\), entries of $\bm{T}(\xi)$ follow $\text{Unif}[-50, 100]$, entries of $\bm{W}$ follow $\text{Unif}[20, 80]$, entries of $\bm{b}$ follow $\text{Unif}[0.5 (\bm{A} \bm{e} + \bm{C} \bm{e}), 0.75 (\bm{A} \bm{e} + \bm{C} \bm{e})]$, each entry $\bm{h}_i(\xi), i\in [m_2]$ follows $\text{Unif}[0.5 (\bm{W}_i^{\top}\bm{e} + \bm{T}_i(\xi)^{\top} \bm{e}), 0.75 (\bm{W}_i^{\top}\bm{e} + \bm{T}_i(\xi)^{\top} \bm{e})]$ if $(\bm{W}_i^{\top}\bm{e} + \bm{T}_i(\xi)^{\top} \bm{e}) \geq 0$ or otherwise follows $\text{Unif}[1.5 (\bm{W}_i^{\top}\bm{e} + \bm{T}_i(\xi)^{\top} \bm{e}), 1.25 (\bm{W}_i^{\top}\bm{e} + \bm{T}_i(\xi)^{\top} \bm{e})]$, entries of $\bm{c}$ follow $\text{Unif}[150, 250]$, entries of $\bm{d}$ follow $\text{Unif}[0, 100]$, and entries of $\bm{q}(\xi)$ follow $\text{Unif}[50, 150]$.


\begin{table}[htbp]
\centering
\begin{tabular}{lccccc}
\hline
\textbf{Methods} & \makecell{\textbf{Average}\\\textbf{Time (s)}} &
\makecell{\textbf{Average}\\\textbf{\# Iterations}} &
\makecell{\textbf{Average}\\\textbf{Gap (\%)}} &
\makecell{\textbf{Strong Duality}\\\textbf{Achieved (\%)}} \\
\hline
\ref{problem:dual_decomposition} & 119.84 & 815.55 & 0.63 & 20 \\
\ref{problem:generalized_dual_decomposition} & 9.58 & 48.55 & 0.00 & 100 \\
\hline
\end{tabular}
\caption{Convergence Comparison of~\ref{problem:dual_decomposition} and~\ref{problem:generalized_dual_decomposition} in the (SM\(b\)K) instances}
\label{tab:strongDuality}
\end{table}

\begin{figure}[htbp]
    \centering
    \begin{subfigure}{0.48\textwidth}
        \centering
        \includegraphics[width=\linewidth]{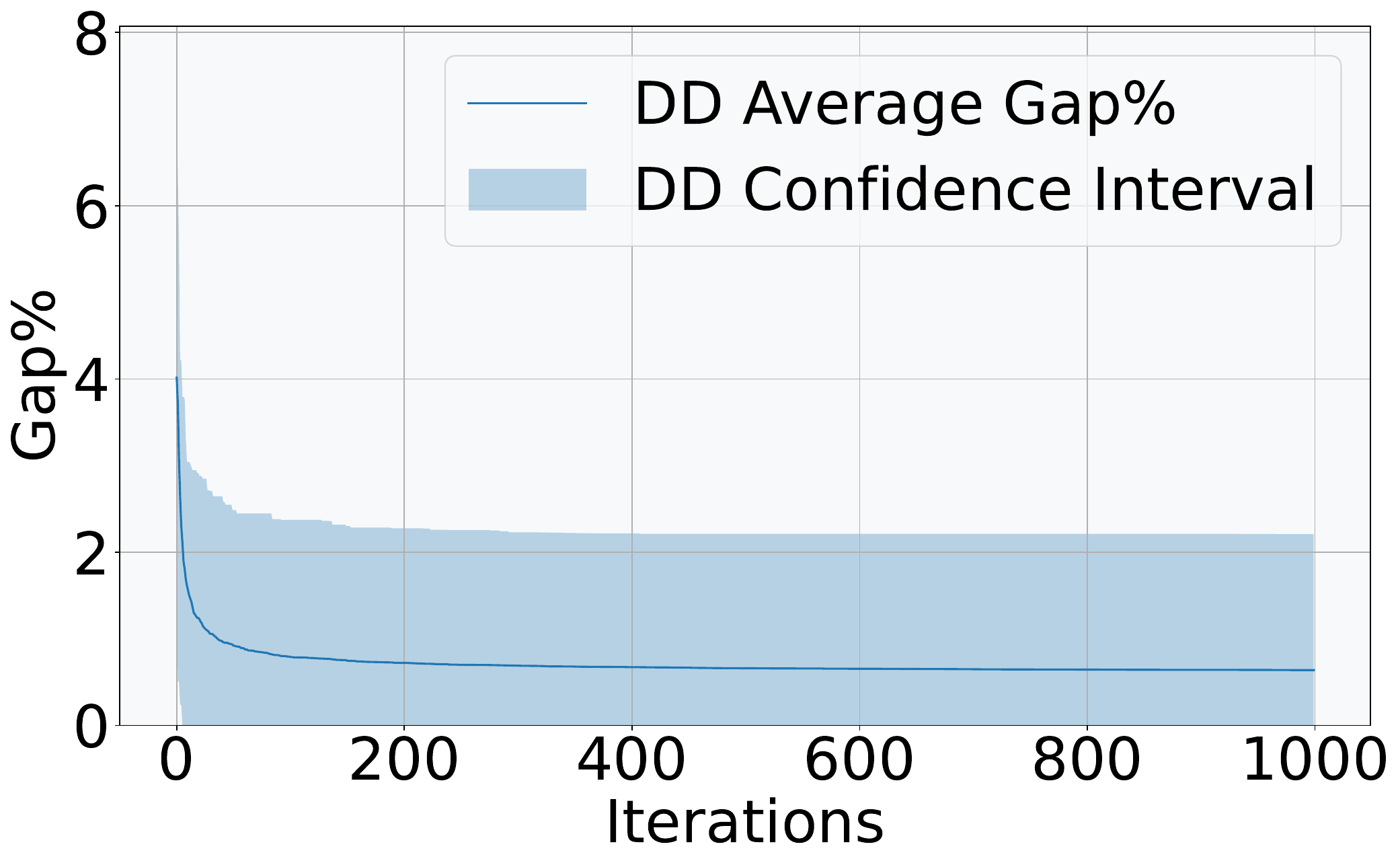}
        \caption{Average gap of~\ref{problem:dual_decomposition} versus iterations}
        \label{fig:ddConvergence}
    \end{subfigure}
    \hfill
    \begin{subfigure}{0.48\textwidth}
        \centering
        \includegraphics[width=\linewidth]{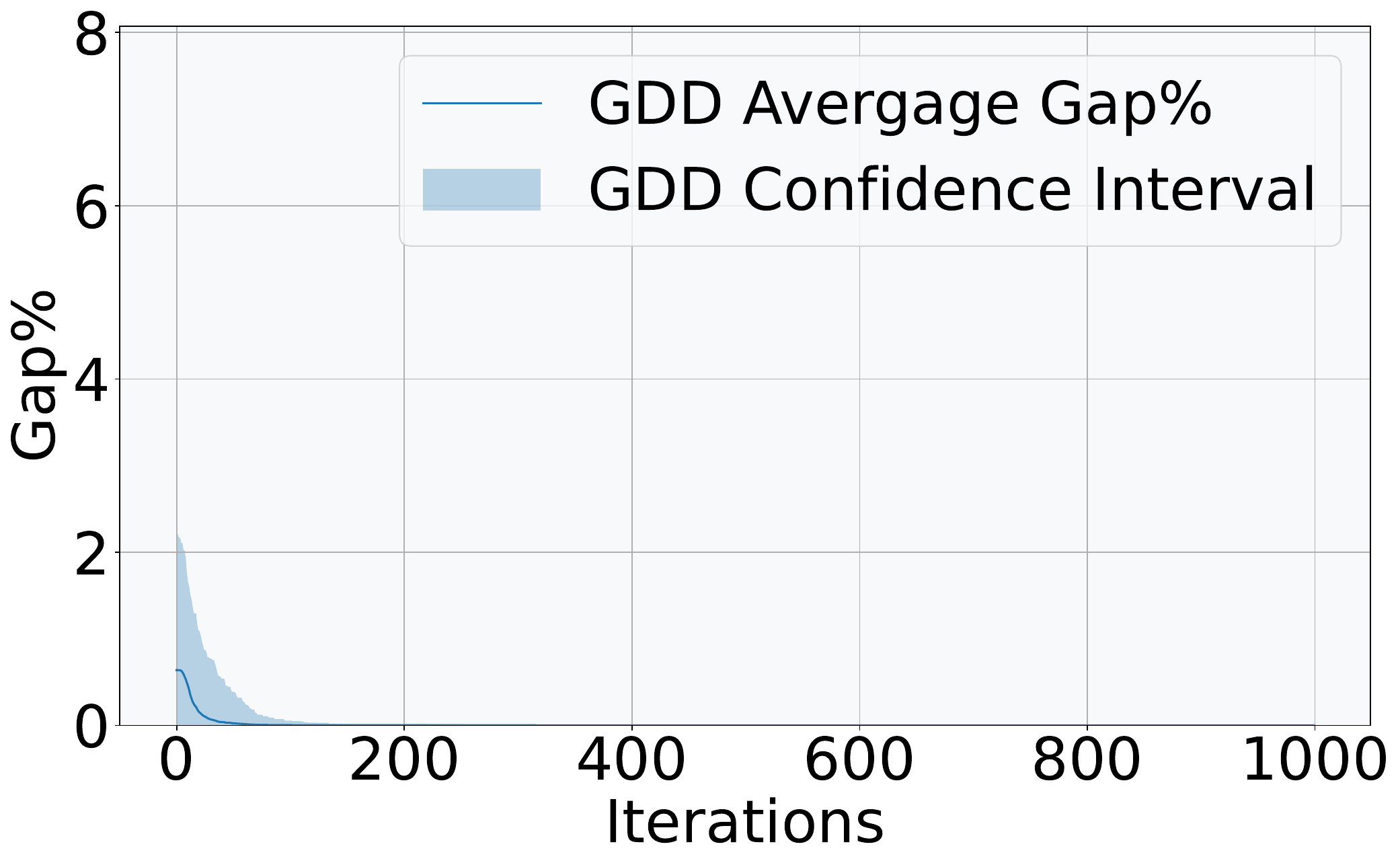}
        \caption{Average gap of~\ref{problem:generalized_dual_decomposition} versus iterations}
        \label{fig:gddConvergence}
    \end{subfigure}
    \caption{Comparison of the optimality gap proved by~\ref{problem:dual_decomposition} and~\ref{problem:generalized_dual_decomposition} in (SM\(b\)K) instances. The shaded region covers from the minimum to the maximum optimality gap across 20 random instances.}
    \label{fig:strongDuality}
\end{figure}

\subsubsection{Strong duality of~\ref{problem:generalized_dual_decomposition}}\label{subsec:ExpStrongDuality}
We compare the optimality gap achieved by~\ref{problem:dual_decomposition} and~\ref{problem:generalized_dual_decomposition} on 20 randomly generated instances with $n = 10$, $m_1 = 50$, $m_2 = 20$, and \(N = 5\). The maximum iteration number of~\ref{problem:dual_decomposition} and~\ref{problem:generalized_dual_decomposition} are set to be 1,000. We report the results in Table~\ref{tab:strongDuality} and Fig.~\ref{fig:strongDuality}. In Table~\ref{tab:strongDuality}, we report the average time, average number of iterations to achieve strong duality (i.e., zero optimality gap), average ending optimality gap, and the proportion of instances where strong duality is achieved. From this table, we observe that~\ref{problem:dual_decomposition} failed to converge to global optimum in 80\% (16/20) of the instances. In contrast, using~\ref{problem:dual_decomposition} as warm up,~\ref{problem:generalized_dual_decomposition} was able to close the gap in all instances. 
In addition, Fig.~\ref{fig:strongDuality} indicates that~\ref{problem:dual_decomposition} quickly shrank the optimality gap to around 1\% and then got stuck after about 400 iterations. In contrast,~\ref{problem:generalized_dual_decomposition} closed the gap with about 100 additional iterations. This demonstrates the strong duality of~\ref{problem:generalized_dual_decomposition} as in Theorem~\ref{prop:strong_duality}.

\begin{table}[htbp]
\centering
\begin{tabular}{cccccc}
\toprule
\# of Processes & 1 & 2 & 3 & 4 & 5 \\ 
\midrule
Average Time (s)   & 601.91 & 365.46  & 226.96 & 215.97 & 129.43   \\ 
Speedup & 1x & 1.65x & 2.65x & 2.79x & 4.65x\\
\bottomrule
\end{tabular}
\caption{Average wall clock time and speedup of the parallel implementation of~\ref{problem:generalized_dual_decomposition} in (SM\(b\)K) instances. For each \(p \in [5]\), the speedup equals the average wall clock time of solving 20 random instances using 1 process divided by that of using \(p\) processes.}
\label{tab:speedUp}
\end{table}

\begin{figure}
    \centering
    \includegraphics[width=0.6\linewidth]{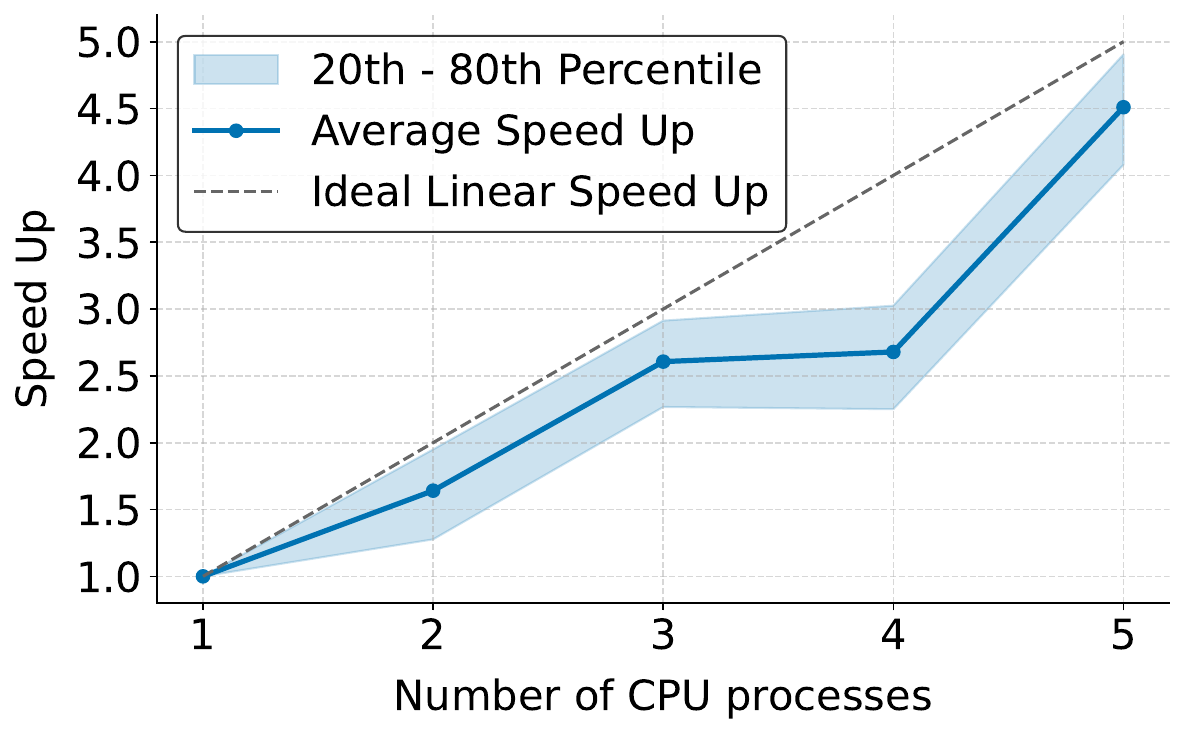}
    \caption{Visualization of speedup of parallel~\ref{problem:generalized_dual_decomposition} in (SM\(b\)K) instances. For reference, the dotted line indicates a linear speedup.}
    \label{fig:speedUp}
\end{figure}

\subsubsection{Speedup through parallel computing} \label{subsec:expParallel}
We evaluate the benefit of parallelization based on the 20 instances run in section~\ref{subsec:ExpStrongDuality} with an increasing number of processes. To this end, we assign the scenario subproblems uniformly among all available processes and solve them synchronously. When evaluating the objective value of each solution obtained, we assign the corresponding second-stage formulations to any idle processes for acceleration. We test this parallel~\ref{problem:generalized_dual_decomposition} implementation with the number of processes increasing from \(1\) to \(5\) and each process consists of 10 threads. For each \(p \in [5]\), the speedup of using \(p\) processes is evaluated by the average wall clock time of solving the 20 random instances using 1 process divided by that of using \(p\) processes. We report the results in Table~\ref{tab:speedUp} and Fig.~\ref{fig:speedUp}. These results demonstrate that the speedup proportionally increases as the number of processes increases, providing a nearly linear speedup. This suggests the potential benefit of implementing~\ref{problem:generalized_dual_decomposition} to solve~\eqref{problem:primal} with a large number of scenarios.

\subsection{\ref{problem:generalized_dual_decomposition} for mixed-integer tender} \label{sec:GDD-mit-numerical}

We solve dynamic capacity acquisition and assignment (DCAP) problem instances in Section~\ref{subsubsec:DCAP} and stochastic multiple knapsack problem (SMKP) instances in Section~\ref{subsubsec:SMKP}. In both classes of instances, the entries of \(\bm{x}\) that appear in the second-stage formulation of~\eqref{problem:primal} include continuous variables.

\subsubsection{Results of solving DCAP instances} \label{subsubsec:DCAP}

We consider the DCAP model proposed by~\cite{ahmed2003dynamic} and in the SIPLIB~\citep{ahmed2015siplib}. DCAP is an instance of~\eqref{problem:primal}, with mixed-integer first-stage variables $(u_{it}, x_{it})$ determining whether and how much capacity to acquire for resource $i \in [m]$ in period $t \in [T]$, respectively, and second-stage variables \(y_{ij}\) assigning resource \(i\) to task \(j \in [n]\) after the uncertain demands \(d^s_{jt}\) for task \(j\) in period \(t\) are revealed in each scenario \(s \in [S]\). 
For completeness, we state the DCAP model as follows:
\begin{align}
     \min_{\bm{x}} \ & \ \sum_{t=1}^\top \sum_{i=1}^m (\alpha_{it} x_{it} + \beta_{it}u_{it}) + \sum_{s=1}^S \sum_{t=1}^\top p^s Q_t^s(x) \notag\\
     \text{s.t.} \ & \ 0 \leq x_{it} \leq u_{it} \quad \forall i \in [m],\; t \in [T] \tag{DCAP} \label{dcap}\\
                   & u_{it} \in \left\{0, 1\right\} \qquad \forall i \in [m],\; t \in [T], \notag
\end{align}
with
\begin{align*}
     Q_t^s(\bm{x}) \ := \ \min_{\bm{y}} \ & \ \sum_{i=1}^m \ \sum_{j=1}^n c_{ijt}^s y_{ij}\\
                \text{s.t.} \ & \ \sum_{j=1}^n d_{jt}^s y_{ij} \leq \sum_{\tau = 1}^t x_{i\tau} \quad \forall i \in [m] \\
                           & \ \sum_{i=1}^m y_{ij} = 1 \quad \forall j \in [n] \\
                           & \ y_{ij} \in \{0, 1\} \quad \forall i \in [m], j \in [n].
\end{align*}

\paragraph{Instance generation.}
We set \(m = 3\), \(n = 6\), \(T = 2\), and \(p^s = 1/S\) for all \(s \in [S]\). Then, we sample each $\alpha_{i,t}$ from $\text{Unif}[1, 6]$, $\beta_{it}$ from $\text{Unif}[6, 41]$, $c_{ijt}^s$ from $\text{Unif}[1, 6]$, $c_{0jt}^s$ from $\text{Unif}[30, 180]$, and $d_{jt}^s$ from $\text{Unif}[0.1, 0.5]$. 


\begin{figure}[htbp]
    \centering
    \begin{subfigure}{0.48\textwidth}
        \centering
        \includegraphics[width=\linewidth]{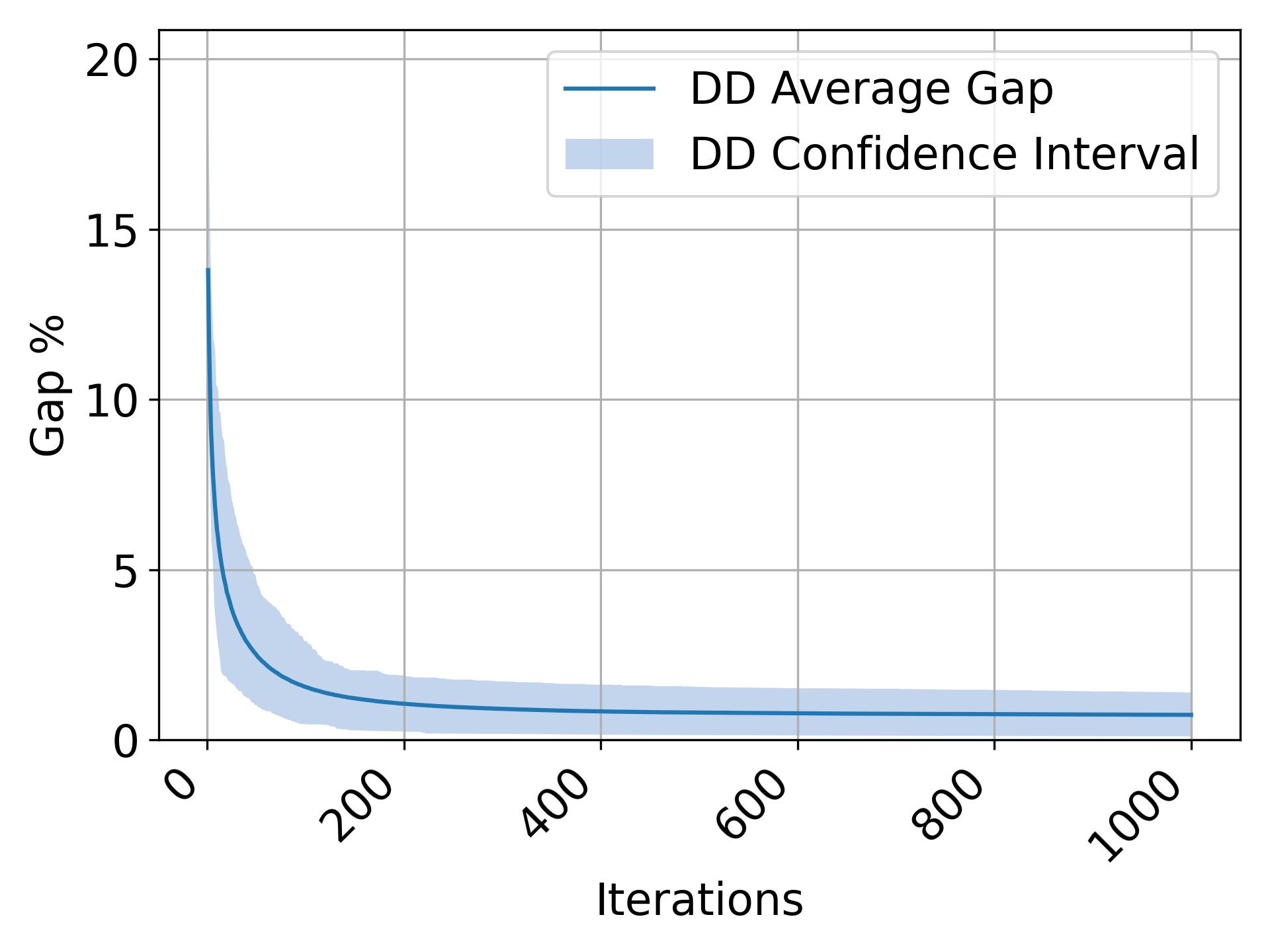}
        \caption{Average gap of~\ref{problem:dual_decomposition} versus iterations}
        \label{fig:ddConvergenceMIP}
    \end{subfigure}
    \hfill
    \begin{subfigure}{0.48\textwidth}
        \centering
        \includegraphics[width=\linewidth]{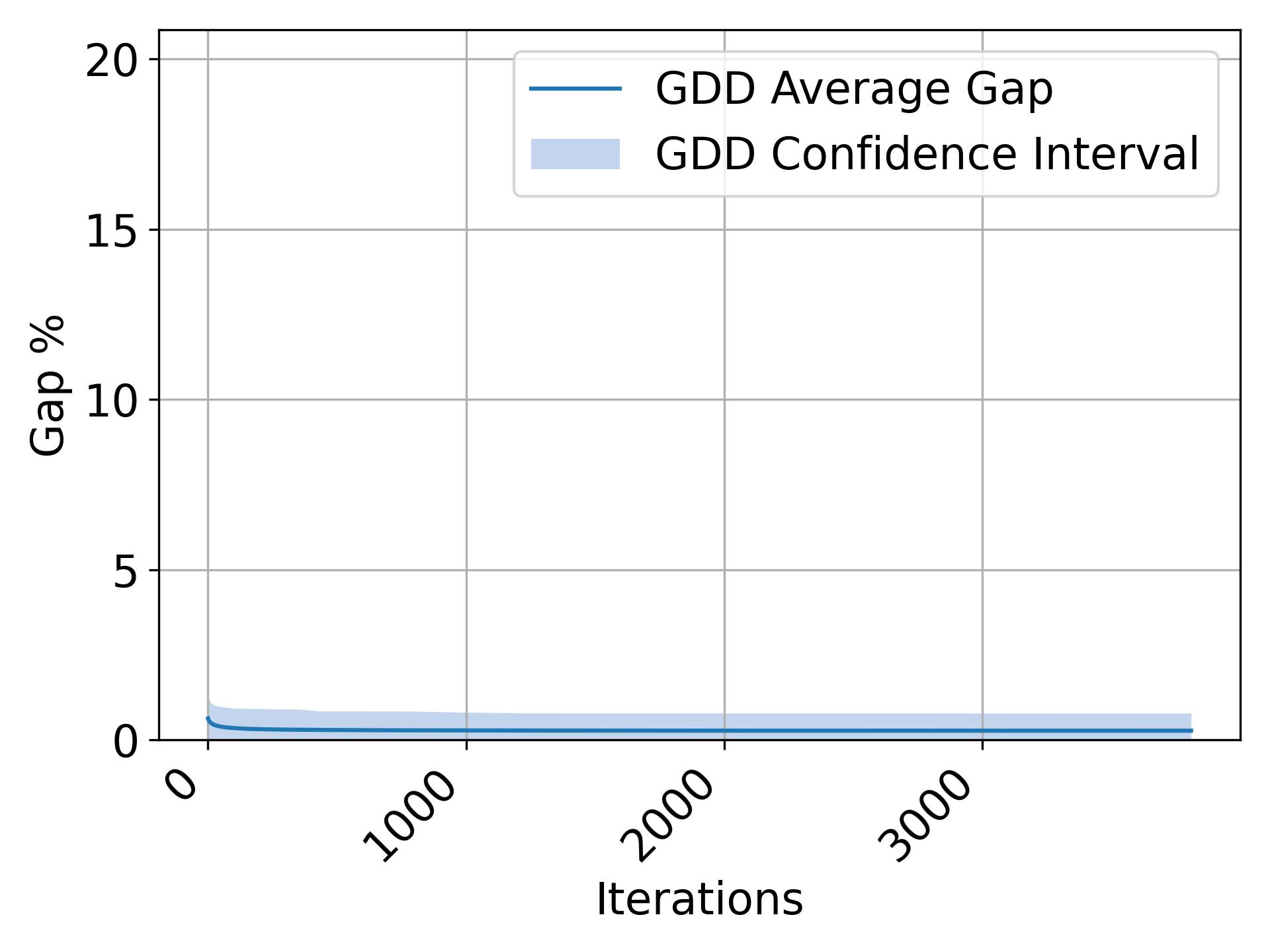}
        \caption{Average gap of~\ref{problem:generalized_dual_decomposition} versus iterations}
        \label{fig:gddConvergenceMIP}
    \end{subfigure}
    \caption{Comparison of the optimality gap proved by~\ref{problem:dual_decomposition} and~\ref{problem:generalized_dual_decomposition} in (DCAP) instances.}
    \label{fig:strongDualityMIPDCAP}
\end{figure}

\paragraph{Optimality gaps of~\ref{problem:dual_decomposition} and~\ref{problem:generalized_dual_decomposition}.}

First, we compare the optimality gap achieved by~\ref{problem:dual_decomposition} and~\ref{problem:generalized_dual_decomposition} on 46 randomly generated instances, each with \(S = 32\) scenarios, and report the results in Fig.~\ref{fig:strongDualityMIPDCAP}. 
From this figure, we observe that, while~\ref{problem:dual_decomposition} stalled at a positive optimality gap,~\ref{problem:generalized_dual_decomposition} was able to continue tightening the gap to near-zero. Similarly to the results in the (SM\(b\)K) instances, this empirically supports that the nonlinear regularization in~\ref{problem:generalized_dual_decomposition} is effective at closing the duality gap, even in these mixed-integer (DCAP) instances.

\begin{table}[!h]
    \centering
    \caption{Optimality gaps of~\ref{problem:generalized_dual_decomposition}, DEF, and the primal decomposition algorithm in (DCAP) instances.}
    \label{tab:GRB-VS-GDD-DCAP}
    \begin{tabular}{ccp{10mm}p{10mm}p{10mm}p{10mm}p{10mm}p{10mm}cc}
        \hline
        \rule{0pt}{2.8ex} & & \multicolumn{6}{c}{\ref{problem:generalized_dual_decomposition} with \# iterations of~\ref{problem:dual_decomposition} warm-start} & \multirow{2}{*}{DEF} & \multirow{2}{*}{Primal} \\
        \cline{3-8}
        \rule{0pt}{1.8ex} & & 100 & 200 & 300 & 400 & 500 & 700 & & \\
        \hline
        \multirow{4}{*}{Ending Gaps (\%)} & 25\% Quantile & 0.40 & 0.44 & 0.37 & 0.21 & 0.35 & 0.30 & 4.00 & 15.76\\
        & Median & 0.55 & 0.55 & 0.44 & 0.47 & 0.50 & 0.55 & 6.15 & 18.74\\
        & Average & 0.64 & 0.68 & 0.52 & 0.49 & 0.50 & 0.51 & 5.63 & 17.53\\
        & 75\% Quantile & 0.82 & 0.93 & 0.71 & 0.74 & 0.66 & 0.69 & 7.49 & 20.27\\
        \hline
        \multirow{4}{*}{DD Gaps (\%)} & 25\% Quantile & 0.90 & 0.88 & 0.73 & 0.51 & 0.63 & 0.46 & NA & NA\\
        & Median & 1.34 & 1.09 & 0.85 & 0.76 & 0.70 & 0.72 & NA & NA\\
        & Average & 1.40 & 1.10 & 0.87 & 0.78 & 0.73 & 0.68 & NA & NA\\
        & 75\% Quantile & 1.90 & 1.51 & 1.04 & 1.04 & 0.94 & 0.84 & NA & NA\\
        \hline
    \end{tabular}
\end{table}

\begin{figure}[htbp]
    \centering
    \includegraphics[width=0.55\textwidth]{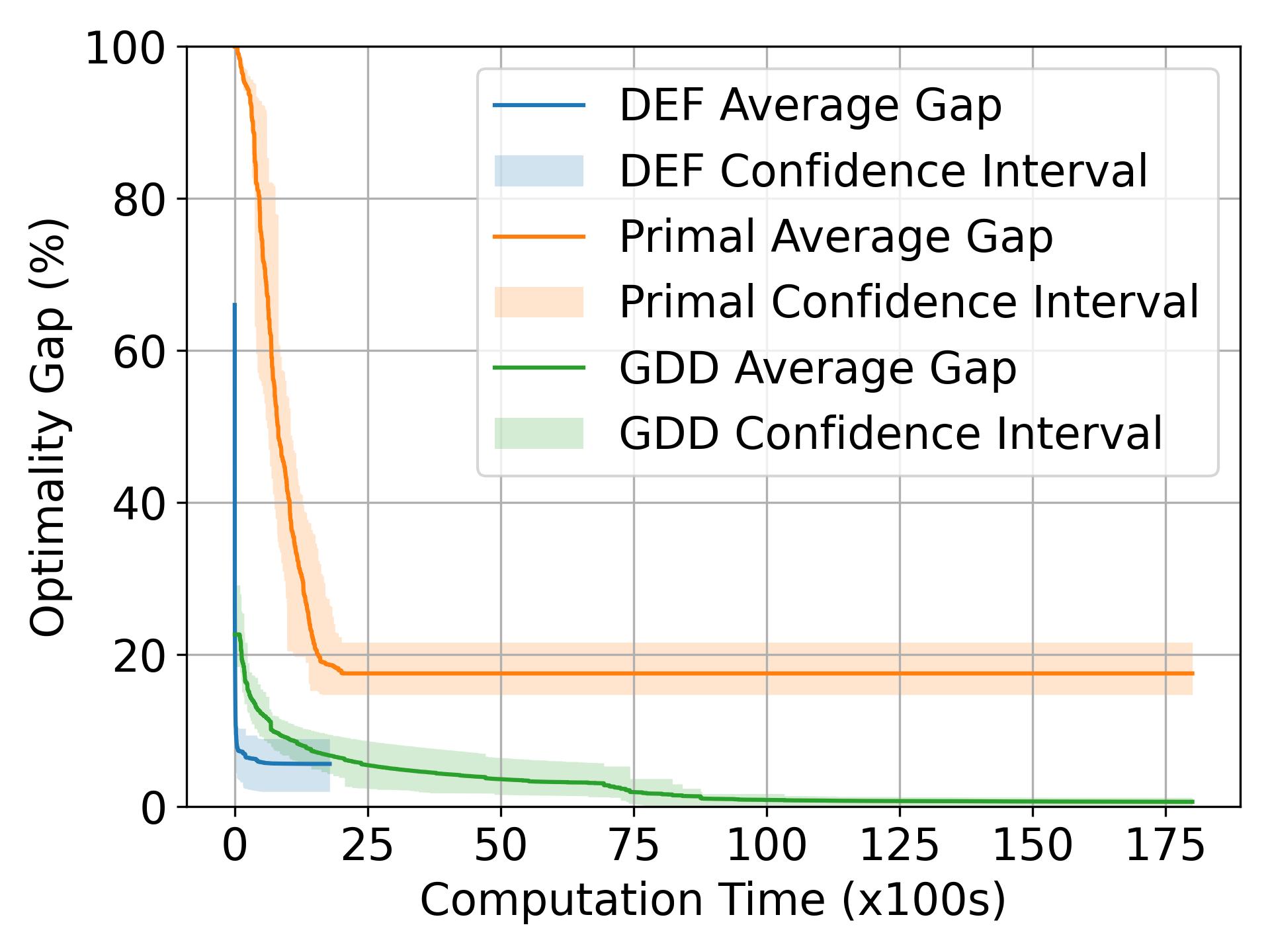} 
    \caption{Optimality gaps of~\ref{problem:generalized_dual_decomposition}, DEF, and the primal decomposition algorithm in (DCAP) instances.}
    \label{fig:GRB-Primal-GDD-Conv}
\end{figure}

\paragraph{Comparison with DEF and primal decomposition algorithms.}
Second, we compare the performance of~\ref{problem:generalized_dual_decomposition} with solving the DEF by Gurobi and a Benders-type primal decomposition algorithm. Here, DEF refers to replacing the value functions \(Q^s_t(\bm{x})\) with the second-stage formulation in the~\eqref{dcap} model. As for the primal decomposition algorithm, we implement the reverse norm cuts of~\cite{ahmed2022stochastic} and the strengthened Benders cuts of~\cite{zou2019stochastic}. In addition, we use~\ref{problem:dual_decomposition} to warm start~\ref{problem:generalized_dual_decomposition}. We summarize the results in Table~\ref{tab:GRB-VS-GDD-DCAP} and Figs.~\ref{fig:GRB-Primal-GDD-Conv},~\ref{fig:vanillaComparisonDCAP}. From Table~\ref{tab:GRB-VS-GDD-DCAP} and Fig.~\ref{fig:vanillaComparisonDCAP}, we observe that, across all warm-start settings, GDD achieved substantially smaller ending optimality gaps than both DEF and the primal decomposition algorithm. In particular, the average ending gap of GDD ranges from $0.49\%$ to $0.68\%$, whereas the corresponding average gap is $5.63\%$ by solving DEF and $17.53\%$ by applying the primal algorithm. Similar observations can be made from the 25\%, 50\%, and 75\% quantiles of the ending optimality gaps. From Fig.~\ref{fig:GRB-Primal-GDD-Conv}, we observe that Gurobi proved an optimality gap of around 5\% in 2,000 seconds before running out of memory, and the primal decomposition algorithm shrank the gap more slowly and got stuck at an optimality gap of more than 15\%. In contrast,~\ref{problem:generalized_dual_decomposition} proved a similar optimality gap as Gurobi did by 2,000 seconds. In addition, because~\ref{problem:generalized_dual_decomposition} used a much smaller memory than Gurobi did, it was able to continue improving the gap till near-optimum. This demonstrates the effectiveness of the proposed~\ref{problem:generalized_dual_decomposition} approach for solving~\eqref{problem:primal} with mixed-integer tender variables and a large number of scenarios.

\begin{figure}[htbp]
    \centering
    \begin{subfigure}[t]{0.49\textwidth}
        \centering
        \includegraphics[width=\linewidth]{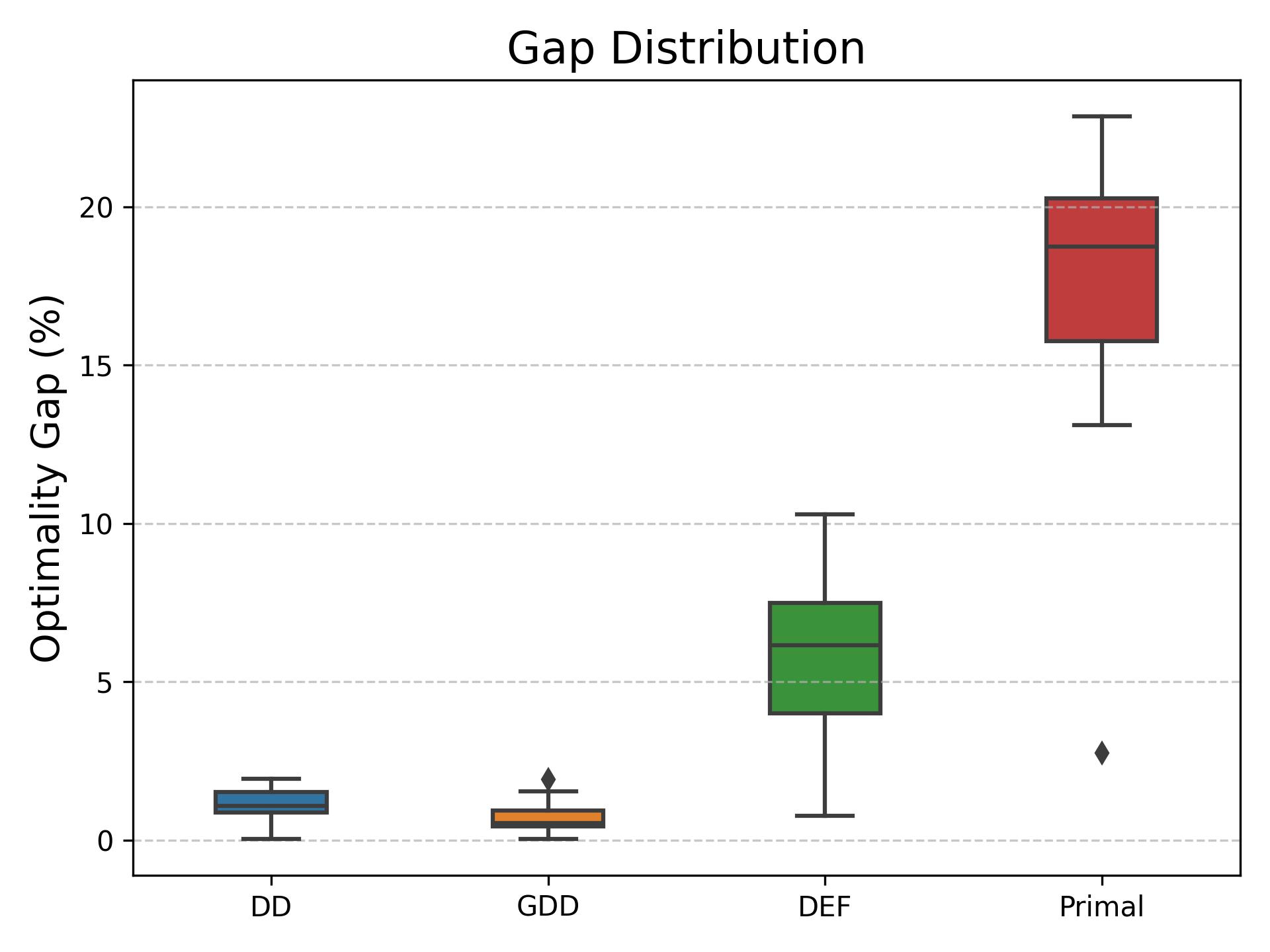}
        \caption{Gaps of~\ref{problem:generalized_dual_decomposition}, DEF, and primal algorithm}
        \label{fig:vanillaComparisonDCAP}
    \end{subfigure}
    \hfill
    \begin{subfigure}[t]{0.49\textwidth}
        \centering
        \includegraphics[width=\linewidth]{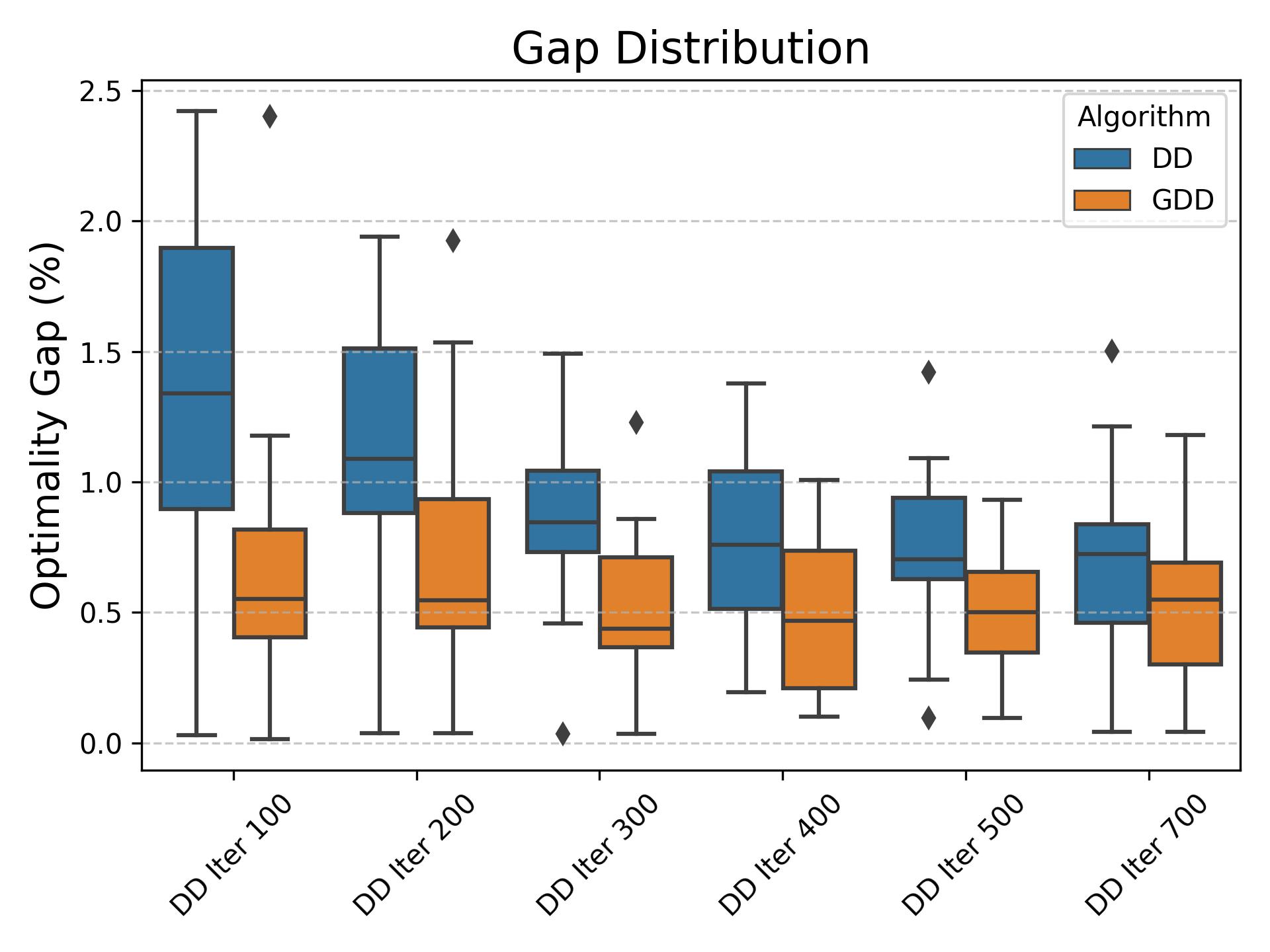}
        \caption{Gaps of~\ref{problem:generalized_dual_decomposition} with varying~\ref{problem:dual_decomposition} warm-start}
        \label{fig:DDVsGDD-only-DCAP}
    \end{subfigure}
    \caption{Comparison of the optimality gaps of~\ref{problem:generalized_dual_decomposition}, DEF, and the primal decomposition algorithm in (DCAP) instances.}
    \label{fig:DDVsGDD-Box}
\end{figure}

In addition, we compare the performance of~\ref{problem:generalized_dual_decomposition} under varying levels of~\ref{problem:dual_decomposition} warm-start and report the findings in Table~\ref{tab:GRB-VS-GDD-DCAP} and Fig.~\ref{fig:DDVsGDD-only-DCAP}. We observe that increasing the number of DD warm-start iterations produced an improved optimality gap at the end of the warm-start (or at the beginning of~\ref{problem:generalized_dual_decomposition}), and this improvement slowed down as we kept increasing the warm-start iteration (see the blue boxes in Fig.~\ref{fig:DDVsGDD-only-DCAP}). This is not surprising because (i)~\ref{problem:dual_decomposition} does not admit strong duality and (ii)~\ref{problem:dual_decomposition} updates Lagrangian multipliers through a subgradient algorithm, which was reported to have slow convergence empirically~\citep[see, e.g.,][]{yang2025globally}. Nevertheless, these different warm-start settings produced similar ending optimality gaps (see Table~\ref{tab:GRB-VS-GDD-DCAP} and the orange boxes in Fig.~\ref{fig:DDVsGDD-only-DCAP}). This suggests that~\ref{problem:generalized_dual_decomposition} is less sensitive to the~\ref{problem:dual_decomposition} warm-start and is able to robustly close the optimality gap.

\begin{figure}[htbp]
    \centering
    \begin{subfigure}[t]{0.49\textwidth}
        \centering
        \includegraphics[width=\linewidth]{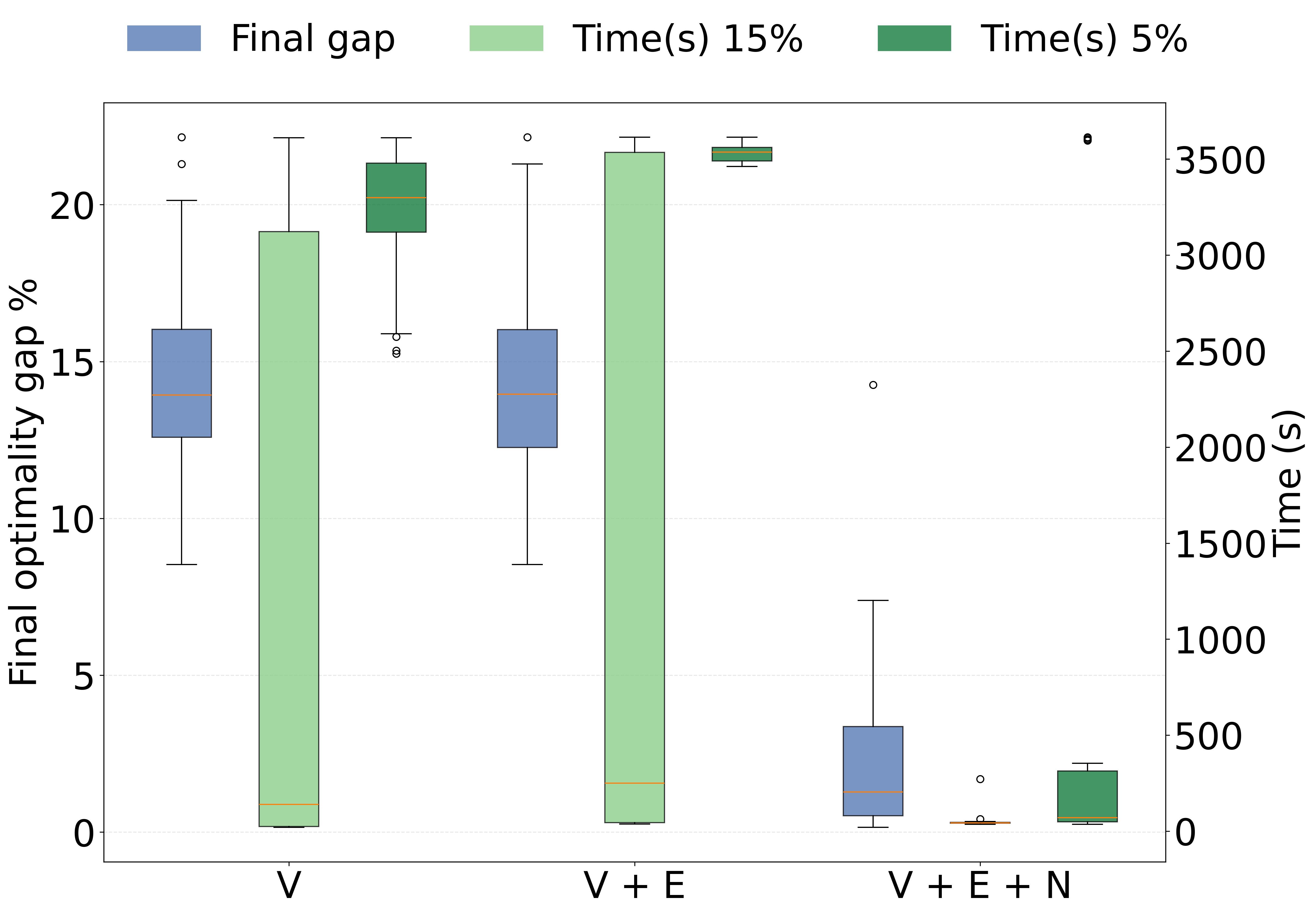}
        \caption{Comparison among improvement strategies.} 
        \label{fig:vanillaComparison}
    \end{subfigure}
    \hfill
    \begin{subfigure}[t]{0.49\textwidth}
        \centering
        \includegraphics[width=\linewidth]{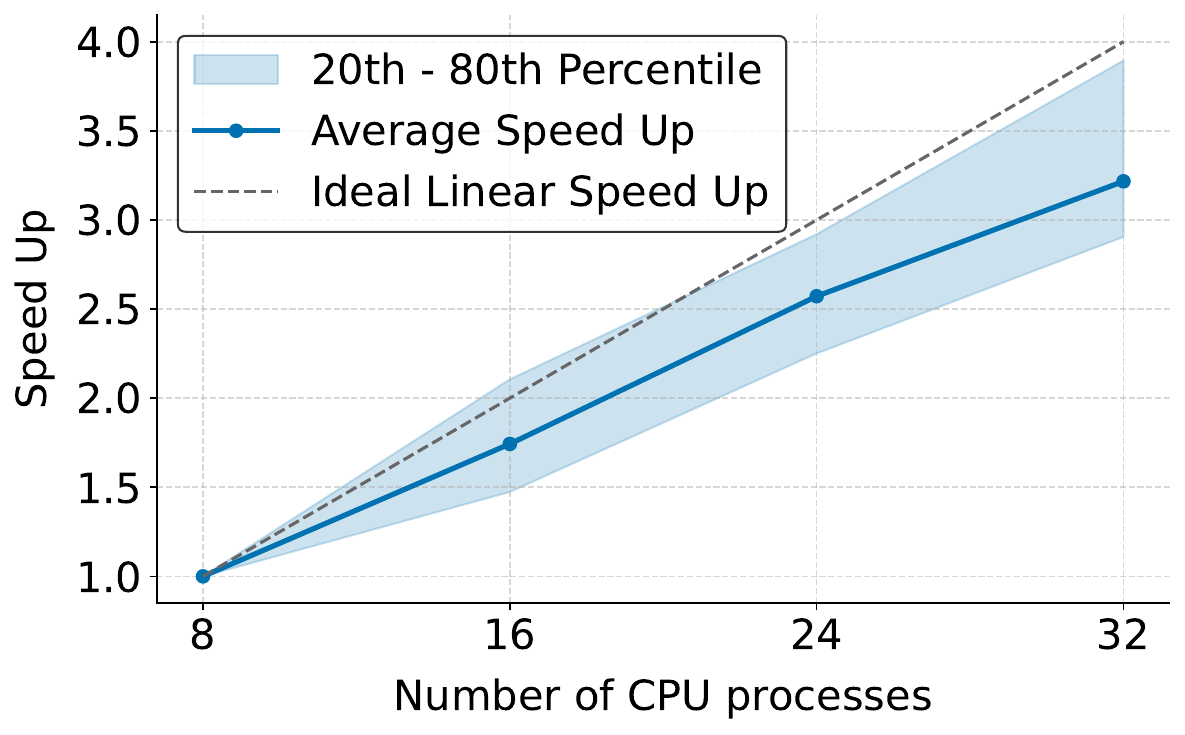}
        \caption{Speedup versus \# of CPU processes.} 
        \label{fig:speedup_processes}
    \end{subfigure}
    \caption{Acceleration of~\ref{problem:generalized_dual_decomposition} through algorithmic improvement strategies and parallel computing in (DCAP) instances. V = Vanilla, E = Enumeration, and N = Nested Voronoi.}
    \label{fig:accelerationComparison}
\end{figure}

\paragraph{Acceleration of~\ref{problem:generalized_dual_decomposition} through algorithmic strategies and parallel computing.} Finally, we demonstrate the acceleration of~\ref{problem:generalized_dual_decomposition} through the algorithmic improvement strategies in Section~\ref{subsec:improvement} and through parallel computing. We report the results in Fig.~\ref{fig:accelerationComparison}.

To demonstrate the algorithmic improvement strategies, we compare vanilla GDD (V) with its two variants: GDD with partition enumeration (V+E) and GDD with both partition enumeration and nested Voronoi refinement (V+E+N). We test these three variants without DD warm-start on randomly generated DCAP instances with 32 scenarios under the same parameter ranges as above, using a time limit of 1 hour for each run. From Fig.~\ref{fig:vanillaComparison}, we observe that both enumeration (E) and nested Voronoi (N) strategies improved upon the vanilla implementation, and the combination of (E) and (N) yields the best performance overall. For example, this combination significantly reduced the ending optimality gap, as well as the time~\ref{problem:generalized_dual_decomposition} took to reduce the optimality gap to be within 15\% and 5\%. These results indicate that exploiting the piecewise structure of the regularizers can improve the practical efficacy of~\ref{problem:generalized_dual_decomposition}.

To demonstrate the benefit of parallel computing, we generate another 50 random DCAP instances with the same parameter settings as above. We vary the number of CPU processes from 8 to 32 and impose a time limit of 1 hour for each run. From Fig.~\ref{fig:speedup_processes}, we observe that parallel computing yields a clear and consistent reduction in computation time as the number of processes increases. Although the observed speedup is sublinear, which is expected because of synchronization overhead and instance-dependent load imbalance across scenario subproblems, the gains remain substantial across the test instances.

\subsubsection{Results of solving SMKP instances} \label{subsubsec:SMKP}

We revise the (SM\(b\)K) model to be SMKP with mixed-integer tender by allowing the second half of the \(\bm{x}\) entries to take values continuously in the interval \([0,1]\).

\paragraph{Instance Generation.} We follow the same parameter settings in Section~\ref{sec:GDDForBinaryNumerical} to generate random SMKP instances.

\begin{figure}[htbp]
    \centering
    \begin{subfigure}{0.48\textwidth}
        \centering
        \includegraphics[width=\linewidth]{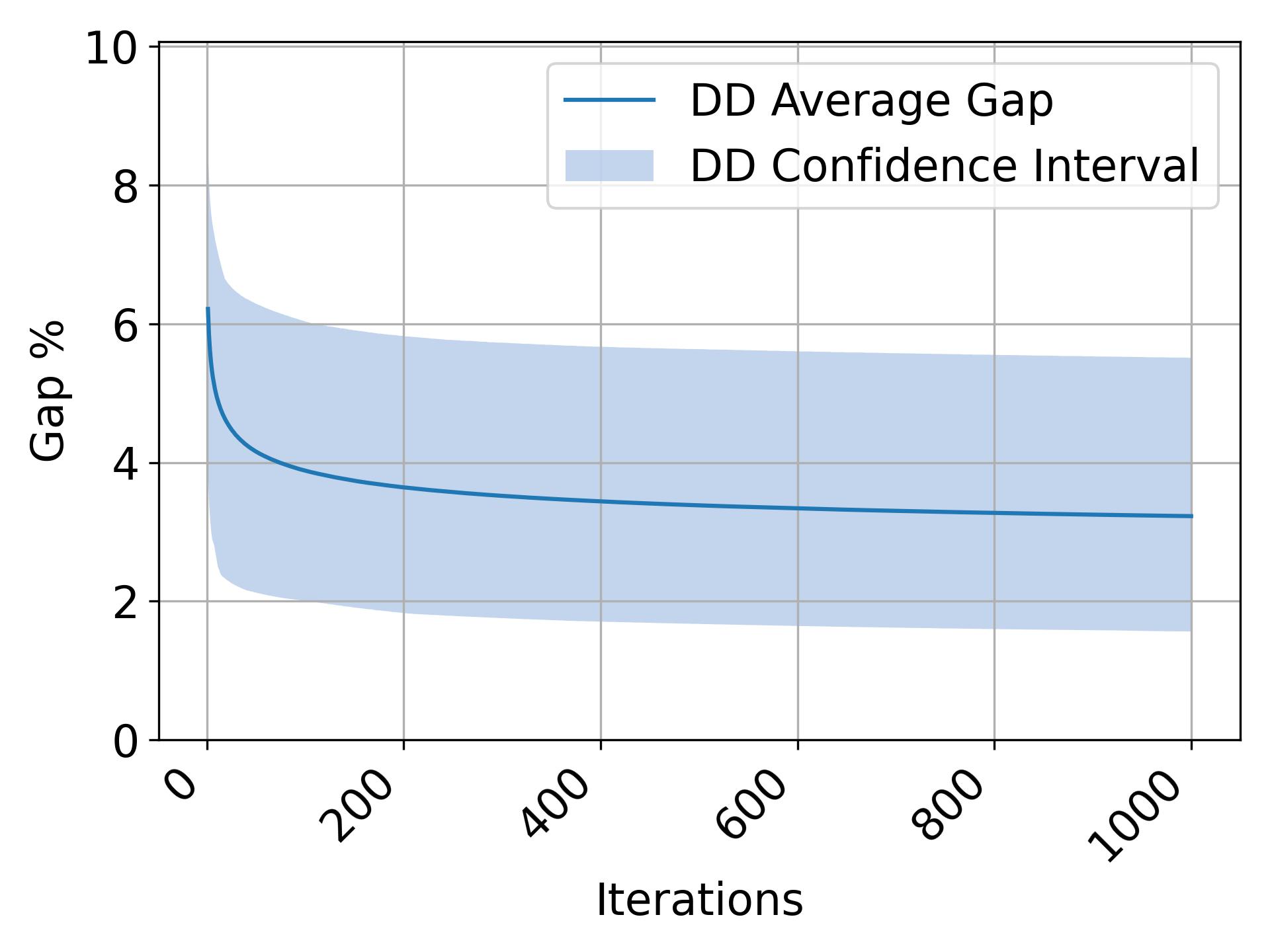}
        \caption{Average Gap of DD versus iterations}
        \label{fig:ddConvergenceMIP}
    \end{subfigure}
    \hfill
    \begin{subfigure}{0.48\textwidth}
        \centering
        \includegraphics[width=\linewidth]{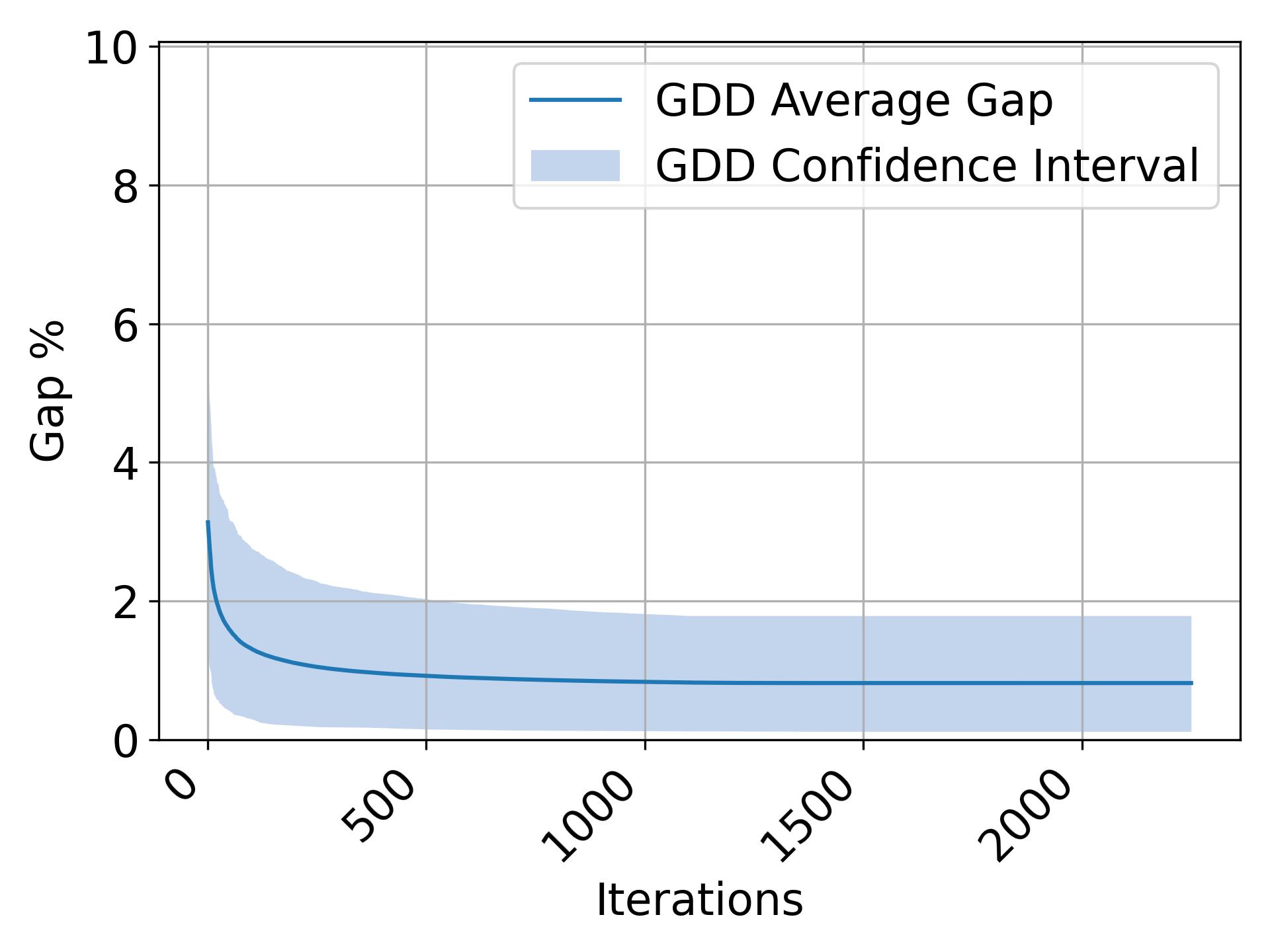}
        \caption{Average Gap of GDD versus iterations}
        \label{fig:gddConvergenceMIP}
    \end{subfigure}
    \caption{Comparison of the optimality gaps proved by DD and GDD in (SMKP) instances.}
    \label{fig:strongDualityMIPSMKP}
\end{figure}

\paragraph{Optimality gaps of~\ref{problem:dual_decomposition} and~\ref{problem:generalized_dual_decomposition}.} First, we compare the optimality gap achieved by~\ref{problem:dual_decomposition} and~\ref{problem:generalized_dual_decomposition} on 50 randomly generated instances, each with \(n=8\) mixed-integer first-stage variables and \(S=10\) scenarios, and report the results in Fig.~\ref{fig:strongDualityMIPSMKP}. From this figure, we observe that, while~\ref{problem:dual_decomposition} stalled at a positive optimality gap of more than 2\%,~\ref{problem:generalized_dual_decomposition} continued tightening the gap to be around 1\%. This is consistent with the earlier observations in (SM\(b\)K) and (DCAP) instances, and empirically supports that the effectiveness of~\ref{problem:generalized_dual_decomposition} in closing the duality gap.

\begin{figure}[htbp]
    \centering
    \includegraphics[width=0.55\linewidth]{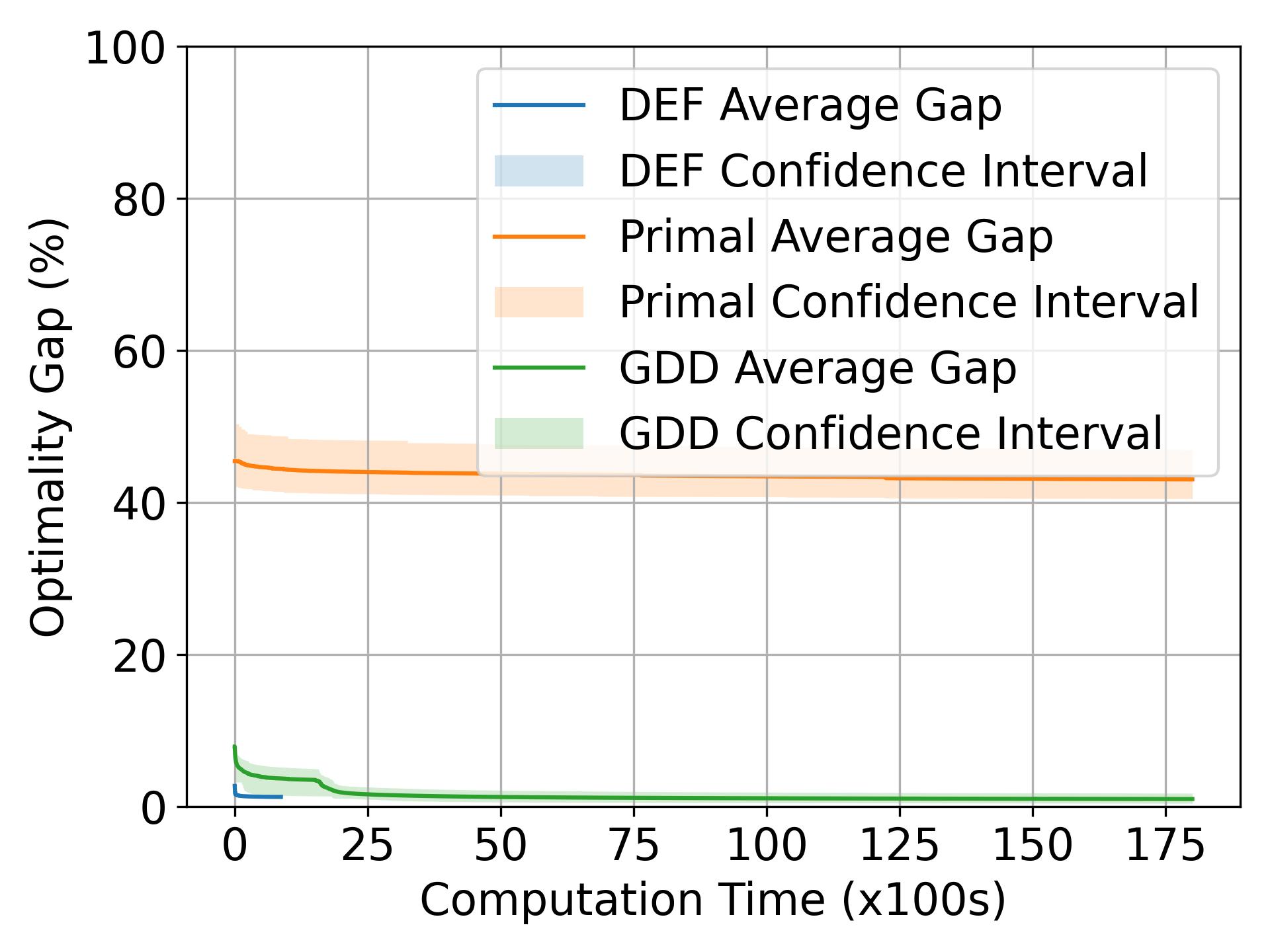}
    \caption{Optimality gaps of~\ref{problem:generalized_dual_decomposition}, DEF, and the primal decomposition algorithm in (SMKP) instances.}
    \label{fig:GRB-Primal-GDD-Conv-SMKP}
\end{figure}

\begin{table}[!h]
    \centering
    \caption{Optimality gaps of~\ref{problem:generalized_dual_decomposition}, DEF, and the primal decomposition algorithm in (SMKP) instances.}
    \label{tab:spComparisonWithGRBSMKP}
    \begin{tabular}{ccw{c}{1.3cm}w{c}{1.3cm}w{c}{1.3cm}w{c}{1.3cm}cc}
        \hline
        \rule{0pt}{2.8ex} & & \multicolumn{4}{c}{\ref{problem:generalized_dual_decomposition} with \# iterations of~\ref{problem:dual_decomposition} warm-start} & \multirow{2}{*}{DEF} & \multirow{2}{*}{Primal} \\
        \cline{3-6}
        \rule{0pt}{1.8ex} & & 0 & 500 & 1000 & 2000 & & \\
        \hline
        \multirow{4}{*}{Ending Gaps (\%)} & 25\% Quantile & 0.47 & 0.43 & 0.44 & 0.51 & 1.20 & 40.68\\
        & Median & 0.81 & 0.70 & 0.72 & 0.75 & 1.28 & 42.94\\
        & Average & 0.82 & 0.73 & 0.74 & 0.81 & 1.29 & 43.06\\
        & 75\% Quantile & 1.09 & 0.92 & 0.98 & 1.03 & 1.38 & 44.49\\
        \hline
        \multirow{4}{*}{DD Gaps (\%)} & 25\% Quantile & 3.91 & 2.76 & 2.85 & 2.44 & NA & NA\\
        & Median & 4.96 & 3.19 & 3.47 & 3.48 & NA & NA\\
        & Average & 4.92 & 3.51 & 3.55 & 3.49 & NA & NA\\
        & 75\% Quantile & 5.51 & 4.29 & 4.06 & 3.85 & NA & NA\\
        \hline
    \end{tabular}
\end{table}

\begin{figure}[htbp]
    \centering
    \begin{subfigure}[t]{0.49\textwidth}
        \centering
        \includegraphics[width=\linewidth]{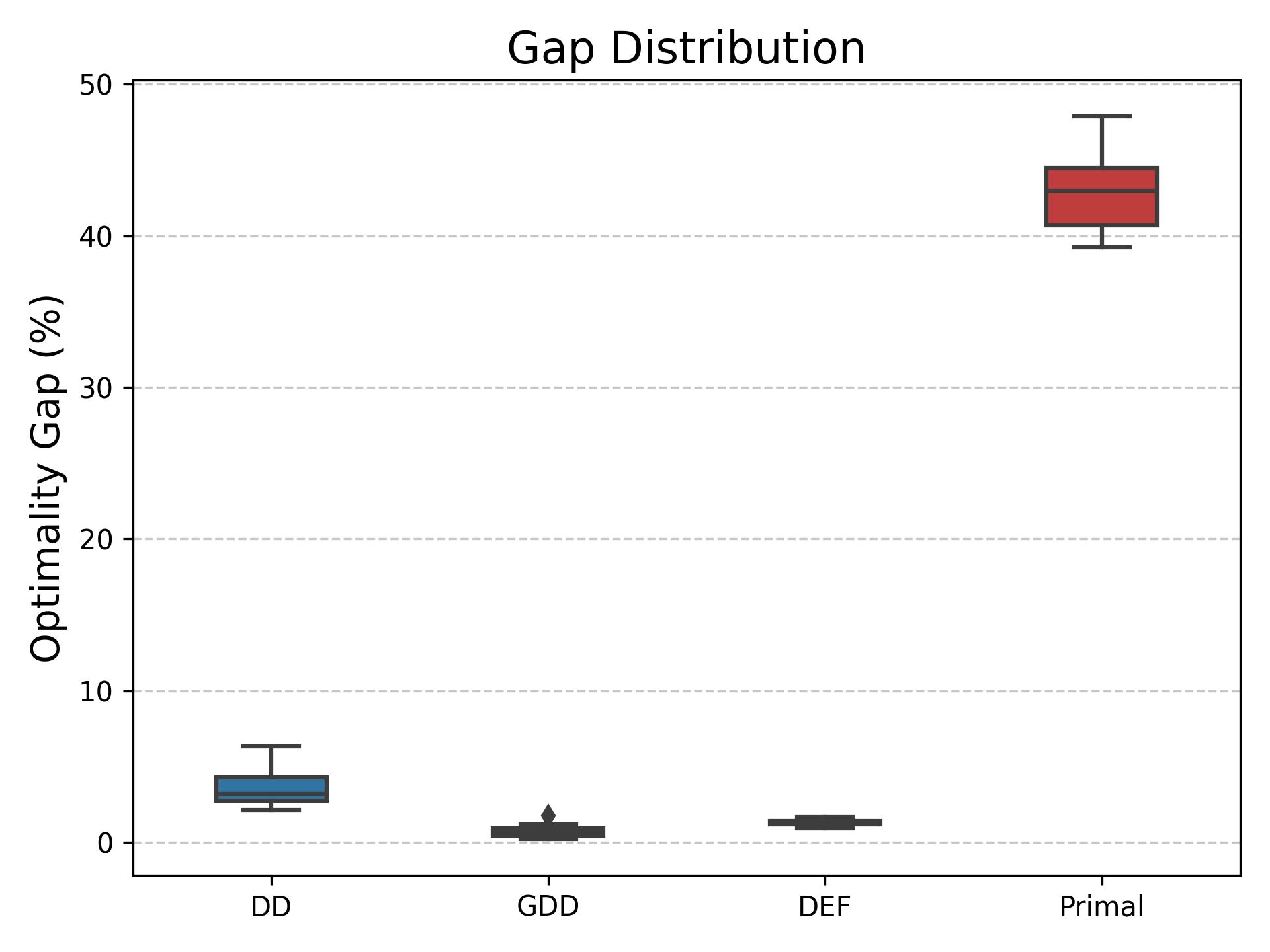}
        \caption{Gaps of~\ref{problem:generalized_dual_decomposition}, DEF, and primal algorithm}
        \label{fig:GDD-DEF-Primal-SMKP}
    \end{subfigure}
    \hfill
    \begin{subfigure}[t]{0.49\textwidth}
        \centering
        \includegraphics[width=\linewidth]{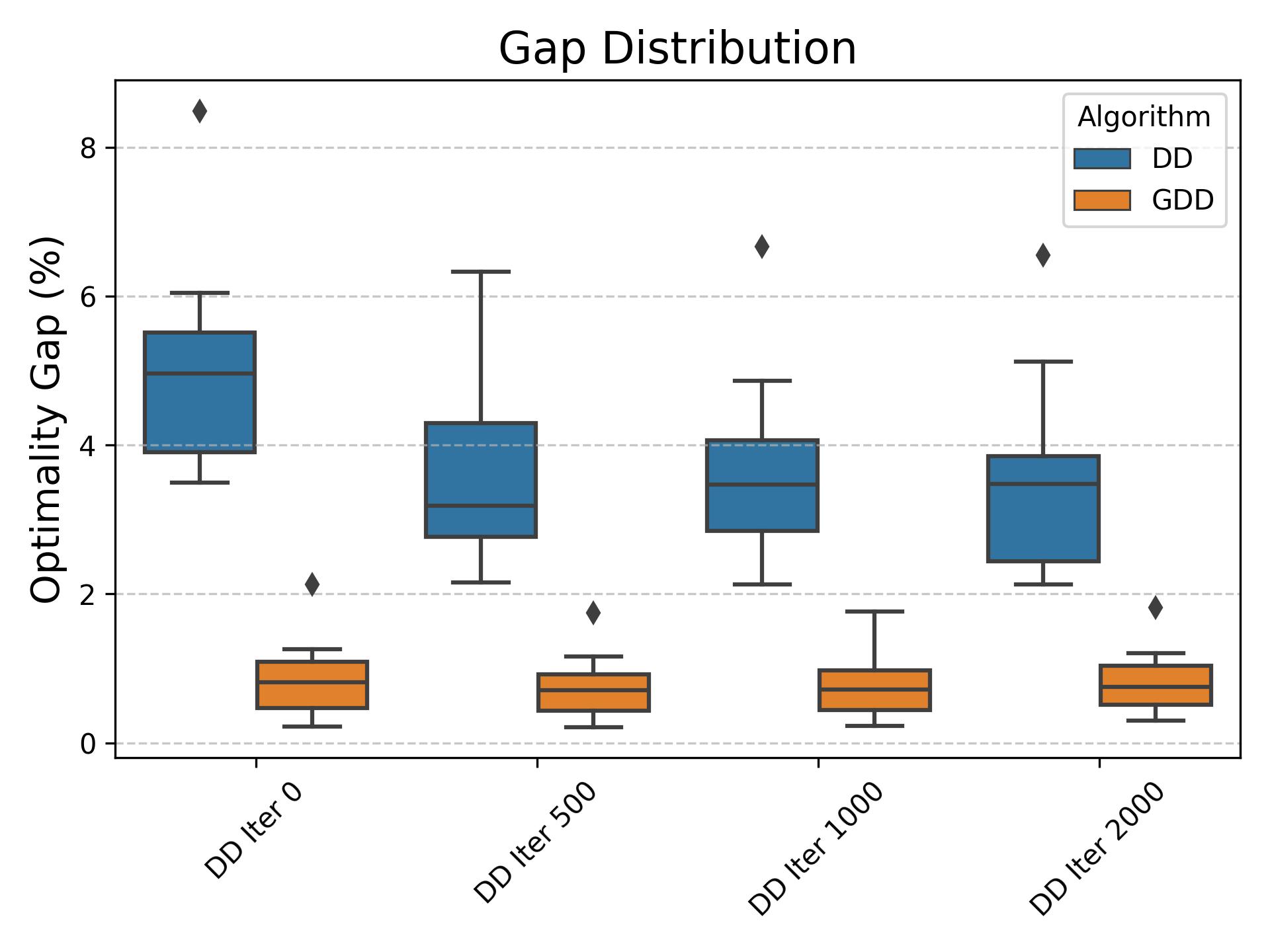}
        \caption{Gaps of~\ref{problem:generalized_dual_decomposition} with varying~\ref{problem:dual_decomposition} warm-start}
        \label{fig:DDVsGDD-only-SMKP}
    \end{subfigure}
    \caption{Comparison of the optimality gaps of~\ref{problem:generalized_dual_decomposition}, DEF, and the primal decomposition algorithm in (SMKP) instances.}
    \label{fig:DDVsGDD-Box-SMKP}
\end{figure}

\paragraph{Comparison with DEF and primal decomposition algorithms.}
Second, we compare the performance of~\ref{problem:generalized_dual_decomposition} with solving the DEF with Gurobi and the Benders-type primal decomposition algorithm as in Section~\ref{subsubsec:DCAP}. We solved 20 random SMKP instances with a time limit of 5 hours for each run. For GDD, we tested four DD warm-start settings, with 0, 500, 1,000, and 2,000 iterations, respectively. We summarize the results in Table~\ref{tab:spComparisonWithGRBSMKP} and Figs.~\ref{fig:GRB-Primal-GDD-Conv-SMKP},~\ref{fig:GDD-DEF-Primal-SMKP}. From Table~\ref{tab:spComparisonWithGRBSMKP} and Fig.~\ref{fig:GDD-DEF-Primal-SMKP}, we observe that~\ref{problem:generalized_dual_decomposition} achieved substantially smaller ending optimality gaps than both DEF and the primal decomposition algorithm. For example, the average optimality gap of GDD ranges from $0.73\%$ to $0.82\%$, while Gurobi achieved an average optimality gap of $1.29\%$ by solving DEF and the primal algorithm achieved an average optimality gap of $43.06\%$. Similar observations can be made from the 25\%, 50\%, and 75\% quantiles of the ending optimality gaps. In addition, from Table~\ref{tab:spComparisonWithGRBSMKP} and Fig.~\ref{fig:GRB-Primal-GDD-Conv-SMKP}, we observe that Gurobi proved an optimality gap of around 1.3\% in 1,000 seconds before running out of memory, and the primal decomposition algorithm shrank the gap more slowly and got stuck at an optimality gap of more than 40\%. In contrast,~\ref{problem:generalized_dual_decomposition} used a much smaller memory than Gurobi did and it was able to continue improving the gap till near-zero. This confirmed the earlier observations in the (DCAP) instances and the effectiveness of~\ref{problem:generalized_dual_decomposition} for solving~\eqref{problem:primal} with mixed-integer tender variables and a large number of scenarios.

In addition, we compare the performance of~\ref{problem:generalized_dual_decomposition} under varying levels of~\ref{problem:dual_decomposition} warm-start and report the findings in Table~\ref{tab:spComparisonWithGRBSMKP} and Fig.~\ref{fig:DDVsGDD-only-SMKP}. We observe that, as in the (DCAP) instances, increasing the number of~\ref{problem:dual_decomposition} warm-start iterations improved the beginning optimality gap of the~\ref{problem:generalized_dual_decomposition} but produced similar ending optimality gaps (see Table~\ref{tab:spComparisonWithGRBSMKP} and the orange boxes in Fig.~\ref{fig:DDVsGDD-only-SMKP}). This confirms our earlier observations that~\ref{problem:generalized_dual_decomposition} is able to robustly close the remaining optimality gap after applying~\ref{problem:dual_decomposition}.

\begin{figure}[htbp]
    \centering
    \begin{subfigure}[t]{0.49\textwidth}
        \centering
        \includegraphics[width=\linewidth]{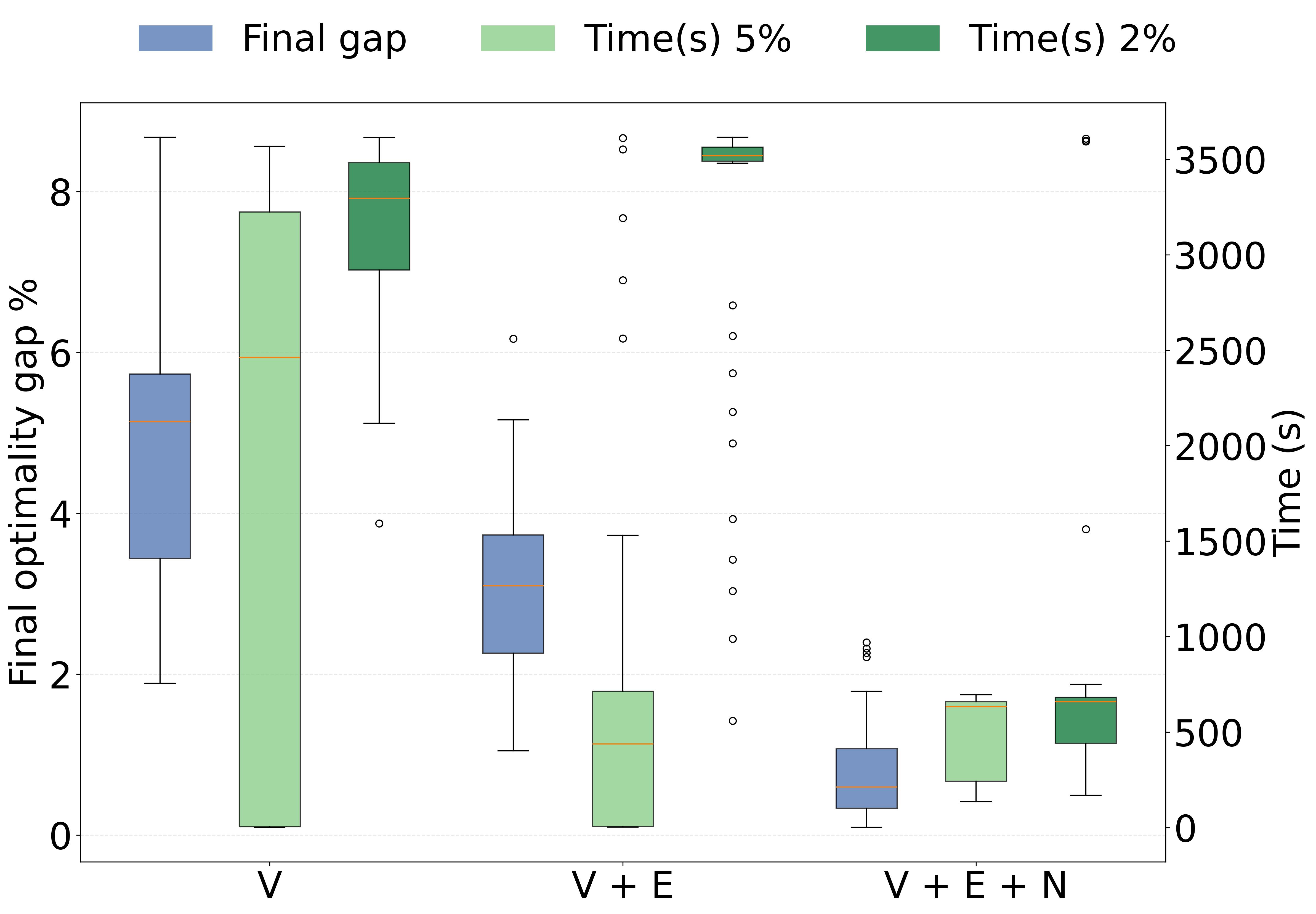}
        \caption{Comparison among improvement strategies.}
        \label{fig:vanillaComparisonSMKP}
    \end{subfigure}
    \hfill
    \begin{subfigure}[t]{0.49\textwidth}
        \centering
        \includegraphics[width=\linewidth]{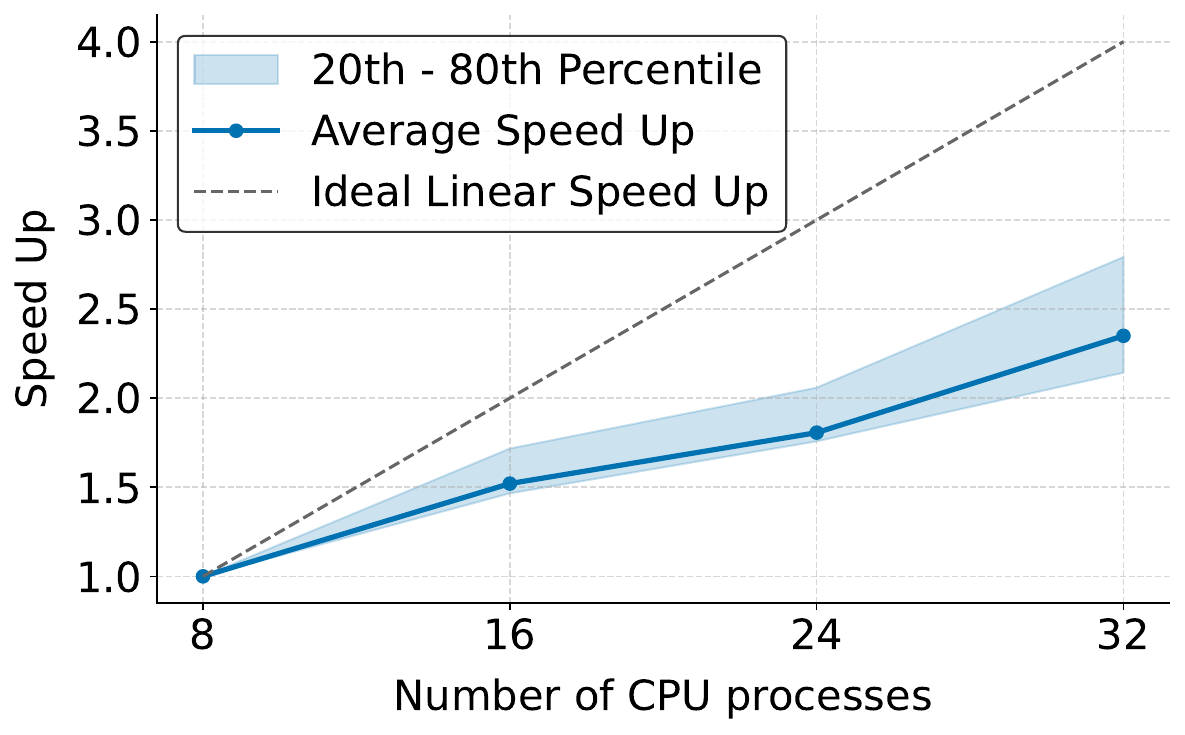}
        \caption{Speedup versus \# of CPU processes.}
        \label{fig:speedup_processesSMKP}
    \end{subfigure}
    \caption{Acceleration of~\ref{problem:generalized_dual_decomposition} through algorithmic improvement strategies and parallel computing in (SMKP) instances. V = Vanilla, E = Enumeration, and N = Nested Voronoi.}
    \label{fig:accelerationComparisonSMKP}
\end{figure}

\paragraph{Acceleration of~\ref{problem:generalized_dual_decomposition} through algorithmic strategies and parallel computing.}
Finally, we demonstrate the acceleration of~\ref{problem:generalized_dual_decomposition} through the algorithmic improvement strategies in Section~\ref{subsec:improvement} and through parallel computing. We report the results in Fig.~\ref{fig:accelerationComparisonSMKP}.

To demonstrate the algorithmic improvement strategies, we compare vanilla GDD (V) with its two variants: GDD with partition enumeration (V+E) and GDD with both partition enumeration and nested Voronoi refinement (V+E+N). We test these three variants without DD warm-start on randomly generated SMKP instances with \(n=8\) and \(S=10\) scenarios under the same parameter ranges as above. From Fig.~\ref{fig:vanillaComparisonSMKP}, we observe that both enumeration (E) and nested Voronoi (N) strategies improved upon the vanilla implementation, and the combination of (E) and (N) yields the best performance overall. For example, this combination significantly reduced the ending optimality gap, as well as the time~\ref{problem:generalized_dual_decomposition} took to reduce the optimality gap to be within 5\% and 2\%. These results confirm the earlier observations in the (DCAP) instances that exploiting the piecewise structure of the regularizers can improve the practical efficacy of~\ref{problem:generalized_dual_decomposition}.

To demonstrate the benefit of parallel computing, we generate another 50 random SMKP instances with \(n=8\), \(S = 32\) scenarios, and a termination criterion of reaching either a 1-hour time limit or an optimality gap of less than 1.5\%. Fig.~\ref{fig:speedup_processesSMKP} reports the speedup (average and the 20th-80th confidence interval) as we increase the number of CPU processes from 8 to 32. We observe that, as in the (DCAP) instances, parallel computing yields a clear and consistent reduction in computation time as the number of processes increases, although the speedup is sublinear in these (SMKP) instances. The deviation from the ideal linear speedup was attributed to a subset of challenging instances, for which~\ref{problem:generalized_dual_decomposition} was not able to achieve the target 1.5\% optimality gap within the 1-hour time limit, regardless of how many CPU processes were assigned.

\section{Conclusion} \label{sec:conclude}
This paper proposed a generalized dual decomposition approach for solving two-stage stochastic programs with mixed-integer decision variables in both stages. This approach admits strong duality by generalizing the linear regularizers in dual decomposition to nonlinear ones. 
Besides, we proposed an algorithmic framework to approximate optimal regularizers with cutting planes and derived sufficient conditions on these cuts for the convergence of the framework to achieve global optimality. In addition, we found that several existing and popular cutting planes satisfy these conditions, making them applicable with the proposed GDD framework. Finally, we conducted extensive numerical experiments to demonstrate the performance of GDD in stochastic programs with both binary and mixed-integer tenders. From the experimental results, we observed that GDD was able to effectively close the optimality gap and compare favorably with the standard DD, solving the DEF with Gurobi, and a primal decomposition algorithm. Finally, the speedup of GDD through parallel computing suggests future work to design other efficient implementations of the GDD framework (e.g., using asynchronous algorithms, other cutting planes, or other parallel computing mechanisms).

%
%
%
\bibliographystyle{informs2014}
\bibliography{refs}
\newpage

\begin{APPENDICES}

\section{Proofs}
\subsection{Proof of Theorem~\ref{prop:strong_duality}} \label{app-prop:strong_duality}
\proof{Proof.}
    \begin{enumerate}
        \item We first show the optimal value of (\ref{problem:primal}) is not less than that of (\ref{problem:generalized_dual_decomposition}). Let $\{g_i\}_{i \in [N]}$ be any feasible solution of problem (\ref{problem:generalized_dual_decomposition}), then
        \begin{align*}
             & \sum_{i=1}^N \min_{\bm{x}^i \in X} \left(f_i(\bm{x}^i) + g_i(\bm{x}^i)\right)\\
            &\leq \min_{\bm{x} \in X} \sum_{i = 1}^N \left(f_i(\bm{x}) + g_i(\bm{x})\right)\\
            &=\min_{\bm{x} \in X}\left(\sum_{i=1}^N f_i(\bm{x}) + \sum_{i=1}^N g_i(\bm{x})\right)\\
            &= \min_{\bm{x} \in X}\sum_{i=1}^N f_i(\bm{x}) = (\ref{problem:primal}).
        \end{align*}
        By taking maximum over $g_i, i \in [N]$, we proved that $v_{\text{GDD}} \leq v^*.$
        \item Then we show the strong duality by proving that the optimal value of (\ref{problem:primal}) is not greater than that of (\ref{problem:generalized_dual_decomposition}). Define $g_i'(\bm{x}):=\frac{\sum_{i'=1}^N f_{i'}(\bm{x})}{N} - f_i(\bm{x}), \forall i \in [N]$. Then we have 
        \begin{align*}
            \sum_{i=1}^N g_i'(\bm{x}) = \sum_{i=1}^N \left(\frac{\sum_{i'=1}^N f_{i'}(\bm{x})}{N} - f_i(\bm{x})\right) = \sum_{i'=1}^N f_{i'}(\bm{x}) - \sum_{i=1}^N f_i(\bm{x}) = 0.
        \end{align*}
        Thus, $\{g_i'\}_{i\in [N]}$ is feasible for (\ref{problem:generalized_dual_decomposition}), and it follows that
        \begin{align*}
            v_{\text{GDD}} &\geq \sum_{i=1}^N \min_{\bm{x}^i \in X} \left(f_i(\bm{x}^i) + g'_i(\bm{x}^i)\right)\\
            &=\sum_{i=1}^N \min_{\bm{x}^i \in X} \frac{\sum_{i' = 1}^N f_{i'}(\bm{x}^i)}{N}\\
            &= \min_{\bm{x} \in X} \sum_{i = 1}^N f_{i}(\bm{x}) = v^*
        \end{align*}
    \end{enumerate}
    Therefore, $v_{\text{GDD}} = v^*$.
\endproof

\subsection{Proof of Proposition \ref{prop:GDDDual}} \label{App:GDDDual}
\proof{Proof. }
We prove the statement by taking the dual of the $\min_{x \in \mathcal{X}}\; \overline{co}(\sum_{i \in [N]} f_i)(\bm{x})$.
\[
\begin{aligned}
    &\min_{x \in \mathcal{X}}\; \overline{co}(\sum_{i \in [N]} f_i)(\bm{x})\\
    =& \begin{aligned}[t]
        \min_{\rho(\cdot) \geq 0}& \; \int_{\bm{x} \in \mathcal{X}} \rho(\bm{x}) \sum_{i \in [N]}f_i(\bm{x})d\bm{x}\\
        \text{s.t.}&\; \int_{\bm{x}} \rho(\bm{x}) d\bm{x} = 1
    \end{aligned}\\
    =& \begin{aligned}[t]
        \min_{\rho(\cdot) \geq 0}& \; \sum_{i \in [N]}\int_{\bm{x} \in \mathcal{X}} \rho(\bm{x}) f_i(\bm{x})d\bm{x}\\
        \text{s.t.}&\; \int_{\bm{x}} \rho(\bm{x}) d\bm{x} = 1
    \end{aligned}\\
    =& \begin{aligned}[t]
        \min_{\rho(\cdot) \geq 0, \rho_i(\cdot) \geq 0, i \in [N]}& \; \sum_{i \in [N]}\int_{\bm{x} \in \mathcal{X}} \rho_i(\bm{x}) f_i(\bm{x})d\bm{x}\\
        \text{s.t.}&\; \int_{\bm{x}} \rho_i(\bm{x}) d\bm{x} = 1, \; \forall i \in [N]\\
        &\; \rho_i(\bm{x}) = \rho(\bm{x}), \quad \text{a.s.}, \; \forall i \in [N]
    \end{aligned}\\
    =& \begin{aligned}[t]
        \max_{g_i(\cdot), \i \in [N]}\min_{\rho(\cdot) \geq 0, \rho_i(\cdot) \geq 0, i \in [N]}& \; \sum_{i \in [N]}\int_{\bm{x} \in \mathcal{X}} \rho_i(\bm{x}) f_i(\bm{x})d\bm{x} + \sum_{i \in [N]} \int_{\bm{x} \in \mathcal{X}} g_i(\bm{x}) (\rho_i(\bm{x}) - \rho(\bm{x})) d\bm{x}\\
        \text{s.t.}&\; \int_{\bm{x}} \rho_i(\bm{x}) d\bm{x} = 1, \; \forall  i \in [N]
    \end{aligned}\\
    =& \begin{aligned}[t]
        \max_{g_i(\cdot), \i \in [N]}\min_{\rho(\cdot) \geq 0, \rho_i(\cdot) \geq 0, i \in [N]}& \; \sum_{i \in [N]}\int_{\bm{x} \in \mathcal{X}} \rho_i(\bm{x}) (f_i(\bm{x}) + g_i(\bm{x}))d\bm{x} - \sum_{i \in [N]} \int_{\bm{x} \in \mathcal{X}} g_i(\bm{x}) \rho(\bm{x}))d\bm{x}\\
        \text{s.t.}&\; \int_{\bm{x}} \rho_i(\bm{x}) d\bm{x} = 1, \; \forall  i \in [N]
    \end{aligned}\\
    =& \begin{aligned}[t]
        \max_{\substack{g_i(\cdot), \i \in [N]\\ \sum_{i \in [N]} g_i(\bm{x}) = 0,\; \text{a.s.}}}\min_{\rho(\cdot) \geq 0, \rho_i(\cdot) \geq 0, i \in [N]}& \; \sum_{i \in [N]}\int_{\bm{x} \in \mathcal{X}} \rho_i(\bm{x}) (f_i(\bm{x}) + g_i(\bm{x}))d\bm{x}\\
        \text{s.t.}&\; \int_{\bm{x}} \rho_i(\bm{x}) d\bm{x} = 1, \; \forall  i \in [N]
    \end{aligned}\\
    =& \begin{aligned}[t]
        \max_{\substack{g_i(\cdot), \i \in [N]\\ \sum_{i \in [N]} g_i(\bm{x}) = 0,\; \text{a.s.}}} \sum_{i \in [N]}\min_{\bm{x}^i \in \mathcal{X}}& \;f_i(\bm{x}^i) + g_i(\bm{x}^i)
    \end{aligned}\\
    =&\; v_{\text{GDD}}
\end{aligned}
\]

\endproof

\subsection{Proof of proposition \ref{prop:recoverOptimal}} \label{App:recoverOptimal}
\proof{Proof. }
    Suppose there exists $j \in [N]$ such that $\bm{x}^*$ is not optimal, i.e., 
    $$f_j(\bm{x}^*) + g^*_j(\bm{x}^*) > \min_{\bm{x} \in X} \left(f_j(\bm{x}) + g^*_j(\bm{x})\right).$$
    
    Since $\bm{x}^*$ is the optimal solution of problem (\ref{problem:primal}), 
    \begin{align*}
        v^* = \sum_{i=1}^N f_i(\bm{x}^*) &= \sum_{i=1}^N f_i(\bm{x}^*) + g^*_i(\bm{x}^*)\\
        &\geq \left(\sum_{i \neq j}^N \min_{\bm{x}^i \in X}\left(f_i(\bm{x}^i) + g^*_i(\bm{x}^i)\right)\right) + f_j(\bm{x}^*) + g^*_j(\bm{x}^*)\\
        &> \sum_{i = 1}^N \min_{\bm{x}^i \in X}\left(f_i(\bm{x}^i) + g^*_i(\bm{x}^i)\right) = v_{\text{GDD}},
    \end{align*}
    which contradicts the strong duality. 
\endproof

\subsection{Proof of Proposition \ref{prop:optimal-regularizer}}
\proof{Proof.}
    By Theorem~\ref{prop:strong_duality}, we have $v_{\operatorname{GDD}} = v^*$. It remains to show that $\{g_i^*\}_{i \in [N]}$ attains the maximum in~\eqref{problem:generalized_dual_decomposition}.

    First, $\{g_i^*\}_{i \in [N]}$ is feasible for~\eqref{problem:generalized_dual_decomposition} since for any $\bm{x} \in X$,
    \begin{align*}
        \sum_{i=1}^N g_i^*(\bm{x}) \ = \ \sum_{i=1}^N \left(\frac{1}{N}\sum_{i'=1}^N f_{i'}(\bm{x}) - f_i(\bm{x})\right) \ = \ \sum_{i'=1}^N f_{i'}(\bm{x}) - \sum_{i=1}^N f_i(\bm{x}) \ = \ 0.
    \end{align*}
    Moreover, for each $i \in [N]$ and any $\bm{x}^i \in X$,
    \[
        f_i(\bm{x}^i) + g_i^*(\bm{x}^i) \ = \ \frac{1}{N}\sum_{i'=1}^N f_{i'}(\bm{x}^i).
    \]
    Therefore,
    \begin{align*}
        \sum_{i=1}^N \min_{\bm{x}^i \in X} \left\{f_i(\bm{x}^i) + g_i^*(\bm{x}^i)\right\} \ &= \ \sum_{i=1}^N \min_{\bm{x}^i \in X} \frac{1}{N}\sum_{i'=1}^N f_{i'}(\bm{x}^i) \\
        \ &= \ N \cdot \frac{1}{N} \min_{\bm{x} \in X} \sum_{i'=1}^N f_{i'}(\bm{x}) \\
        \ &= \ \min_{\bm{x} \in X} \sum_{i=1}^N f_i(\bm{x}) \ = \ v^* \ = \ v_{\operatorname{GDD}},
    \end{align*}
    where the second equality holds because the $N$ inner minimization problems are identical.
\endproof

\subsection{Proof of Theorem~\ref{prop:sufficientConverge}}\label{sed:suffAlg}
\proof{Proof. }
    We denote lower bound in iteration $t$ of the Algorithm \ref{alg:MIPGDD} by $lb_t$. Since it is non-decreasing and upper bounded by $v^*$, the limit $\lim_{t \rightarrow \infty} lb_t$ exists. Then we prove that $\lim_{t \rightarrow \infty} lb_t = v^*$ if the conditions in the statement are satisfied.
    

    If the first condition is satisfied, then for any $g_i(\mathcal{X}^{t}, \bm{x})$ satisfying $\sum_{i \in [N]} g_i(\mathcal{X}^{t}, \bm{x}) \leq 0$ and let $\bm{x}^*$ denote the optimal solution of problem \ref{problem:primal}, we have
    \begin{align*}
        & \lim_{t \rightarrow \infty} lb_t\\
        = & \lim_{t \rightarrow \infty} \max_{t' \leq t} \sum_{i \in [N]} \min_{\bm{x}^i \in X} \left(f_i(\bm{x}^i) + g_i(\mathcal{X}^{t'}, \bm{x}^i)\right)\\
        \leq & \lim_{t \rightarrow \infty} \max_{t' \leq t} \sum_{i \in [N]} f_i(\bm{x}^*) + g_i(\mathcal{X}^{t'}, \bm{x}^*)\\
        \leq & \lim_{t \rightarrow \infty} \sum_{i \in [N]} f_i(\bm{x}^*)\\
        =& v^*.
    \end{align*}

    Suppose the second condition is also satisfied. Then without loss of generality, we can assume that the subsequences for each scenario share the same index. Otherwise, we can select the converging subsequences for each scenario in a nested way, i.e., selecting the converging subsequences from the iterations of the converging subsequences of the previous scenario. It follows that 
    \begin{align*}
        & \lim_{t \rightarrow \infty} lb_t\\
        \geq &\liminf_{r \rightarrow \infty} \sum_{i=1}^N \left(f_i(\bm{x}_{(r)}^i) + g_i(\mathcal{X}^{t(r)}, \bm{x}_{(r)}^i)\right)\\
        \geq & \sum_{i = 1}^N \liminf_{r \rightarrow \infty} \left(f_i(\bm{x}_{(r)}^i) + g_i(\mathcal{X}^{t(r)}, \bm{x}_{(r)}^i)\right)\\
        \geq& \sum_{i=1}^N \liminf_{r \rightarrow \infty} \left(f_i(\bm{x}_{(r)}^i) + g_i^*(\bm{x}_{(r)}^i)\right)\\
        \geq & \sum_{i=1}^N \liminf_{r \rightarrow \infty} \min_{\bm{x}^i \in X} \left(f_i(\bm{x}^i) + g_i^*(\bm{x}^i)\right)\\
        =& N \min_{\bm{x}^i \in X}\frac{\sum_{i \in [N]} f_i(\bm{x}^i)}{N}\\
        =& v^*.
    \end{align*}

    Thus, we have $\lim_{t \rightarrow \infty} lb_t = v^*$.
\endproof

\subsection{Proof of Proposition \ref{prop:gi-milp}} \label{sec:gi-milp}
\proof{Proof. }
It is sufficient to show that $\operatorname{lsc} \left(g_i\right)(\mathcal{X}^t, \bm{x}):=\liminf g_i(\mathcal{X}^t, \bm{x}) = \min_{k: \bm{x} \in \operatorname{cl}(\mathcal{P}_k^t)} g_i^*(\bm{x}_{(k)})$. 

    We start by  
    \begin{align*}
        &\liminf g_i(\mathcal{X}^t, \bm{x})\\
        =& \lim_{\epsilon \rightarrow 0} \left(\inf_{\bm{x}' \in X\cap B(\bm{x}, \epsilon)} g_i(\mathcal{X}^t, \bm{x}')\right)\\
        =& \lim_{\epsilon \rightarrow 0} \left(\min_{k: \bm{x} \in \operatorname{cl}(\mathcal{P}_k^t)} \inf_{\bm{x}' \in \mathcal{P}_k^t \cap B(\bm{x}, \epsilon)} g_i(\mathcal{X}^t, \bm{x}')\right)\\
        =& \lim_{\epsilon \rightarrow 0} \left(\min_{k: \bm{x} \in \operatorname{cl}(\mathcal{P}_k^t)} \inf_{\bm{x}' \in \mathcal{P}_k^t \cap B(\bm{x}, \epsilon)} g^*_i(\bm{x}_{(k)})\right)\\
        =& \min_{k: \bm{x} \in \operatorname{cl}(\mathcal{P}_k^t)} g_i^*(\bm{x}_{(k)}).
    \end{align*}

Therefore the statement in the proposition holds. 

Besides, we have 
\begin{align*}
    &\inf_{\bm{x} \in X} \left\{f_i(\bm{x}) + \liminf g_i(\mathcal{X}^t, \bm{x})\right\}\\
    =& \inf_{\bm{x} \in X} \left\{f_i(\bm{x}) + \lim_{\epsilon \rightarrow 0} \left(\inf_{\bm{x}' \in X \cap B(\bm{x}, \epsilon)}g_i(\mathcal{X}^t, \bm{x}')\right)\right\}\\
    =& \inf_{\bm{x} \in X} \lim_{\epsilon \rightarrow 0} \left(f_i(\bm{x}) + \inf_{\bm{x}' \in X \cap B(\bm{x}, \epsilon)}g_i(\mathcal{X}^t, \bm{x}')\right)\\
    =& \inf_{\bm{x} \in X} \lim_{\epsilon \rightarrow 0} \left(\inf_{\bm{x}' \in X \cap B(\bm{x}, \epsilon)}f_i(\bm{x}') + g_i(\mathcal{X}^t, \bm{x}')\right)\\
    =& \inf_{\bm{x} \in X} \liminf \left(f_i(\bm{x}) + g_i(\mathcal{X}^t, \bm{x})\right)\\
    =& \inf_{\bm{x} \in X} f_i(\bm{x}) + g_i(\mathcal{X}^t, \bm{x}),
\end{align*}
where the third equation from the end holds because $f_i$ is continuous by assumption and in the last equation we use the fact that $\inf_{\bm{x} \in X} h(\bm{x}) = \inf_{\bm{x} \in X} \liminf h(\bm{x})$ for any function $h$. For completeness, we show the proof below:
\begin{enumerate}
    \item $\inf_{\bm{x} \in X} \liminf h(\bm{x}) = \inf_{\bm{x} \in X} \lim_{\epsilon \rightarrow 0} \left(\inf_{\bm{x}' \in X \cap B(\bm{x}, \epsilon)}h(\bm{x}')\right) \geq \inf_{\bm{x} \in X} \lim_{\epsilon \rightarrow 0} \left(\inf_{\bm{x}' \in X}h(\bm{x}')\right) = \inf_{\bm{x}\in X} h(\bm{x})$.
    \item $\inf_{\bm{x} \in X} \liminf h(\bm{x}) = \inf_{\bm{x}\in X} \lim_{\epsilon \rightarrow 0} \left(\inf_{\bm{x}' \in X \cap B(\bm{x}, \epsilon)}h(\bm{x}')\right) \leq \inf_{\bm{x} \in X} \lim_{\epsilon \rightarrow 0} h(\bm{x}) = \inf_{\bm{x} \in X} h(\bm{x})$.
\end{enumerate}
Thus $\inf_{\bm{x} \in X} h(\bm{x}) = \inf_{\bm{x} \in X} \liminf h(\bm{x})$.
\endproof

\subsection{Proof of Proposition \ref{prop:simpleApprox}} \label{sec:simpleApprox}
\proof{Proof. }
Since $X$ is compact by Assumption \ref{assumption:compact}, we define the radius of $X$ as $r:= \max_{\bm{x}, \bm{y} \in X} \|\bm{x} - \bm{y}\|_2$ and there exists $l, u \in \mathbb{R}$ such that $X \subseteq \prod_{d=1}^{n_1} [l, u]$. Since $g_i^*, \;\forall i \in [N]$ is Lipschitz continuous by Assumption \ref{assump:CCR} and Lemma \ref{lemma:LipschitzCont}, we denote the maximum Lipschitz constant of $g_i^*$ over $i \in [N]$ as $\mathcal{L}$.

Let $\epsilon' > 0$ be an arbitrary scalar, we consider the point set $X' = \left\{\bm{x}: x_d = l + z_d \frac{\epsilon'}{n_1 \mathcal{L}}, z_d \in \mathbb{Z}^+, \forall d \in [n_1]\right\} \cap X$ and create a Voronoi partition $\left\{\mathcal{P}_k\right\}_{k=1}^{K}$ with the points in $X'$ as the centers (denoted by $\bm{x}_{(k)}$). It follows that $\max_{\bm{x}, \bm{y} \in \mathcal{P}_k} \|\bm{x} - \bm{y}\|_2 \leq \frac{\epsilon'}{\sqrt{n_1}\mathcal{L}} \leq \frac{\epsilon'}{\mathcal{L}}$. Thus, $\max_{\bm{x} \in \mathcal{P}_k} |g_i^*(\bm{x}) - g_i^*(\bm{x}_{(k)})| \leq \mathcal{L}\max_{\bm{x} \in \mathcal{P}_k}\|\bm{x} - \bm{x}_{(k)}\|_2 \leq \mathcal{L} \frac{\epsilon'}{\mathcal{L}} = \epsilon'$.

Let $g_i(\bm{x}) := \sum_{k \in [K]} g_i^*(\bm{x}_{(k)}) \mathbbm{1}\left({\bm{x} \in \mathcal{P}_k}\right)$, then we have $\min_{\bm{x}^i \in X} \left\{f_i(\bm{x}^i) + g_i(\bm{x}^i)\right\} \geq \min_{\bm{x}^i \in X} \left\{f_i(\bm{x}^i) + g_i^*(\bm{x}^i) - \epsilon'\right\}$. The statement follows if we let $\epsilon' = \frac{\epsilon}{N}$.
\endproof

\subsection{Proof of Theorem~\ref{prop:MIPConvergence}} \label{sec:suffSimple}
\proof{Proof. }
    We show that the first condition in Theorem~\ref{prop:sufficientConverge} is satisfied as below: 
    \begin{align*}
        &\sum_{i \in [N]} g_i(\mathcal{X}^{t}, \bm{x}) \\
        =& \sum_{i \in [N]} \sum_{k=1}^{K_t} g_i^*(\bm{x}_{(k)}) \mathbbm{1}\left({\bm{x} \in \mathcal{P}_k^{t}}\right)\\
        =& \sum_{k=1}^{K_t} \sum_{i \in [N]} g_i^*(\bm{x}_{(k)}) \mathbbm{1}\left({\bm{x} \in \mathcal{P}_k^{t}}\right)\\
        =& 0.
    \end{align*}

    For each scenario $i$, for any subsequence of the sequence of optimal subproblem solutions generated in Algorithm \ref{alg:MIPGDD}, there is a converging subsequence by Assumption \ref{assumption:compact}. Denote the sequence by $\left\{\bm{x}^i_{(r)}\right\}_{r=1}^\infty$ and let $\tilde{\bm{x}}^{i*}$ be the limit, i.e., 
    $$\lim_{r \rightarrow \infty} \bm{x}^i_{(r)} = \tilde{\bm{x}}^{i*} \in X$$

    For each element $\bm{x}^i_{(r)}$, we denote its iteration index in Algorithm \ref{alg:MIPGDD} as $t(r)$. Then, we denote the element closest to $\bm{x}^i_{(r)}$ in $\mathcal{X}^{t(r)}$ as $c(\bm{x}^i_{(r)})$. We can claim that $\lim_{r \rightarrow \infty} c(\bm{x}^i_{(r)}) = \tilde{\bm{x}}^{i*}$: For any given $\epsilon > 0$ there exists $R' \in \mathbb{Z}^+$, such that for all $r \geq R'$, $\|\bm{x}^i_{(r)} - \tilde{\bm{x}}^{i*}\| \leq \frac{\epsilon}{3}$. Thus, $\|c(\bm{x}^i_{(r)}) -\tilde{\bm{x}}^{i*}\| \leq \|c(\bm{x}^i_{(r)}) - \bm{x}^i_{(r)}\| + \|\bm{x}^i_{(r)} - \tilde{\bm{x}}^{i*}\| \leq \|\bm{x}^i_{(R')} - \bm{x}^i_{(r)}\| + \|\bm{x}^i_{(r)} - \tilde{\bm{x}}^{i*}\| \leq \frac{2\epsilon}{3} + \frac{\epsilon}{3} = \epsilon$.  
    
    It follows that 
    \begin{align}
        &\liminf_{r \rightarrow \infty} g_i(\mathcal{X}^{t(r)}, \bm{x}^i_{(r)}) \nonumber\\
        =& \liminf_{r \rightarrow \infty} g_i^*(c(\bm{x}^i_{(r)})) \\
        =& g_i^*(\tilde{\bm{x}}^{i*})
        \label{eq:limitEquiv}
    \end{align}
    where the first equation holds by definition of $g_i$ and Voronoi partition, and the last equation holds because both $f_i$ and $g_i^*$ is continuous by Assumption \ref{assump:CCR} and Lemma \ref{lemma:LipschitzCont}. Then, by the continuity of $f_i,\; i \in [N]$, the second condition in Theorem~\ref{prop:sufficientConverge} is also satisfied.
\endproof

\subsection{Proof of Theorem~\ref{prop:LipschitzContCut}} \label{sec:LipschitzContCut}
\proof {Proof. }
It is sufficient to show that $g_i(\mathcal{X}^{t}, \bm{x})$ satisfies the conditions in Theorem~\ref{prop:sufficientConverge}. 
\begin{enumerate}
    \item We have 
    \begin{align*}
        &\sum_{i \in [N]} g_i(\mathcal{X}^{t}, \bm{x})\\
        =&\sum_{i \in [N]} \max_{k \in [K_t]} \mathcal{L}_i(\bm{x}_{(k)}, \bm{x})\\
        \leq & \sum_{i \in [N]} g_i^*(\bm{x})\\
        = & 0,
    \end{align*}
    where the first equation holds by definition of $g_i,\; i \in [N]$; The first inequality holds by the validity of cuts and the last equation holds by definition of $g_i^*, \; i\in [N]$.
    \item Since $\mathcal{L}_i$ is $L_i-$Lipschitz continuous, we can claim that $g_i(\mathcal{X}^{t}, \bm{x})$ is also $L_i-$Lipschitz continuous and we define $L:=\max_{i \in [N]} L_i$. By Assumption \ref{assumption:compact}, for each scenario $i$ and any subsequence of optimal subproblem solutions generated in Algorithm \ref{alg:MIPGDD}, there is a converging subsequence $\left\{\bm{x}_{(r)}^i\right\}_{r=1}^\infty$ with limit $\tilde{\bm{x}}^{i*}$, we use $t(r)$ to denote the iteration index of the element $\bm{x}_{(r)}^i$. It follows that: 
    \begin{align*}
        &\liminf_{r \rightarrow \infty }g_i(\mathcal{X}^{t(r)}, \bm{x}_{(r)}^i)\\
        \geq &\liminf_{r \rightarrow \infty } g_i(\mathcal{X}^{t(r)}, \bm{x}^i_{(r-1)}) - L\|\bm{x}_{(r)}^i - \bm{x}_{(r - 1)}^i\|\\
        = & \liminf_{r \rightarrow \infty } g_i^*(\bm{x}^i_{(r-1)})\\
        =& g_i^*(\tilde{\bm{x}}^{i*}),
    \end{align*}
    where the last equation holds by the continuity of $g_i^*$ from Assumption \ref{assump:CCR} and Lemma \ref{lemma:LipschitzCont}. Furthermore, utilizing the continuity of $f_i$, we have $\liminf_{r \rightarrow \infty } f_i(\bm{x}^i_{(r)}) + g_i(\mathcal{X}^{t(r)}, \bm{x}_{(r)}^i) \geq \liminf_{r \rightarrow \infty } f_i(\bm{x}^i_{(r)}) + g_i^*(\bm{x}^i_{(r)})$.
    Therefore both conditions in Theorem~\ref{prop:sufficientConverge} are satisfied.
\end{enumerate}
\endproof

\subsection{Lipschitz Continuity of $f_i$} \label{sec:LipschitzCont}
\begin{lemma}\label{lemma:LipschitzCont}
    $f_i$ is Lipschitz continuous under Assumption \ref{assump:CCR}.
\end{lemma}
\proof{Proof. } The statement holds for the first case as shown in \cite{ahmed2022stochastic}. Thus, it remains to show the statement also holds for the second case. Define
\begin{align*}
    V(\bm{x}, \bm{y}):= \min_{\bm{z}}\; &\ell(\bm{y}, \bm{z})\\
    \text{s.t. }& L(\bm{x}, \bm{y}, \bm{z}) \leq 0\\
                & \bm{z} \in Z \cap \mathbb{R}^{n_2 - p_2}
\end{align*}
 Then it is sufficient to show that $V(\bm{x}, \bm{y})$ is Lipschitz continuous in $\bm{x}$ for any $\bm{y}$ since we can rewrite 
\begin{align*}
    f_i(\bm{x}) = \min_{\bm{y} \in Y \cap \mathbb{Z}^{p_2}} V(\bm{x}, \bm{y}).
\end{align*}


By the second case of Assumption \ref{assump:CCR} for any $\bm{x} \in \mathbb{R}^{n_1}, \bm{y} \in Y \cap \mathbb{Z}^{p_2}$, $L(\bm{x}, \bm{y}, \bm{z})$ is convex in $\bm{z}$ and there exists a $\bm{z}'$ such that $L(\bm{x}, \bm{y}, \bm{z}') < -\delta$. Let $R$ be the diameter of $Z$. By the convex extension of the Hoffman's theorem in \cite{robinson1975application}, we have 
$$d\left(\bm{z}^*, \left\{L(\bm{x}, \bm{y}, \bm{z}) \leq 0, \bm{z} \in Z \cap \mathbb{R}^{n_2 - p_2} \right\}\right) \leq \delta^{-1} R \|L(\bm{x}, \bm{y}, \bm{z}^*)^+\|$$
where $L(\bm{x}, \bm{y}, \bm{z}^*)^+$ is the projection of $L(\bm{x}, \bm{y}, \bm{z}^*)$ to $\mathbb{R}^{m_2}_+$, and $\bm{z}^*$ is an arbitrary vector in $Z \cap \mathbb{R}^{n_2 - p_2}$.

Suppose $\bm{z}^*$ and $\bm{x}^*$ satisfy $L(\bm{x}^*, \bm{y}, \bm{z}^*) \leq 0$. Since $L$ is Lipschitz continuous (denote its Lipschitz constant as $\mathcal{L}$), we have 
$$L(\bm{x}, \bm{y}, \bm{z}^*) \leq L(\bm{x}^*, \bm{y}, \bm{z}^*) + \mathcal{L}\|\bm{x} - \bm{x}^*\|\bm{e} \leq \mathcal{L}\|\bm{x} - \bm{x}^*\|\bm{e}$$
It follows that 
\begin{align*}
    &d\left(\bm{z}^*, \left\{L(\bm{x}, \bm{y}, \bm{z}) \leq 0, \bm{z} \in Z \cap \mathbb{R}^{n_2 - p_2} \right\}\right) \\
    \leq& \delta^{-1} R \|L(\bm{x}, \bm{y}, \bm{z}^*)^+\|\\
    \leq & \delta^{-1}R \|\bm{e}\|\mathcal{L}\|\bm{x} - \bm{x}^*\|
\end{align*}

Therefore, we have 
$$\left\{\bm{z}: L(\bm{x}^*, \bm{y}, \bm{z}) \leq 0\right\} \subseteq \left\{\bm{z}: L(\bm{x}, \bm{y}, \bm{z}) \leq 0\right\} + \delta^{-1}R \|\bm{e}\|\mathcal{L}\|\bm{x} - \bm{x}^*\|\Delta $$
where $\Delta:=\left\{\epsilon \in \mathbb{R}^{n_2 - p_2}:\|\epsilon\| \leq 1\right\}$.

Next, consider $V(\bm{x}, \bm{y})$ and $V(\bm{x}^*, \bm{y})$ and let $\bm{z}$ and $\bm{z}^*$ be the feasible solutions respectively. It follows that 
$$V(\bm{x}^*, \bm{y}) \leq \ell(\bm{y}, \bm{z}^*) \leq \ell(\bm{y}, \bm{z}) + \mathcal{L}_\ell \|\bm{z} - \bm{z}^*\| \leq \ell(\bm{y}, \bm{z}) + \mathcal{L}_\ell\delta^{-1}R\|\bm{e}\|\mathcal{L}\|\bm{x} - \bm{x}^*\|$$
where $\mathcal{L}_\ell$ is the Lipschitz constant of $\ell$. Taking minimum over $\bm{z}$ on both sides gives
$$V(\bm{x}^*, \bm{y}) \leq V(\bm{x}, \bm{y}) + \mathcal{L}_\ell\delta^{-1}R\|\bm{e}\|\mathcal{L} \|\bm{x} - \bm{x}^*\|$$
i.e., $V(\bm{x}, \bm{y})$ is Lipschitz continuous in $\bm{x}$. And it follows that $f_i(\bm{x})$ is Lipschitz continuous.
\endproof

\subsection{Proof of Proposition \ref{prop:polyStudy}} \label{sec:polyStudy}
\proof{Proof. }
By definition, we have the convex hull of the epigraph of $g$ is 
\begin{align*}
    \operatorname{conv}\left(\operatorname{epi}\left(g\right)\right) :=& \left\{\left(\bm{x}, \theta\right) \in \mathbb{R}^{n_1} \times \mathbb{R}\; \middle|\; 
    \begin{aligned}
        \sum_{k=1}^K \lambda_k &= 1\\
        \sum_{k=1}^K \lambda_k \bm{x}_k &= \bm{x},\\
        \sum_{k=1}^K \lambda_k \theta_k &= \theta \\
        \theta_k &\geq g_k, \forall k \in [K]\\
        \bm{A}_k \bm{x}_k &\leq \bm{b}_k\\
        \lambda_k &\geq 0, \forall k \in [K]\\
    \end{aligned}\right\}
\end{align*}
It follows that 
\begin{align*}
    \overline{\operatorname{co}}\left(g\right)(\bm{x}) &= \begin{aligned}[t]
                                            \min_{\theta}\; & \theta\\
                                            \text{s.t.}\; & (\bm{x}, \theta) \in \operatorname{conv}(\operatorname{epi}(g))
                                        \end{aligned}\\
                                     &= \begin{aligned}[t]
                                         \min\; & \theta\\
                                         \text{s.t. }& \sum_{k=1}^K \lambda_k = 1\\
                                         & \lambda_k \geq 0, \forall k \in [K]\\
                                         & \sum_{k=1}^K \lambda_k \bm{x}_k = \bm{x}\\
                                         & \sum_{k=1}^K \lambda_k \theta_k  = \theta\\
                                         & \theta_k \geq g_k, \forall k \in [K]\\
                                         & \bm{A}_k \bm{x}_k \leq \bm{b}_k, \forall k \in [K]
                                     \end{aligned}\\
                                     &= \begin{aligned}[t]
                                         \min\; & \theta\\
                                         \text{s.t.}\; & \sum_{k=1}^K \lambda_k = 1\\
                                         & \lambda_k \geq 0, \forall k \in [K]\\
                                         & \sum_{k=1}^K \bm{y}_k = \bm{x}\\
                                         & \sum_{k=1}^K \lambda_k \theta_k  = \theta\\
                                         & \theta_k \geq g_k, \forall k \in [K]\\
                                         & \bm{A}_k \bm{y}_k \leq \lambda_k \bm{b}_k, \forall k \in [K]
                                     \end{aligned}\\
                                     &= \begin{aligned}[t]
                                         \max_{\alpha, \bm{\beta}, \bm{\eta}_k \geq 0, \forall k \in [K]} \; & \alpha + \bm{\beta}^\top \bm{x}\\
                                         \text{s.t.}\; & \alpha + \bm{b}_k^\top \bm{\eta}_k \leq g_k, k \in [K]\\
                                         & \bm{A}_k^\top \bm{\eta}_k = \bm{\beta}, k \in [K]
                                     \end{aligned}
\end{align*}
The LP attains its optimum at an extreme point of $E$ in $(\alpha, \bm{\beta})$-space. Hence
\[
    \overline{\operatorname{co}}(g)(\bm{x}) \ = \ \max_{(\alpha, \bm{\beta}) \in \operatorname{ext}(E)} \alpha + \bm{\beta}^\top \bm{x},
\]
which gives the epigraph representation in the proposition.
\endproof

\subsection{Proof of Proposition \ref{prop:convexHullProj}} \label{App:convexHullProj}
\proof{Proof. }
    Define the set of index $I:= \left\{j \;\middle|\; \bm{x}^{(j)} \in A\right\} \subseteq [2^{n_1}]$. 
    Then by definition,
    \begin{align*}
        \text{conv}\left(\text{epi } I_A\right):= \left\{\left(\bm{x}, \theta\right)\;
        \middle|\;\begin{aligned}
            \sum_{j \in I} \lambda_j &\leq \theta\\
            \sum_{j \in I} \lambda_j \bm{x}^{(j)} + \sum_{j \notin I} \lambda_j \bm{x}^{(j)} &= \bm{x}\\
            \sum_{j \in I}\lambda_j + \sum_{j \notin I} \lambda_j &= 1\\
            \lambda_j &\geq 0, j \in [2^{n_1}]
        \end{aligned}
        \right\}
    \end{align*}

    Thus, 
    \begin{align*}
        &\operatorname{conv}\left(\operatorname{epi} I_A\right) \cap \left\{\left(\bm{x}, 0\right) \;\middle |\; \forall \bm{x} \in [0, 1]^{n_1}\right\}\\
        =& \left\{\left(\bm{x}, 0\right) \; \middle| \; 
        \begin{aligned}
            \sum_{j \in I} \lambda_j &= 0\\
            \sum_{j \in I} \lambda_j \bm{x}^{(j)} + \sum_{j \notin I} \lambda_j \bm{x}^{(j)} &= \bm{x}\\
            \sum_{j\in I}\lambda_j + \sum_{j \notin I}\lambda_j &= 1\\
            \lambda_j &\geq 0, \forall j\in [2^{n_1}]\\
            \bm{x} &\in [0, 1]^{n_1}
        \end{aligned}
            \right\}\\
        =& \left\{\left(\bm{x}, 0\right) \; \middle| \;
        \begin{aligned}
            \sum_{j \notin I} \lambda_j \bm{x}^{(j)} &= \bm{x}\\
            \sum_{j \notin I}\lambda_j &= 1\\
            \lambda_j &\geq 0, \forall j\notin I\\
            \bm{x}&\in [0, 1]^{n_1}
        \end{aligned}
            \right\}\\
        =& \left\{\left(\bm{x}, 0\right) \;\middle| \;
        \bm{x} \in \text{conv}\left(A^c\right)
            \right\}
    \end{align*}
    It follows that
     \[
        \operatorname{conv}\left(A^c\right) = \operatorname{Proj}_{\bm{x}}\left(\text{conv}\left(\operatorname{epi} I_A\right) \cap \left\{\left(\bm{x}, 0\right): \bm{x} \in [0, 1]^{n_1}\right\}\right)
        \]
    Finally, we connect the above to the statement of the proposition. On the one hand, since $A^c$ is a finite set of binary vectors, the set $\{(\theta, \bm{x}) : \theta \geq 0, \bm{x} \in A^c\}$ has convex hull
\[
    \operatorname{conv}\big\{(\theta, \bm{x}) : \theta \geq 0,\ \bm{x} \in A^c\big\} \ = \ [0, \infty) \times \operatorname{conv}(A^c).
\]
On the other hand, $(0, \bm{x}) \in \overline{\operatorname{co}}(I_A)$ (i.e., $(0, \bm{x}) \in \operatorname{epi}(\overline{\operatorname{co}}(I_A))$) is equivalent to $\overline{\operatorname{co}}(I_A)(\bm{x}) \leq 0$. Since $I_A \geq 0$ implies $\overline{\operatorname{co}}(I_A) \geq 0$, this is in turn equivalent to $\overline{\operatorname{co}}(I_A)(\bm{x}) = 0$, which by the preceding chain of equalities holds if and only if $\bm{x} \in \operatorname{conv}(A^c)$. Therefore,
\[
    \big\{(\theta, \bm{x}) : \theta \geq 0,\ (0, \bm{x}) \in \overline{\operatorname{co}}(I_A)\big\} \ = \ [0, \infty) \times \operatorname{conv}(A^c),
\]
and the two sides of the proposition coincide.
\endproof

\subsection{Proof of Proposition \ref{prop:convexEnvForChain}} \label{App:convexEnvForChain}
\proof{Proof. }
    We start from an equivalent definition of convex envelope by applying the conjugate twice on $I_A$ (see Theorem 12.2 in \cite{rockafellar1997convex}):
    \begin{subequations}
    \begin{align}
        \overline{\operatorname{co}}
        \left(I_A\right)(\bm{x}) = \max_{\bm{\alpha} \in \mathbb{R}^{n_1}, \beta \in \mathbb{R}} & \bm{\alpha}^\top\bm{x} + \beta\\
        \text{s.t. } & \bm{\alpha}^\top \bar{\bm{x}} + \beta \leq 1, \bar{\bm{x}} \in A \label{eq:conenvCons1}\\
        & \bm{\alpha}^\top \bar{\bm{x}} + \beta \leq 0, \bar{\bm{x}} \notin A \label{eq:conenvCons2}
    \end{align}
    \end{subequations}
    Since the optimal solutions of linear programming problems always contain the extreme point, it is sufficient to find all extreme points of $E := \left\{\left(\bm{\alpha}, \beta\right)\mid (\ref{eq:conenvCons1}), (\ref{eq:conenvCons2})\right\}$ to describe the convex envelope of $I_A(\bm{x})$.

    Suppose $\left(\bm{\alpha}^*, \beta^*\right)$ is an extreme point of $E$, then $\left(\bm{\alpha}^*, \beta^*\right)$ is the 
    intersection of $n_1 + 1$ linearly independent hyperplanes in $E$.

    Let $H_{A} := \left\{\bm{\alpha}^\top\bar{\bm{x}} + \beta = 1| \bar{\bm{x}} \in A\right\}$, 
    $H_{A^c} := \left\{\bm{\alpha}^\top\bar{\bm{x}} + \beta = 0| \bar{\bm{x}} \notin A\right\}$, there are 
    three cases:
    \begin{enumerate}
        \item Case 1: All $n_1 + 1$ hyperplanes come from $H_{A^c}$. 

        In this case, the solution $\left(\bm{\alpha}^*, \beta^*\right)$ of the system is $(\bm{0}, 0)$ 
        because the right-hand side is $\bm{0}$, and this gives the cut
        $\theta \geq 0$.
        
        The case exists because one can choose hyperplanes with $\bar{\bm{x}} = \bm{e}_2, \ldots, \bm{e}_{n_1}$, 
        $\bar{\bm{x}} = \bm{e}_2 + \ldots + \bm{e}_{n_1}$ and $\bar{\bm{x}} = \bm{e}_1 + \bm{e}_{n_1}$.
        Those $n_1 + 1$ hyperplanes are linearly independent and those points are not in $A$ since all $A$ elements have their 1-entries in the beginning. Then the result follows.

        \item Case 2: At least one hyperplane comes from $H_{A}$ and suppose $\bm{\alpha}^\top \bar{\bm{x}}^{*} + \beta = 1$ is 
        one of the $n_1 + 1$ hyperplanes, where $\bar{\bm{x}}^* \in A,\; \bar{\bm{x}}^* + \bm{e}_{|I(\bar{\bm{x}}^*)| + 1} \in A$.
        
        Let $n:=|I(\bar{\bm{x}}^*)|$, consider the following subsystem
        \[
        \begin{aligned}
                E' :&= \begin{cases}
                    \sum_{j=1}^n \alpha_{j} + \beta = 1, &\\
                    \sum_{j = 1}^{n} \alpha_j + \alpha_{n + 1} + \beta \leq 1 &\\
                    \sum_{j = 1}^{n} \alpha_j + \alpha_{n + j'}  + \beta \leq 0, & \forall j' \in [n_1 - n] \setminus \left\{1\right\}\\
                    \sum_{j \in [n] \setminus \left\{j'\right\}} \alpha_j + \beta \leq 0, & \forall j' \in [n - 1]\\
                    \sum_{j = 1}^{n - 1} \alpha_j + \beta \leq 0\text{ or }1\\
                    \sum_{j=1}^{n - 1} \alpha_j + \alpha_{n + 1} + \beta \leq 0 &
                \end{cases}\\
                &\Leftrightarrow
                \begin{cases}
                    \sum_{j=1}^n \alpha_{j} + \beta = 1, &\\
                    \alpha_{n + 1} \leq 0, & \\
                    \alpha_{n + j} \leq -1, &\forall j \in [n_1 - n] \setminus \{1\}\\
                    \alpha_{j'} \geq 1, & \forall j' \in [n - 1]\\
                    \alpha_{n} \geq 1 \text{ or }0 & \\
                    \alpha_{n} - \alpha_{n + 1} \geq 1 & 
                \end{cases}
        \end{aligned}
        \]
        where $\sum_{j = 1}^{n - 1} \alpha_j + \beta \leq 0$, i.e., $\alpha_{n} \geq 1$  when $\bar{\bm{x}}^* - \bm{e}_{n} \notin A$; And $\sum_{j = 1}^{n - 1} \alpha_j + \beta \leq 1$, i.e., $\alpha_{n} \geq 0$  when $\bar{\bm{x}}^* - \bm{e}_{n} \in A$. In the following proof, we will discuss on those two cases when necessary.
        
        Next, we show that $E' = E \cap \left\{\bm{\alpha}^\top \bar{\bm{x}}^* + \beta = 1\right\}$. It is sufficient to show that $E' \subseteq E \cap \left\{\bm{\alpha}^\top \bar{\bm{x}}^* + \beta = 1\right\}$ because $E \cap \left\{\bm{\alpha}^\top \bar{\bm{x}}^* + \beta = 1\right\} \subseteq E'$ holds by definition of $E'$.

        Suppose $(\bm{\alpha}, \beta) \in E'$ and for each $\bm{x} \in \{0, 1\}^{n_1}$, let $n':= |I(\bm{x})|$. We first show that if $\bm{x} \in A$, then $\bm{\alpha}^\top \bm{x} + \beta \leq 1$.
        
        There are two possibilities:
        \begin{enumerate}
            \item If $I(\bm{x}) \subseteq I(\bar{\bm{x}}^*)$, then 
            \begin{align*}
                \bm{\alpha}^\top \bm{x} + \beta  &= \alpha_1 + \ldots + \alpha_{n'} + \beta\\
                &= 1 - \alpha_{n' + 1} - \ldots - \alpha_{n}\\
                &\leq 1
            \end{align*}
            \item If $I(\bar{\bm{x}}^*) \subseteq I(\bm{x})$ then 
            \begin{align*}
                \bm{\alpha}^\top \bm{x} + \beta &= \alpha_1 + \ldots + \alpha_{n'} + \beta \\
                &= 1 + \alpha_{n + 1} + \ldots + \alpha_{n'}\\
                &\leq 1
            \end{align*}
        \end{enumerate}

        If $\bm{x} \notin A$, we have 
        
        \begin{align*}
            \bm{\alpha}^\top \bm{x} + \beta &= \sum_{i \in I(\bm{x}) \cap I(\bar{\bm{x}}^*)} \alpha_i + \sum_{i \in I(\bm{x}) \cap I^c(\bar{\bm{x}}^*)} \alpha_i + \beta\\
            &= 1 - \sum_{i \in I^c(\bm{x}) \cap I(\bar{\bm{x}}^*)} \alpha_i + \sum_{i \in I(\bm{x}) \cap I^c(\bar{\bm{x}}^*)} \alpha_i
        \end{align*}
        The right hand side (RHS) is either not greater than 0 or equal to 1 by definition of $E'$. If the RHS is 1, then we must have $\sum_{i \in I^c(\bm{x}) \cap I(\bar{\bm{x}}^*)} \alpha_i = 0, \sum_{i \in I(\bm{x}) \cap I^c(\bar{\bm{x}}^*)} \alpha_i = 0$, where the first equation indicates that $I^c(\bm{x}) \cap I(\bar{\bm{x}}^*) = \{n\}$ or $\emptyset$ and the second equation indicates that $I(\bm{x}) \cap I^c(\bar{\bm{x}}^*) = \left\{n + 1\right\}$ or $\emptyset$.

        Thus, there are four possible values for $\bm{x}$:
        \begin{enumerate}
            \item $I^c(\bm{x}) \cap I(\bar{\bm{x}}^*) = \emptyset, I(\bm{x}) \cap I^c(\bar{\bm{x}}^*) = \emptyset$, indicating that $\bm{x} = \bar{\bm{x}}^*$ which contradicts the assumption that $\bm{x} \notin A$;
            \item $I^c(\bm{x}) \cap I(\bar{\bm{x}}^*) = \left\{n\right\}, I(\bm{x}) \cap I^c(\bar{\bm{x}}^*) = \emptyset$, indicating that $\bm{x} = \bar{\bm{x}}^* - \bm{e}_{n}$. If $\bar{\bm{x}}^* - \bm{e}_{n} \in A$, then this contradicts the assumption that $\bm{x} \notin A$; If $\bar{\bm{x}}^* - \bm{e}_{n } \notin A$, then $\alpha_{n} \geq 1$ which also contradicts the assumption that $\sum_{i \in I^c(\bm{x}) \cap I(\bar{\bm{x}}^*)} \alpha_i = 0$.
            \item $I^c(\bm{x}) \cap I(\bar{\bm{x}}^*) = \emptyset, I(\bm{x}) \cap I^c(\bar{\bm{x}}^*) = \left\{n + 1\right\}$, indicating that $\bm{x} = \bar{\bm{x}}^* + \bm{e}_{n + 1} \in A$, contradicting the assumption that $\bm{x} \notin A$. 
            \item $I^c(\bm{x}) \cap I(\bar{\bm{x}}^*) = \left\{n\right\}, I(\bm{x}) \cap I^c(\bar{\bm{x}}^*) = \left\{n + 1\right\}$, indicating that $\bm{x} = \bar{\bm{x}}^* - \bm{e}_{n} + \bm{e}_{n + 1} \notin A$ but this gives $\bm{\alpha}^\top \bm{x} + \beta = \sum_{j=1}^{n - 1} \alpha_j + \alpha_{n + 1} + \beta \leq 0$ by definition of $E'$.
        \end{enumerate}

        Therefore, we have $\bm{\alpha}^\top \bm{x} + \beta \leq 0, \forall \bm{x} \notin A$. And it follows that $E' = E \cap \left\{\alpha^\top \bar{\bm{x}}^* + \beta = 1\right\}$.

        Then it is sufficient to find all extreme points of $E'$. Since there are $n_1 + 2$ hyperplanes in $E'$, we observed that $\alpha_n \geq 1 \text{ or } 0$, $\alpha_{n + 1} \leq 0$ and $\alpha_n - \alpha_{n + 1} \geq 1$ cannot be in the $n_1 + 1$ hyperplanes at the same time since they are not linearly independent. Thus, at most two of them appear in the $n_1 + 1$ linearly independent hyperplanes. Thus, there are three possible combinations of $n_1 + 1$ linearly independent hyperplanes:
        \begin{enumerate}
            \item $\alpha_{n} \geq 1 \text{ or } 0, \alpha_{n} - \alpha_{n + 1} \geq 1$ are in the $n_1 + 1$ hyperplanes while $\alpha_{n + 1} \leq 0$ is not. This gives the following extreme point:
            \begin{enumerate}
                \item If $\alpha_{n} \geq 0$, then the extreme point is $\alpha_{n} = 0, \alpha_{n + 1} = -1, \alpha_{n + j} = -1, j = 2, \ldots, n_1 - n, \alpha_j = 1, j = 1, \ldots, n - 1, \beta = 2 - n$ and provides the hyperplane for the convex envelope:
                $$\theta \geq \sum_{j = 1}^{n - 1} x_j - \sum_{j=1}^{n_1 - n} x_{n + j} + 2 - n$$
                which is consistent with the expression in the proposition.
                \item If $\alpha_{n} \geq 1$, then the extreme point is $\alpha_{n} = 1, \alpha_{n + 1} = 0, \alpha_{n + j} = -1, j = 2, \ldots, n_1 - n, \alpha_j = 1, j = 1, \ldots, n - 1, \beta = 1 - n$ and provides the following hyperplane for the convex envelope:
                $$\theta \geq \sum_{j = 1}^{n} x_j - \sum_{j=2}^{n_1 - n} x_{n + j} + 1 - n$$
                which is also consistent with the expression in the proposition.
            \end{enumerate}
            \item $\alpha_{n} \geq 1 \text{ or } 0, \alpha_{n + 1} \leq 0$ are in the $n_1 + 1$ hyperplanes while $\alpha_{n} - \alpha_{n + 1} \geq 1$ is not. This gives the following extreme point:
            \begin{enumerate}
                \item If $\alpha_{n} \geq 0$, then the extreme point has $\alpha_{n} = 0, \alpha_{n + 1} = 0$ but it is not in $E'$ due to conflict with $\alpha_n - \alpha_{n + 1}\geq 1$.
                \item If $\alpha_{n} \geq 1$, then the extreme point is $\alpha_{n} = 1, \alpha_{n + 1} = 0, \alpha_{n + j} = -1, j = 2, \ldots, n_1 - n, \alpha_j = 1, j = 1, \ldots, n - 1, \beta = 1 - n$, which provides the following hyperplane for the convex envelope:
                $$\theta \geq \sum_{j = 1}^{n} x_j - \sum_{j=2}^{n_1 - n} x_{n + j} + 1 - n$$
                which is consistent with the expression in the proposition.
            \end{enumerate}
            \item $\alpha_{n} - \alpha_{n + 1} \geq 1, \alpha_{n + 1} \leq 0$ are in the $n_1 + 1$ hyperplanes while $\alpha_{n} \geq 0 / 1$ is not. This gives the following extreme point: $\alpha_{n} = 1, \alpha_{n + 1} = 0, \alpha_{n + j} = -1, j = 2, \ldots, n_1 - n, \alpha_j = 1, j = 1, \ldots, n - 1, \beta = 1 - n$. And this provides the following hyperplane for the convex envelope:
            $$\theta \geq \sum_{j = 1}^{n} x_j - \sum_{j=2}^{n_1 - n} x_{n + j} + 1 - n$$
            which is consistent with the expression in the proposition.
        \end{enumerate}
        \item Case 3: At least one hyperplane comes from $H_{A}$ and suppose $\bm{\alpha}^\top \bar{\bm{x}}^{*} + \beta = 1$ is 
        one of the $n_1 + 1$ hyperplanes and $\bar{\bm{x}}^* \in A, \bar{\bm{x}}^* + \bm{e}_{n + 1} \notin A$.

        Consider the following subsystem

         \[
        \begin{aligned}
                E' :&= \begin{cases}
                    \sum_{j=1}^n \alpha_{j} + \beta = 1, &\\
                    \sum_{j = 1}^{n} \alpha_j + \alpha_{n + j'} + \beta \leq 0, & \forall j' \in [n_1 - n]\\
                    \sum_{j \in [n] \setminus \left\{j'\right\}} \alpha_j + \beta \leq 0, & \forall j' \in [n - 1]\\
                    \sum_{j = 1}^{n - 1} \alpha_j + \beta \leq 0 \text{ or } 1\\
                \end{cases}\\
                &\Leftrightarrow
                \begin{cases}
                    \sum_{j=1}^n \alpha_{j} + \beta = 1, &\\
                    \alpha_{n + j} \leq -1, &\forall j \in [n_1 - n]\\
                    \alpha_{j} \geq 1, & \forall j \in [n - 1]\\
                    \alpha_{n} \geq 1 \text{ or } 0 & \\
                \end{cases}
        \end{aligned}
        \]
        where $\sum_{j = 1}^{n - 1} \alpha_j + \beta \leq 0$, i.e., $\alpha_{n} \geq 1$  when $\bar{\bm{x}}^* - \bm{e}_{n} \notin A$; And $\sum_{j = 1}^{n - 1} \alpha_j + \beta \leq 1$, i.e., $\alpha_{n} \geq 0$  when $\bar{\bm{x}}^* - \bm{e}_{n} \in A$. In the following proof, we will discuss those two cases when necessary.

        Next we prove that $E' = E \cap \left\{\bm{\alpha}^\top \bar{\bm{x}}^* + \beta = 1\right\}$. Similarly, it is sufficient to show that $E' \subseteq E \cap \left\{\bm{\alpha}^\top \bar{\bm{x}}^* + \beta = 1\right\}$. 

        Suppose $(\bm{\alpha}, \beta) \in E'$ and let $n':= |I(\bm{x})|$ for $\bm{x} \in \{0, 1\}^{n_1}$, then, if $\bm{x} \in A$
        \begin{enumerate}
            \item If $I(\bm{x}) \subseteq I(\bar{\bm{x}}^*)$, then 
            \begin{align*}
                \bm{\alpha}^\top \bm{x} + \beta  &= \alpha_1 + \ldots + \alpha_{n'} + \beta\\
                &= 1 - \alpha_{n'+ 1} - \ldots - \alpha_{n}\\
                &\leq 1
            \end{align*}
            \item If $I(\bar{\bm{x}}^*) \subseteq I(\bm{x})$ then 
            \begin{align*}
                \bm{\alpha}^\top \bm{x} + \beta &= \alpha_1 + \ldots + \alpha_{n'} + \beta \\
                &= 1 + \alpha_{n + 1} + \ldots + \alpha_{n'}\\
                &\leq 1
            \end{align*}
        \end{enumerate}

        For $\bm{x} \notin A$, we have 
        \begin{align*}
            \bm{\alpha}^\top \bm{x} + \beta &= \sum_{i \in I(\bm{x}) \cap I(\bar{\bm{x}}^*)} \alpha_i + \sum_{i \in I(\bm{x}) \cap I^c(\bar{\bm{x}}^*)} \alpha_i + \beta\\
            &= 1 - \sum_{i \in I^c(\bm{x}) \cap I(\bar{\bm{x}}^*)} \alpha_i + \sum_{i \in I(\bm{x}) \cap I^c(\bar{\bm{x}}^*)} \alpha_i\\
        \end{align*}

        If $\bm{\alpha}^\top \bm{x} + \beta = 1$, then we have $ \sum_{i \in I^c(\bm{x}) \cap I(\bar{\bm{x}}^*)} \alpha_i = 0, \sum_{i \in I(\bm{x}) \cap I^c(\bar{\bm{x}}^*)} \alpha_i = 0$. There are two possibilities:
        \begin{enumerate}
            \item $I^c(\bm{x}) \cap I(\bar{\bm{x}}^*) = \emptyset, I(\bm{x}) \cap I^c(\bar{\bm{x}}^*) = \emptyset$, indicating that $\bm{x} = \bar{\bm{x}}^*$ which contradicts $\bm{x} \notin A$.
            \item $I^c(\bm{x})\cap I(\bar{\bm{x}}^*) = \left\{n\right\}, I(\bm{x}) \cap I^c(\bar{\bm{x}}^*) = \emptyset$, indicating that $\bm{x} = \bar{\bm{x}}^* - \bm{e}_{n}$. Since $\bm{x} \notin A$, we have $\alpha_{n} \geq 1$, which contradicts $\sum_{i \in I^c(\bm{x}) \cap I(\bar{\bm{x}}^*)} \alpha_i = 0$.
        \end{enumerate}
        So we have $\bm{\alpha}^\top \bm{x} + \beta \leq 0, \forall \bm{x} \notin A$ and it follows that $E' = E \cap \left\{\bm{\alpha}^\top \bar{\bm{x}}^* + \beta = 1\right\}$. 

        Since there are only $n_1 + 1$ hyperplanes, we have the following extreme points:
        \begin{enumerate}
            \item If $\alpha_{n} \geq 0$, then the extreme point is $\alpha_j = 1, j = 1, \ldots, n - 1, \alpha_{n} = 0, \alpha_{n + j} = -1, j = 1, \ldots, n_1 - n, \beta = 2 - n$, which gives the following hyperplane for the convex envelope
            $$\theta \geq \sum_{j = 1}^{n - 1} x_j - \sum_{j=1}^{n_1 - n} x_{n + j} + 2 - n$$
            which is consistent with the statement in the proposition.
            \item If $\alpha_{n} \geq 1$, then the extreme point is $\alpha_j = 1, j = 1, \ldots, n, \alpha_{n + j} = -1, j = 1, \ldots, n_1 - n, \beta = 1 - n$, which gives the following hyperplane for the convex envelope
            $$\theta \geq \sum_{j = 1}^{n} x_j - \sum_{j=1}^{n_1 - n} x_{n + j} + 1 - n$$
            which is consistent with the statement in the proposition.
        \end{enumerate}
    \end{enumerate}

    Combining Cases 1, 2, and 3, the extreme points of $E$ correspond to: (i) the trivial cut $\theta \geq 0$ (Case 1), and (ii) cuts of the form in the proposition, indexed by $\bar{\bm{x}}^* \in A$. As $\bar{\bm{x}}^*$ ranges over $A$ in Cases 2 and 3, all cuts in the proposition's two families are obtained. Therefore, $\overline{\operatorname{co}}(I_A)$ admits the polyhedral representation given in the proposition.
\endproof

\subsection{Proof of Proposition \ref{prop:convexEnvForTree}}\label{sec:convexEnvBinaryTree}
\proof{Proof. }
    As in the proof of Proposition~\ref{prop:convexEnvForChain}, it is sufficient to enumerate all extreme points of $E := \left\{(\bm{\alpha}, \beta) \mid (\ref{eq:conenvCons1}), (\ref{eq:conenvCons2})\right\}$.

    Consider an extreme point of $E$ at which the tight $H_A$-constraint corresponds to an element of $A$ other than $\bar{\bm{x}}$. Such an element lies in $A_1$ or $A_2$ (or both). Without loss of generality, suppose it lies in $A_1$. By the 2-tree structure, the Hamming-1 neighbors of this element have the same $A$-membership as their $A_1$-membership, so the local constraint structure in $E$ matches the local constraint structure in the corresponding $E$-system for chain $A_1$ alone. Consequently, the argument from the proof of Proposition~\ref{prop:convexEnvForChain} applied to chain $A_1$ identifies the same extreme point, which is therefore also an extreme point of $E$, and yields the same anchored cut. Hence the extreme points of $E$ with tight $H_A$-constraint at an element of $A$ other than $\bar{\bm{x}}$ are exactly the extreme points (and cuts) of $\overline{\operatorname{co}}(I_{A_1})$ and (symmetrically) $\overline{\operatorname{co}}(I_{A_2})$.

    It remains to analyze the extreme points of $E$ at which the tight $H_A$-constraint is at $\bar{\bm{x}}$.
   
    
    Suppose $\bm{\alpha}^\top \bm{x} + \beta = I_A(\bar{\bm{x}})$ is one of the equations defining one of $E's$ extreme points and let $n = |I(\bar{\bm{x}})|$. Then, we consider the following system $E'$:

\begin{subequations}
\begin{empheq}[left=\empheqlbrace]{align}
\textbf{G1:}\;&
\alpha_{1} + \cdots + \alpha_{n} + \beta = 1
\label{eq:EpG1-1}
\\[3pt]
&
\alpha_{1} + \cdots + \alpha_{n}
+ \alpha_{n + 1} + \beta \le 1 \text{ or } 0
\label{eq:EpG1-2}
\\[3pt]
&
\alpha_{1} + \cdots + \alpha_{n}
+ \alpha_{n + i} + \beta \le 0\nonumber
\\[-0.3em]
& \hfill \forall\, i = 2,\ldots,n_1 - n \label{eq:EpG1-3}
\\[6pt]
\textbf{G2:}\;&
\alpha_{1} + \cdots + \alpha_{i - 1}
+ \alpha_{i + 1} + \cdots + \alpha_{n} + \beta \le 0
\nonumber
\\[-0.3em]
& \hfill \forall\, i \in [n - 1],\ i \neq i^* \label{eq:EpG2-1}
\\[3pt]
&
\alpha_{1} + \cdots + \alpha_{n - 1} + \beta \le 1
\label{eq:EpG2-2}
\\[3pt]
&
\alpha_{1} + \cdots + \alpha_{i^* - 1} + 0
+ \alpha_{i^* + 1} + \cdots + \alpha_{n} + \beta \le 1
\label{eq:EpG2-3}
\\[6pt]
\textbf{G3:}\;&
\alpha_{1} + \cdots + \alpha_{i^* - 1} + 0
+ \alpha_{i^* + 1} + \cdots + \alpha_{n - 1} + \beta \le 0
\label{eq:EpG3-1}
\\[3pt]
&
\alpha_{1} + \cdots + \alpha_{i^* - 1} + 0
+ \alpha_{i^* + 1} + \cdots + \alpha_{n}
+ \alpha_{n + 1} + \beta \le 0
\label{eq:EpG3-2}
\\[3pt]
&
\alpha_{1} + \cdots + \alpha_{n - 1}
+ \alpha_{n + 1} + \beta \le 0
\label{eq:EpG3-3}
\end{empheq}
\end{subequations}

where $\bm{x}(\cdot)$ maps the index set into a binary vector. For example, $\bm{x}\left(I\left(\bm{x}'\right)\right) = \bm{x}'$. Besides, we divide the constraints into three groups as shown above: group 1 with $n_1 - n + 1$ constraints; group 2 with $n$ constraints and group 3 with 3 constraints.

$E'$ is equivalent to the following system:
\begin{subequations}
\begin{empheq}[left={\empheqlbrace}]{align}
\textbf{G1:}\;& \alpha_{1} + \cdots + \alpha_{n} + \beta = 1  &  &        \label{eq:EppG1-1}\\
                   & \alpha_{n + 1} \leq 0 \text{ or -1}& &\label{eq:EppG1-2}\\
                   & \alpha_{n + i} \leq -1 & i = 2, \ldots, n_1 - n &               \label{eq:EppG1-3}\\[4pt]
\textbf{G2:}\;&  \alpha_{i} \geq 1& i \in [n - 1], i \neq i^* &    \label{eq:EppG2-1}\\
                   & \alpha_{n} \geq 0& & \label{eq:EppG2-3}\\
                   &\alpha_{i^*} \geq 0   &  &             \label{eq:EppG2-2}\\[4pt]
\textbf{G3:}\;& \alpha_{n} + \alpha_{i^*} \geq 1&  &    \label{eq:EppG3-1}\\
                   & \alpha_{n + 1} \leq \alpha_{i^*} - 1 &  &    \label{eq:EppG3-2}\\
                   & \alpha_{n + 1} \leq \alpha_{n} - 1 &  &    \label{eq:EppG3-3}
\end{empheq}
\end{subequations}
We first show that $E' = E\;\cap \;\left\{(\bm{\alpha}, \beta)| \bm{\alpha}^\top \bar{\bm{x}} + \beta = 1\right\}$. Since $E'$ only contains a subset of constraints in $E\;\cap \;\left\{(\bm{\alpha}, \beta)| \bm{\alpha}^\top \bar{\bm{x}} + \beta = 1\right\}$, we have $E' \supseteq E\cap \left\{(\bm{\alpha}, \beta)| \bm{\alpha}^\top \bar{\bm{x}} + \beta = 1\right\}$. Next we show $E' \subseteq E \cap \left\{(\bm{\alpha}, \beta)| \bm{\alpha}^\top \bar{\bm{x}} + \beta = 1\right\}$ by showing that $\bm{\alpha}^\top \bm{x} + \beta \leq I_A(\bm{x}), \forall (\bm{\alpha}, \beta) \in E', \bm{x} \in \{0, 1\}^{n_1}$:

\begin{enumerate}
    \item Case 1: $\bm{x} \in A_1$ and $I(\bm{x}) \subseteq I(\bar{\bm{x}})$.

    In this case, $\bm{\alpha}^\top\bm{x} + \beta = \sum_{i \in I(\bm{x})} \alpha_i + \beta = 1 - \sum_{i \in I(\bar{\bm{x}}) \setminus I(\bm{x})} \alpha_i \leq 1$ because $\alpha_i \geq 0, i \in I(\bm{x})$ by the group 2 constraints in $E'$.
    \item Case 2: $\bm{x} \in A_1$ and $I(\bar{\bm{x}}) \subseteq I(\bm{x})$.

    In this case, $\bm{\alpha}^\top \bm{x} + \beta = \sum_{i \in I(\bm{x})} \alpha_i + \beta = 1 + \sum_{i \in I(\bm{x}) \setminus I(\bar{\bm{x}})} \alpha_i \leq 1$ because $\alpha_i \leq 0, i \notin I(\bar{\bm{x}})$ by the group 1 constraints in $E'$.

    \item Case 3: $\bm{x} \in A_2$ and $I(\bm{x}) \subseteq I(\bar{\bm{x}})$.

    In this case, $\bm{\alpha}^\top \bm{x} + \beta \leq 1$ by the same argument as in Case 1.

    \item Case 4: $\bm{x} \notin A$.

    In this case,
    \begin{subequations}
    \begin{align}
        \bm{\alpha}^\top \bm{x} + \beta &= \sum_{i \in I(\bm{x})} \alpha_i + \beta\\
        &= \sum_{i \in I(\bm{x}) \cap I(\bar{\bm{x}})} \alpha_i + \sum_{i \in I(\bm{x})\cap I^c(\bar{\bm{x}})} \alpha_i + \beta\\
        &= 1 - \sum_{i \in I^c(\bm{x}) \cap I(\bar{\bm{x}})} \alpha_i + \sum_{i \in I(\bm{x}) \cap I^c(\bar{\bm{x}})}\alpha_i\\
        &= 1 - \sum_{i \in I^c(\bm{x}) \cap I(\bar{\bm{x}}) \cap \left\{n, i^*\right\}} \alpha_i - \sum_{i \in I^c(\bm{x}) \cap I(\bar{\bm{x}}) \cap \left\{n, i^*\right\}^c} \alpha_i \nonumber\\
        & \qquad + \sum_{i \in  I(\bm{x}) \cap I^c(\bar{\bm{x}}) \cap \left\{n + 1\right\}} \alpha_i + \sum_{i \in  I(\bm{x}) \cap I^c(\bar{\bm{x}}) \cap \left\{n + 1\right\}^c} \alpha_i \label{eq:beforeRelax}\\
        & \leq 1 - \left|I^c(\bm{x}) \cap I(\bar{\bm{x}}) \cap \left\{n, i^*\right\}^c\right| - \left|I(\bm{x}) \cap I^c(\bar{\bm{x}}) \cap \left\{n + 1\right\}^c\right| \nonumber\\
    \end{align}
    \end{subequations}

    It is sufficient to consider the case 
    \[
    I^c(\bm{x}) \cap I(\bar{\bm{x}}) \cap \left\{n, i^*\right\}^c = \emptyset
    \Leftrightarrow 
    I^c(\bm{x}) \cap I(\bar{\bm{x}}) \subseteq \left\{n, i^*\right\}
    \] 
    and
    \[
    I(\bm{x}) \cap I^c(\bar{\bm{x}}) \cap \left\{n+ 1\right\}^c = \emptyset
    \Leftrightarrow
    I(\bm{x}) \cap I^c(\bar{\bm{x}}) \subseteq \left\{n + 1\right\}
    \]

    If $I(\bm{x}) \cap I^c(\bar{\bm{x}}) = \emptyset$, then 
    $I(\bm{x}) \subseteq I(\bar{\bm{x}})$. It follows that $I(\bm{x})$ 
    has three possible values: 1. $I(\bm{x}) = I(\bar{\bm{x}}) \cap 
    \left\{n\right\}^c$ but this contradicts the assumption that 
    $\bm{x} \notin A_1$. 2. $I(\bm{x}) = I(\bar{\bm{x}}) 
    \cap \left\{i^*\right\}^c$ but this contradicts the 
    assumption that $\bm{x} \notin A_2$. 
    3. $I(\bar{\bm{x}}) \cap \left\{n, i^*\right\}^c$, which 
    gives $\bm{\alpha}^\top \bm{x} + \beta \leq 0$ by definition of $E'$.

    If $I(\bm{x}) \cap I^c(\bar{\bm{x}}) = \left\{n + 1\right\}$,
    $I(\bm{x})$ has four possible values: 1. $\left(I(\bar{\bm{x}}) \cap \left\{n\right\}^c\right)
    \cup \left\{n + 1\right\}$, which gives $\bm{\alpha}^\top \bm{x} + 
    \beta \leq 0$ by definition of $E'$. 2. $\left(I(\bar{\bm{x}}) \cap 
    \left\{i^*\right\}^c\right) \cup \left\{n + 1\right\}$, 
    which gives $\bm{\alpha}^\top\bm{x} + \beta \leq 0$ by definition of $E'$.
    3. $\left(I(\bar{\bm{x}}) \setminus \left\{i^*, n\right\}
    \right) \cup \left\{n + 1\right\}$ which gives 
    $\bm{\alpha}^\top \bm{x} + \beta \leq 0$ by the first group 3 constraint in 
    $E'$ and that $\alpha_{n + 1} \leq 0$. 
    4. $I(\bar{\bm{x}}) \cup \left\{n + 1\right\}$, which 
    indicates that $\bm{x}\left(I(\bar{\bm{x}}) \cup \left\{n + 1\right\}\right) \notin A$. Thus, we have $\alpha_{n+1} \leq -1$. Then we have the relaxed term, $\sum_{i \in  I(\bm{x}) \cap I^c(\bar{\bm{x}}) \cap \left\{n + 1\right\}} \alpha_i$, in (\ref{eq:beforeRelax}) to be -1, making $(\ref{eq:beforeRelax}) \leq 0$.

    Thus, $\bm{\alpha}^\top\bm{x} + \beta \leq I_A(\bm{x}), \forall (\bm{\alpha}, \beta) \in E', \bm{x} \in \left\{0, 1\right\}^{n_1}$.
    \end{enumerate}

    Next, we need to find all extreme points of $E'$.

    \begin{enumerate}
        \item Suppose $\bm{x}\left(I(\bar{\bm{x}}) \cup \left\{n + 1\right\}\right) \in A$, i.e., $\alpha_{n + 1} \leq 0$.
        \begin{enumerate}
        \item Case 1: the constraint $\alpha_{n + 1} \leq 0$ is tight, 
        i.e., $\alpha_{n + 1} = 0$.

        By definition of $E'$, we have the following system:
        \begin{subequations}
        \begin{empheq}[left={\empheqlbrace}]{align}
            \textbf{G1:}\;& \alpha_{1} + \cdots + \alpha_{n} + \beta = 1  &  &        \label{eq:EG1-1}\\
                               & \alpha_{n + 1} = 0 & &\label{eq:EG1-2}\\
                               & \alpha_{n + i} \leq -1 & i = 2, \ldots, n_1 - n &               \label{eq:EppG1-3}\\[4pt]
            \textbf{G2:}\;&  \alpha_{i} \geq 1& i \in [n - 1], i \neq i^* &    \label{eq:EG2-1}\\
                               & \alpha_{n} \geq 0& & \label{eq:EG2-3}\\
                               &\alpha_{i^*} \geq 0   &  &             \label{eq:EG2-2}\\[4pt]
            \textbf{G3:}\;& \alpha_{n} + \alpha_{i^*} \geq 1&  &    \label{eq:EG3-1}\\
                               & 0 \leq \alpha_{i^*} - 1 &  &    \label{eq:EG3-2}\\
                               & 0 \leq \alpha_{n} - 1 &  &    \label{eq:EG3-3}
        \end{empheq}
        \end{subequations}
        Constraints (\ref{eq:EG3-2}) and (\ref{eq:EG3-3}) are equivalent to $\alpha_{i^*} \geq 1$ and 
        $\alpha_{n} \geq 1$, which make the constraints (\ref{eq:EG2-2}), 
        (\ref{eq:EG2-3}) and (\ref{eq:EG3-1}) redundant. By solving the remaining system with all equations 
        hold, we get the following solution:
        \begin{align}
            \begin{cases}
                \alpha_i = 1,& i \in [n]\\
                \alpha_{n + 1} = 0, & \\
                \alpha_{n + i} = -1, & i = 2,\ldots, n_1 - n \\
                \beta = 1 - n
            \end{cases}
        \end{align}
        which is a facet in the convex envelope of chain 1.
        \item Case 2: The constraint $\alpha_{i^*} \geq 0$ is tight, i.e., $\alpha_{i^*} = 0$. 
        By definition of $E'$, we have the following system:
        \begin{subequations}
        \begin{empheq}[left={\empheqlbrace}]{align}
            \textbf{G1:}\;& \alpha_{1} + \cdots + \alpha_{n} + \beta = 1  &  &        \label{eq:c2EG1-1}\\
                               & \alpha_{n + 1} \leq 0 & &\label{eq:c2EG1-2}\\
                               & \alpha_{n + i} \leq -1 & i = 2, \ldots, n_1 - n &               \label{eq:c2EppG1-3}\\[4pt]
            \textbf{G2:}\;&  \alpha_{i} \geq 1& i \in [n - 1], i \neq i^* &    \label{eq:c2EG2-1}\\
                                & \alpha_{n} \geq 0& & \label{eq:c2EG2-3}\\
                               &\alpha_{i^*} = 0   &  &             \label{eq:c2EG2-2}\\[4pt]
            \textbf{G3:}\;& \alpha_{n} + 0 \geq 1&  &    \label{eq:c2EG3-1}\\
                               & \alpha_{n + 1} \leq 0 - 1 &  &    \label{eq:c2EG3-2}\\
                               & \alpha_{n + 1} \leq \alpha_{n} - 1 &  &    \label{eq:c2EG3-3}
        \end{empheq}
        \end{subequations}
        Constraints (\ref{eq:c2EG3-2}) and (\ref{eq:c2EG2-3}) make constraints (\ref{eq:c2EG1-2}) and (\ref{eq:c2EG3-3})
        redundant. The constraint (\ref{eq:c2EG3-1}) makes the constraint (\ref{eq:c2EG2-3}) redundant. By solving the
        remaining system with all equations hold, we get the following solution:
        \begin{align}
            \begin{cases}
                \alpha_{i^*} = 0 & \\
                \alpha_{n + i} = -1, & i = 1, \ldots, n_1 - n\\
                \alpha_{i} = 1, & i \in \left[n\right], i \neq i^*\\
                \beta = 2 - n & 
            \end{cases}
        \end{align}
        which is a facet in the convex envelope of chain 2.

        \item Case 3: The constraint $\alpha_{n} \geq 0$ is tight, i.e., 
        $\alpha_{n} = 0$. By definition of $E'$, we have the following 
        system:
        \begin{subequations}
        \begin{empheq}[left={\empheqlbrace}]{align}
            \textbf{G1:}\;& \alpha_{1} + \cdots + \alpha_{n} + \beta = 1  &  &        \label{eq:c3EG1-1}\\
                               & \alpha_{n + 1} \leq 0 & &\label{eq:c3EG1-2}\\
                               & \alpha_{n + i} \leq -1 & i = 2, \ldots, n_1 - n &               \label{eq:c3EG1-3}\\[4pt]
            \textbf{G2:}\;&  \alpha_{i} \geq 1& i \in [n - 1], i \neq i^* &    \label{eq:c3EG2-1}\\
                                & \alpha_{n} = 0& & \label{eq:c3EG2-3}\\
                               &\alpha_{i^*} \geq 0   &  &             \label{eq:c3EG2-2}\\[4pt]
            \textbf{G3:}\;& 0 + \alpha_{i^*} \geq 1&  &    \label{eq:c3EG3-1}\\
                               & \alpha_{n + 1} \leq \alpha_{i^*} - 1 &  &    \label{eq:c3EG3-2}\\
                               & \alpha_{n + 1} \leq 0 - 1 &  &    \label{eq:c3EG3-3}
        \end{empheq}
        \end{subequations}
        Constraints (\ref{eq:c3EG3-3}) and (\ref{eq:c3EG2-2}) make constraints (\ref{eq:c3EG1-2}) and (\ref{eq:c3EG3-2}) redundant. 
        The constraint (\ref{eq:c3EG3-1}) makes the constraint (\ref{eq:c3EG2-2}) redundant. By solving 
        the remaining system with all equations hold, we get the following solution:
        \begin{align}
            \begin{cases}
                \alpha_{i} = 1, i \in \left[n - 1\right], &i \neq i^*\\
                \alpha_{i^*} = 1,&\\
                \alpha_{n} = 0, &\\
                \alpha_{n + i} = -1, & i = 1, \ldots, n_1 - n\\
                \beta = 2 - n
            \end{cases}
        \end{align}
        which is a facet in the convex envelope of chain 1.

        \item Case 4: $\alpha_{n + 1} \leq 0, \alpha_{i^*} \geq 0, \alpha_{n} \geq 0$ are all loose.
        By solving the remaining system will all equations hold, we get the following solution:
        \begin{align}
            \begin{cases}
                \alpha_{i} = 1, & i \in \left[n -1\right], i\neq i^*\\
                \alpha_{i^*} = \frac{1}{2}\\
                \alpha_{n} = \frac{1}{2}\\
                \alpha_{n + 1} = -\frac{1}{2}\\
                \alpha_{n + i} = -1, i = 2, \ldots, n_1 - n\\
                \beta = 2 - n
            \end{cases}
        \end{align}
        which is the facet in the proposition statement.
    \end{enumerate}
    \item Suppose $\bm{x}\left(I(\bar{\bm{x}}) \cup \left\{n + 1\right\}\right) \notin A$, i.e., $\alpha_{n + 1} \leq -1$.
        \begin{enumerate}
        \item Case 1: the constraint $\alpha_{n + 1} \leq -1$ is tight, 
        i.e., $\alpha_{n + 1} = -1$.

        By definition of $E'$, we have the following system:
        \begin{subequations}
        \begin{empheq}[left={\empheqlbrace}]{align}
            \textbf{G1:}\;& \alpha_{1} + \cdots + \alpha_{n} + \beta = 1  &  &        \label{eq:EG1-1}\\
                               & \alpha_{n + 1} = -1 & &\label{eq:EG1-2-2}\\
                               & \alpha_{n + i} \leq -1 & i = 2, \ldots, n_1 - n &               \label{eq:EppG1-3-2}\\[4pt]
            \textbf{G2:}\;&  \alpha_{i} \geq 1& i \in [n - 1], i \neq i^* &    \label{eq:EG2-1-2}\\
                               & \alpha_{n} \geq 0& & \label{eq:EG2-3-2}\\
                               &\alpha_{i^*} \geq 0   &  &             \label{eq:EG2-2-2}\\[4pt]
            \textbf{G3:}\;& \alpha_{n} + \alpha_{i^*} \geq 1&  &    \label{eq:EG3-1-2}\\
                               & 0 \leq \alpha_{i^*} &  &    \label{eq:EG3-2-2}\\
                               & 0 \leq \alpha_{n} &  &    \label{eq:EG3-3-2}
        \end{empheq}
        \end{subequations}
        Constraints (\ref{eq:EG3-2-2}) and (\ref{eq:EG3-3-2}) make the constraints (\ref{eq:EG2-2-2}) and 
        (\ref{eq:EG2-3-2}) redundant. By constraint (\ref{eq:EG3-1-2}), (\ref{eq:EG3-2-2}) and (\ref{eq:EG3-3-2}) cannot be tight at the same time. Thus, there are two possibilities: 
        
        \begin{itemize}
            \item By solving the remaining system with all equations hold, except for $\alpha_n = 0$, we get the following solution:
        \begin{align}
            \begin{cases}
                \alpha_i = 1,& i \in [n] \setminus \{i^*\}\\
                \alpha_{i^*} = 0, & \\
                \alpha_{n + i} = -1, & i = 1,\ldots, n_1 - n \\
                \beta = \left(2 - n\right)
            \end{cases}
        \end{align}
        which is a facet in the convex envelope of $A_2$.
            \item By solving the remaining system with all equations hold, except for $\alpha_{i^*} = 0$, we get the following solution:
        \begin{align}
            \begin{cases}
                \alpha_i = 1,& i \in [n] \setminus \{n\}\\
                \alpha_{n} = 0\\
                \alpha_{n + i} = -1, & i = 1,\ldots, n_1 - n \\
                \beta = \left(2 - n\right)
            \end{cases}
        \end{align}
        which is a facet in the convex envelope of $A_1$.
        \end{itemize}
    
        \item Case 2: The constraint $\alpha_{i^*} \geq 0$ is tight, i.e., $\alpha_{i^*} = 0$. 
        By definition of $E'$, we have the following system:
        \begin{subequations}
        \begin{empheq}[left={\empheqlbrace}]{align}
            \textbf{G1:}\;& \alpha_{1} + \cdots + \alpha_{n} + \beta = 1  &  &        \label{eq:c2EG1-1}\\
                               & \alpha_{n + 1} \leq -1 & &\label{eq:c2EG1-2-2}\\
                               & \alpha_{n + i} \leq -1 & i = 2, \ldots, n_1 - n &               \label{eq:c2EppG1-3-2}\\[4pt]
            \textbf{G2:}\;&  \alpha_{i} \geq 1& i \in [n - 1], i \neq i^* &    \label{eq:c2EG2-1-2}\\
                                & \alpha_{n} \geq 0& & \label{eq:c2EG2-3-2}\\
                               &\alpha_{i^*} = 0   &  &             \label{eq:c2EG2-2-2}\\[4pt]
            \textbf{G3:}\;& \alpha_{n} + 0 \geq 1&  &    \label{eq:c2EG3-1-2}\\
                               & \alpha_{n + 1} \leq 0 - 1 &  &    \label{eq:c2EG3-2-2}\\
                               & \alpha_{n + 1} \leq \alpha_{n} - 1 &  &    \label{eq:c2EG3-3-2}
        \end{empheq}
        \end{subequations}
        Constraints (\ref{eq:c2EG3-2-2}) and (\ref{eq:c2EG2-3-2}) make constraints (\ref{eq:c2EG1-2-2}) and (\ref{eq:c2EG3-3-2})
        redundant. The constraint (\ref{eq:c2EG3-1-2}) makes the constraint (\ref{eq:c2EG2-3-2}) redundant. By solving the
        remaining system with all equations hold, we get the following solution:
        \begin{align}
            \begin{cases}
                \alpha_{i^*} = 0 & \\
                \alpha_{n + i} = -1, & i = 1, \ldots, n_1 - n\\
                \alpha_{i} = 1, & i \in \left[n\right], i \neq i^*\\
                \beta = 2 - n & 
            \end{cases}
        \end{align}
        which is a facet in the convex envelope of chain 2.

        \item Case 3: The constraint $\alpha_{n} \geq 0$ is tight, i.e., 
        $\alpha_{n} = 0$. By definition of $E'$, we have the following 
        system:
        \begin{subequations}
        \begin{empheq}[left={\empheqlbrace}]{align}
            \textbf{G1:}\;& \alpha_{1} + \cdots + \alpha_{n} + \beta = 1  &  &        \label{eq:c3EG1-1-2}\\
                               & \alpha_{n + 1} \leq -1 & &\label{eq:c3EG1-2-2}\\
                               & \alpha_{n + i} \leq -1 & i = 2, \ldots, n_1 - n &               \label{eq:c3EG1-3-2}\\[4pt]
            \textbf{G2:}\;&  \alpha_{i} \geq 1& i \in [n - 1], i \neq i^* &    \label{eq:c3EG2-1-2}\\
                                & \alpha_{n} = 0& & \label{eq:c3EG2-3-2}\\
                               &\alpha_{i^*} \geq 0   &  &             \label{eq:c3EG2-2-2}\\[4pt]
            \textbf{G3:}\;& 0 + \alpha_{i^*} \geq 1&  &    \label{eq:c3EG3-1-2}\\
                               & \alpha_{n + 1} \leq \alpha_{i^*} - 1 &  &    \label{eq:c3EG3-2-2}\\
                               & \alpha_{n + 1} \leq 0 - 1 &  &    \label{eq:c3EG3-3-2}
        \end{empheq}
        \end{subequations}
        Constraints (\ref{eq:c3EG3-3-2}) and (\ref{eq:c3EG2-2-2}) make constraints (\ref{eq:c3EG1-2-2}) and (\ref{eq:c3EG3-2-2}) redundant. 
        The constraint (\ref{eq:c3EG3-1-2}) makes the constraint (\ref{eq:c3EG2-2-2}) redundant. By solving 
        the remaining system with all equations hold, we get the following solution:
        \begin{align}
            \begin{cases}
                \alpha_{i} = 1, i \in \left[n - 1\right], &i \neq i^*\\
                \alpha_{i^*} = 1,&\\
                \alpha_{n} = 0, &\\
                \alpha_{n + i} = -1, & i = 1, \ldots, n_1 - n\\
                \beta = 2 - n
            \end{cases}
        \end{align}
        which is a facet in the convex envelope of chain 1.

        \item Case 4: $\alpha_{n + 1} \leq -1, \alpha_{i^*} \geq 0, \alpha_{n} \geq 0$ are all loose.
        By solving the remaining system will all equations hold, we get the following solution:
        \begin{align}
            \begin{cases}
                \alpha_{i} = 1, & i \in \left[n -1\right], i\neq i^*\\
                \alpha_{i^*} = \frac{1}{2}\\
                \alpha_{n} = \frac{1}{2}\\
                \alpha_{n + 1} = -\frac{1}{2}\\
                \alpha_{n + i} = -1, i = 2, \ldots, n_1 - n\\
                \beta = 1 - \left(n - 1\right)
            \end{cases}
        \end{align}
        However, this contradicts the constraint $\alpha_{n + 1} \leq -1$. Thus, this case is not possible.
    \end{enumerate}
    \end{enumerate}
    Combining all cases, the extreme points of $E$ correspond to:
    \begin{enumerate}
        \item the cuts of $\overline{\operatorname{co}}(I_{A_1})$ and $\overline{\operatorname{co}}(I_{A_2})$, derived in Cases 1, 2, and 3 of both outer cases (these cases are degenerate and produce the same extreme points);
        \item when $\bar{\bm{x}} + \bm{e}_{n + 1} \in A$, the additional 2-tree cut derived in Case 4 of the first outer case.
    \end{enumerate}
    The condition that $I(\bm{x}) \subseteq I(\bar{\bm{x}})$ for all $\bm{x} \in A_1 \cup A_2$ in statement~1 of the proposition implies $\bar{\bm{x}} + \bm{e}_{n + 1} \notin A$, so no additional cut arises and $\overline{\operatorname{co}}(I_A) = \overline{\operatorname{co}}(I_{A_1}) \cap \overline{\operatorname{co}}(I_{A_2})$. Otherwise (statement~2), Case 4 contributes the additional cut stated in the proposition.
\endproof

\subsection{Proof of Theorem~\ref{prop:gddCons}}\label{sec:gddCons}
\proof{Proof. }
    Let $M > N\sup_{\bm{x} \in \cup_{i \in [N]} X_i} f(\bm{x})$. We first show the weak duality between~\eqref{problem:consP} and~\eqref{problem:consD}:

    Let $\bm{x}^*$ be the optimal solution of~\eqref{problem:consP}, then
    \begin{subequations}
    \begin{align}
        \text{\eqref{problem:consP}} &= f(\bm{x}^*) \nonumber\\
        &= \sum_{i=1}^N \left(\frac{1}{N} f(\bm{x}^*) + \frac{1}{N}g_i(\bm{x}^*) \right)\label{eq:zeroSum}\\
        &\geq \sum_{i=1}^N \min_{\bm{x}^i \in X_i} \left(\frac{1}{N} f(\bm{x}^i) + \frac{1}{N}g_i(\bm{x}^i) \right)\label{eq:relax}
    \end{align}
    \end{subequations}
    where (\ref{eq:zeroSum}) holds because $\sum_{i=1}^N g_i(\bm{x}) = 0, \forall \bm{x} \in \cup_{i\in[N]} X_i$ and (\ref{eq:relax}) holds because $\bm{x}^* \in X_i, i \in [N]$. By taking maximum over $g_i$, we have 
    $$\text{\eqref{problem:consP}} \geq \text{\eqref{problem:consD}}.$$

    To show the strong duality, we first define the extension of $f$ as follows:
    \begin{align*}
        F_i(\bm{x}) := \begin{cases}
            f(\bm{x}) & \text{if } \bm{x} \in X_i\\
            M & \text{otherwise}
        \end{cases}
    \end{align*}
    Then, let $g_i^*(\bm{x}) := \frac{\sum_{j=1}^N F_j(\bm{x})}{N} - F_i(\bm{x})$. It follows that
    \begin{align*}
        \text{\eqref{problem:consD}} &\geq \frac{1}{N} \sum_{i=1}^N \min_{\bm{x}^i \in X_i} \left(f(\bm{x}^i) + g^*_i(\bm{x}^i)\right)\\
        &= \frac{1}{N} \sum_{i=1}^N \min_{\bm{x}^i \in X_i} \left(f(\bm{x}^i) + \frac{\sum_{j=1}^N F_j(\bm{x}^i)}{N} - F_i(\bm{x}^i)\right)\\
        &= \frac{1}{N} \sum_{i=1}^N \min_{\bm{x}^i \in X_i} \left(F_i(\bm{x}^i) + \frac{\sum_{j=1}^N F_j(\bm{x}^i)}{N} - F_i(\bm{x}^i)\right)\\
        &= \frac{1}{N} \sum_{i=1}^N \min_{\bm{x}^i \in X_i} \frac{\sum_{j=1}^N F_j(\bm{x}^i)}{N}\\
    \end{align*}
    We claim that the optimal solution $\hat{\bm{x}}^i$ of the inner minimization problem $\min_{\bm{x}^i \in X_i} \frac{\sum_{j=1}^N F_j(\bm{x}^i)}{N}$ is in $\cap_{i\in [N]} X_i$. Otherwise, suppose there exists $j^* \in [N]$ such that $\hat{\bm{x}}^i \notin X_{j^*}$ and let $\bar{\bm{x}} \in \cap_{i\in [N]}X_i$. We have 
    \begin{align*}
        \sum_{j = 1}^N F_j(\hat{\bm{x}}^i) &\geq F_{j^*}(\hat{\bm{x}}^i)\\
        &= M \\
        &> N\sup_{\bm{x} \in \cup_{i \in [N]} X_i} f(\bm{x})\\
        &\geq N f(\bar{\bm{x}})\\
        &= \sum_{j=1}^N F_j(\bar{\bm{x}})
    \end{align*}
    contradicting the optimality of $\hat{\bm{x}}^i$. 

    Thus, we have 
    \begin{align*}
        \frac{1}{N} \sum_{i=1}^N\min_{\bm{x}^i \in X_i} \frac{\sum_{j=1}^N F_j(\bm{x}^i)}{N} &= \frac{1}{N} \sum_{i=1}^N\min_{\bm{x} \in \cap_{j\in [N]}X_j} \frac{\sum_{j=1}^N F_j(\bm{x})}{N}\\
        &= \min_{\bm{x} \in \cap_{j\in [N]}X_j} \frac{\sum_{j=1}^N F_j(\bm{x})}{N}\\
        &= \min_{\bm{x} \in \cap_{j\in [N]}X_j} \frac{\sum_{j=1}^N f(\bm{x})}{N} \\
        &= \min_{\bm{x} \in \cap_{j\in [N]}X_j} f(\bm{x}) = \text{\eqref{problem:consP}}
    \end{align*}

    Therefore, we have $\text{\eqref{problem:consP}} = \text{\eqref{problem:consD}}$.
\endproof

\end{APPENDICES}

\end{document}